\documentclass[11pt]{amsart}

\usepackage{amsmath, amsthm, amssymb, amsfonts, enumerate}
\usepackage[colorlinks=true,linkcolor=blue,urlcolor=blue]{hyperref}
\usepackage{dsfont}
\usepackage{color}
\usepackage{geometry}
\usepackage{todonotes}

\usepackage{graphicx}
\usepackage{caption}	

\usepackage{float}

\usepackage{mathtools}

\mathtoolsset{showonlyrefs}

\def \cC{\mathcal{C}}

\def \cB{\mathcal{B}}

\def \cS{\mathcal{S}}

\def \cF{\mathcal{F}}
\def \cG{\mathcal{G}}
\def \cI{\mathcal{I}}
\def \cJ{\mathcal{J}}
\def \cL{\mathcal{L}}

\def \cO{\mathcal{O}}

\def \cR{\mathcal R}

\def \P{\mathsf P}

\def \PP{\widetilde{\mathsf{P}}}
\def \E{\mathsf E}
\def \EE{\widetilde{\mathsf{E}}}
\def \N{\mathbb{N}}
\def \R{\mathbb{R}}

\def \ud{\mathrm{d}}
\def \e{\mathrm{e}}
\def \1{\mathds{1}}

\newcommand{\eps}{\varepsilon}

\newtheorem{theorem}{Theorem}[section]
\newtheorem{lemma}[theorem]{Lemma}

\newtheorem{proposition}[theorem]{Proposition}

\newtheorem{definition}[theorem]{Definition}
\newtheorem{remark}[theorem]{Remark}
\newtheorem{assumption}[theorem]{Assumption}

\makeatletter
\@namedef{subjclassname@2020}{\textup{2020} Mathematics Subject Classification}
\makeatother

\title[Continuously differentiable optimal boundaries]{A probabilistic approach to continuous differentiability of optimal stopping boundaries}
\author[T.\ De Angelis]{Tiziano De Angelis}
\author[D.\ Lamberton]{Damien Lamberton}
\address{T.\ De Angelis: School of Management and Economics, Dept.\ ESOMAS, University of Torino, Corso Unione Sovietica, 218 Bis, 10134, Torino, Italy; Collegio Carlo Alberto, Piazza Arbarello 8, 10122, Torino, Italy.}
\email{\href{mailto:tiziano.deangelis@unito.it}{tiziano.deangelis@unito.it}}
\address{D.\ Lamberton: LAMA, Univ.\ Gustave Eiffel, Univ.\ Paris Est Creteil, CNRS, Projet Mathrisk INRIA, F-77454 Marne-la-Vall\'ee, France}
\email{\href{mailto:damien.lamberton@univ-eiffel.fr}{damien.lamberton@univ-eiffel.fr}}
\date{\today}

\numberwithin{equation}{section}

\begin{document}

\begin{abstract}
We obtain the first probabilistic proof of continuous differentiability of time-dependent optimal boundaries in optimal stopping problems. The underlying stochastic dynamics is a one-dimensional, time-inhomogeneous diffusion. The gain function is also time-inhomogeneous and not necessarily smooth. Moreover, we include time-and-state-dependent discount rate and the time-horizon can be either finite or infinite. Our arguments of  proof are of a local nature that allows us to obtain the result under more general conditions than those used in the PDE literature. As a byproduct of our main result we also obtain the first probabilistic proof of the link between the value function of an optimal stopping problem and the solution of the Stefan problem.
\end{abstract}

\maketitle

\section{Introduction}

In this paper we obtain the first probabilistic proof of continuous differentiability of free boundaries in time-inhomogenous optimal stopping problems (OSPs). More precisely, when the underlying process is a one-dimensional diffusion $\mathbb{X}:=(X_t)_{t\ge 0}$ and the free boundary is a function of time $t\mapsto b(t)$, we show under mild conditions that $\dot b=\frac{\ud b}{\ud t}$ is well-defined and continuous (except perhaps at the maturity of the OSP). Time dependent boundaries may arise in our setting because of finite-time horizon and/or time-dependence in both the gain function and the discount rate and/or time-inhomogeneous dynamics for $\mathbb X$. Thanks to the local nature of our arguments of proof, we are able to obtain continuous differentiability of $t\mapsto b(t)$ in broad generality, going beyond analogous results from PDE theory. The PDE results typically require smooth coefficients for the second order differential operator associated to the underlying diffusion $\mathbb X$, along with smoothness of the stopping payoff and/or other {\em ad-hoc} structural conditions on those ingredients (see, e.g., Friedman \cite{friedman1975parabolic} or Kinderlehrer and Stampacchia \cite[Chapter VIII]{kinderlehrer2000introduction}).

Our methodology is new in the literature and it is based on a combination of change of scale for the diffusion (due to Lamperti \cite{lamperti1964simple}), Girsanov theorem and, crucially, an application of Pitman's theorem for the law of the three-dimensional Bessel process (see, e.g., Revuz and Yor \cite[Thm.\ III.3.5]{revuz2013continuous}). The key technical step in our analysis is a delicate Taylor-like expansion for the time derivative $\dot v$ of the value function $v$ of the OSP, near the optimal boundary $b$. Such expansion yields existence of the mixed derivative $\dot v_x$ at points of the boundary (i.e., existence of $\dot v_x(t,b(t))$), along with a probabilistic representation thereof. Then we are able to express the time-derivative of $b(t)$ as $\dot b(t)=-\eta(t)\dot u_x(t,b(t))$, where $\eta$ is explicit and continuous, $u=v-g$ and $g$ is the stopping payoff in the OSP. Finally, using the probabilistic representation of $\dot v_x(t,b(t))$ we are able to prove continuity of the mapping $t\mapsto \dot b(t)$.

We emphasise that a study of properties of the mixed derivative $\dot v_x$ near the stopping boundary is a notoriously difficult task. In the PDE literature $\dot v_x$ is studied using Gevrey's lemma \cite{gevrey1918nature}, which offers no probabilistic analogue, whereas in the probabilistic literature no previous results on $\dot v_x(t,b(t))$ are known.   

As an interesting byproduct of our main result (Theorem \ref{thm:main}) we also obtain the first probabilistic proof of the link between the value function of an OSP and the solution of the Stefan problem (Theorem \ref{thm:stefan}). More precisely, we show that the pair $(\dot v, b)$ coincides with a solution pair of the one-sided Stefan problem. Our setup is sufficiently general to accommodate, e.g., the case of the classical American put and call option, for which precise links to Stefan problems appear to be missing from the literature. Notably, McConnell \cite{mcconnell1991two} studied Stefan problems associated to the Laplace operator with low regularity on the data at the free boundary, relating to the latent heat in the ice-melting problem. Instead, we consider more general parabolic operators and low regularity in the terminal condition of the (backward version of the) Stefan problem.

\subsection{A brief review of optimal stopping and free boundary problems}

The development of optimal stopping theory started in the 50's with work by Wald and Wolfowitz \cite{wald1950bayes} and by Snell \cite{snell1952applications}, and it was mainly motivated by statistical questions around sequential testing and quickest detection (see, Shiryaev's recent book \cite{shiryaev2019stochastic} for an overview of the field and historical remarks). Many authors contributed to a systematic development of optimal stopping theory both in Markovian and non-Markovian framework. Foundational results can be found in the monographs by Shiryaev \cite{shiryaev2007optimal} and El Karoui \cite{elk1981aspects} and a more modern account of the theory is contained in the book by Peskir and Shiryaev \cite{peskir2006optimal} and references therein. The probabilistic study of free boundary problems (FBPs) motivated by optimal stopping theory has a long tradition that can be traced back to early contributions by, e.g., McKean \cite{mckean1965free}, Kotlow \cite{kotlow1973free} and Van Moerbeke \cite{van1976optimal}. An interesting historical account of the initial connections between optimal stopping and analytical methods for FBPs can be found in the review by Van Moerbeke \cite{van1974optimal}. 

The popularity of optimal stopping rapidly increased in the 90's when applications in mathematical finance became one of the main drivers of the theory. A central application, which stimulated an enormous amount of research at the interface of probability and mathematical analysis, is the celebrated American option pricing problem (detailed expositions and historical remarks can be found for example in the books by Karatzas and Shreve \cite{karatzas1998methods} and Lamberton and Lapeyre \cite{lamberton2011introduction}). Nowadays, optimal stopping is a very lively branch of stochastic control theory and of applied probability, which expands in several directions (including, e.g., stochastic games, mean-field games, time-inconsistent problems) and finds a broad range of applications.

The regularity of the free boundary $b$ that splits the state-space into the so-called {\em continuation} and {\em stopping} sets is a central theoretical question in the context of optimal stopping. In the PDE literature this issue has been addressed extensively, often via links to the Stefan problem, and numerous results concerning infinite differentiability of the free boundary have been known for a long time (see, e.g., Cannon and Hill \cite{CH67}, Friedman \cite{friedman1975parabolic} and Schaeffer \cite{schaeffer1976new}; earlier work from the 60's is also accounted for in L.-S.\ Jiang \cite{jiang1982free}). However, those results find limited use in modern optimal stopping because of assumptions that are often too restrictive for applications. Notable exceptions are the work by Bayraktar and Xing \cite{bayraktar2009analysis}, who obtain infinite differentiability of the American put boundary in a particular jump diffusion model, and Laurence and Salsa \cite{laurence2009regularity} who obtain infinite differentiability of American options boundaries in multi-dimensional geometric Brownian motion models. Both papers rely upon PDE methods.

Somewhat surprisingly, and in part as a testament to the difficulty of the problem, results in the probabilistic literature are still extremely scarce. In the special case when the underlying process $\mathbb X$ is a Brownian motion, it is possible to use properties of the Gaussian transition density to deduce continuity and higher regularity of $b$ (see, e.g., Chen and Chadam \cite{chen2007mathematical}) following similar ideas as those in the PDE literature (cf., the monographs by Cannon \cite{cannon1984one} and Friedman \cite{friedman1975parabolic}). Those methods do not easily extend beyond the Gaussian framework. In more general situations, probabilistic methods have only been used to prove continuity and, in some cases, Lipschitz continuity of the boundary (cf.\ De Angelis \cite{de2015note}, De Angelis and Stabile \cite{de2019lipschitz} and Peskir \cite{peskir2019continuity} and references therein). Our work advances this strand of the literature and it bridges the gap between the probabilistic and PDE realms. As an auxiliary result, in Section \ref{sec:lip} we prove that optimal stopping boundaries are locally Lipschitz in finite-horizon problems for time-homogeneous diffusions and with time-homogeneous discounted payoff.

Part of our contribution concerns the link between optimal stopping and Stefan problems. Thus we now elaborate a little on this aspect of the paper in the broader context of the existing literature. The link between Stefan problems and obstacle problems was initially highlighted by Duvaut \cite{duvaut1973resolution} (cf.\ also Schatz \cite{schatz1969free}). A detailed proof entirely based on PDE methods can be found in \cite[Chapter VIII]{kinderlehrer2000introduction}. The line of argument typically goes as follows: (i) it is postulated that a solution pair $(w,\varphi)$ of the Stefan problem exists, where $\varphi(t)$ is the free boundary and $w(t,x)$ the solution of the Cauchy-Dirichlet problem associated to $\varphi$; (ii) a new function $\hat w$ is defined as the integral in time of $w$ and it is argued that $\hat w$ should be solution of an obstacle problem; (iii) the obstacle problem becomes the main object of study and it is shown by PDE methods that it admits a unique solution $\hat w$ with suitable regularity (in a Sobolev class); (iv) geometric properties of the contact set associated to $\hat w$ are obtained and, in particular, it is proven that there exists a continuously differentiable free boundary $\varphi$; (v) finally, taking the time derivative of $\hat w$ it is possible to show the existence of the solution $w$ to the Stefan problem, as initially postulated. A very important aspect of this procedure is that $\frac{\ud}{\ud t}\varphi(t)$ is proportional to the spatial derivative of $w$ (equivalently to the mixed derivative of $\hat w$ with respect to time and space). 

Since the solution $\hat w$ of the obstacle problem is naturally associated to the value function $v$ of an optimal stopping problem, it becomes apparent that a characterisation of the time derivative of the optimal stopping boundary requires knowledge of the mixed derivative $\dot v_x$. This has been precisely the missing piece in the puzzle so far in the probabilistic literature. The most recent results (cf.,\ De Angelis and Peskir \cite{de2020global}) concern continuous differentiability in time and space of the value function for quite general OSPs, but this involves no need to derive
precise estimates for $\dot v_x$. Continuity of $\dot v$ alone, on the
other hand, is insufficient to obtain the link with the Stefan
problem rigorously. In this paper we fill this gap and bring the
probabilistic approach on par with the analytical one.

\subsection{Structure of the paper} 
In Section \ref{sec:setting} we formulate the problem and we state our main result (Theorem \ref{thm:main}) concerning continuous differentiability of the optimal boundary under the assumption of its local Lipschitz continuity. Section \ref{sec:lip} illustrates simple sufficient conditions that imply local Lipschitz continuity of the optimal boundary. Finally, Section \ref{sec:outline} provides an outline of our strategy for the proof of Theorem \ref{thm:main}. Section \ref{sec:3} introduces the main ingredients for our line of argument (Lamperti transform, Girsanov theorem and Pitman's theorem) along with numerous preliminary estimates on the three-dimensional Bessel process and on some related stopping times. In Section \ref{sec:expansion} we perform our main technical steps and obtain a Taylor-like expansion of $\dot v$ near the optimal boundary $b$. That establishes existence of the mixed derivative $\dot v_x(t,b(t))$ (Theorem \ref{thm:expansion}). In Section \ref{sec:proofmain} we formally prove Theorem \ref{thm:main} leveraging on the result of Theorem \ref{thm:expansion}. Finally, in Section \ref{sec:OS-Stefan} we establish the link between optimal stopping and Stefan problem.

\section{Setting and main result}\label{sec:setting}

Let $(\Omega,\cF,\P)$ be a complete probability space equipped with a standard 1-dimensional Brownian motion $(B_t)_{t\ge 0}$. Our problem is set on a time-horizon $[0,T]$, where $T>0$ can be either finite or infinite. Let $\mathbb{X}:=(X_t)_{t\ge 0}$ be a 1-dimensional, diffusion on a possibly unbounded interval $\cI=(a,b)$. The dynamics of $\mathbb X$ is specified by the stochastic differential equation (SDE)
\begin{align*}
X_{s}=x+\int_t^s \mu(u,X_u)\ud u+\int_t^s\sigma(X_u)\ud B_u,\quad s\in[t,T],
\end{align*}
for any $t\in[0,T]$ and $x\in\cI$. When $T=\infty$ we assume $\mathbb X$ to satisfy the SDE above on the interval $[t,\infty)$. Assumptions on the coefficients $\mu:[0,T]\times\cI\to \R$ and $\sigma:\cI\to (0,\infty)$ will be provided later, and we do not need to specify the behaviour of $\mathbb X$ at the endpoints of the interval $\cI$, because we perform a local study in $\cI$ (cf.\ Assumption \ref{ass:main}). However, we always assume that $\mathbb X$ is the unique strong solution of the equation above. We will sometimes use the notation $(X^{t,x}_s)_{s\in[t,T]}$ to keep track of the initial condition of $\mathbb X$.

Given continuous functions $r,g:[0,T]\times\R\to \R$, we consider OSPs of the form 
\begin{align}\label{eq:v}
v(t,x)=\sup_{0\le \tau\le T-t}\E_{t,x} \big[\e^{-\int_0^\tau r(t+s,X_{t+s})\ud s}g(t+\tau, X_{t+\tau})\big],\quad (t,x)\in[0,T]\times\cI,
\end{align}
where $\E_{t,x}[\cdot]=\E[\cdot|X_t=x]$ and the supremum is taken over stopping times for the filtration of the process $(X_{t+s})_{s\in[0,T-t]}$. We are going to write $D_{t,\tau}\coloneqq \exp\big(-\int_0^\tau r(t+s,X_{t+s})\ud s\big)$ in order to shorten some of our expressions, so that 
\begin{align*}
v(t,x)=\sup_{0\le \tau\le T-t}\E_{t,x} \big[D_{t,\tau}g(t+\tau, X_{t+\tau})\big],\quad (t,x)\in[0,T]\times\cI.
\end{align*}
The so-called continuation set is defined as 
\[
\cC\coloneqq\{(t,x)\in[0,T)\times\cI: v(t,x)>g(t,x)\}
\] 
and its complement is the stopping set $\cS\coloneqq \big([0,T]\times\cI\big)\setminus\cC$. The existence of an optimal stopping time in the form of an exit time of the process $(t+s,X_{t+s})$ from the set $\cC$ is guaranteed by standard theory (see, e.g., \cite[Appendix D]{karatzas1998methods}) if in addition to continuity of the process $s\mapsto D_{t,s}g(t+s,X_{t+s})$ we also assume
\[
\E_{t,x}\Big[\sup_{0\le s\le T-t}D_{t,s}\big|g(t+s,X_s)\big|\Big]<\infty.
\]

Since our focus is on the smoothness of the optimal (free) boundary we naturally assume that the problem above is well-posed and that an optimal boundary exists as a function of time. More precisely, we assume throughout that there exists a continuous function 
$b:[0,T]\to \cI$ such that the stopping time
\[
\tau_*^{t,x}\coloneqq\inf\{s\in[0,T-t):X^{t,x}_{t+s}\le b(t+s)\}\wedge (T-t),
\]
is optimal in problem \eqref{eq:v}. In this context we denote 
$\cC=\big\{(t,x)\in[0,T)\times\cI: x>b(t)\big\}$.
The function $v$ is assumed to be continuous so that $\cC$ is open and $\cS$ is closed (relatively to $[0,T]\!\times\!\cI$).

\begin{remark}\label{rem:loc}
Existence of an optimal boundary and, more broadly, the geometry of the continuation set is normally determined on a case-by-case basis, with few existing general results. In order for there to be a unique boundary $b(t)$ as above it is sufficient to show that $x\mapsto (v-g)(t,x)$ is non-decreasing. The latter task is often nontrivial but {\em ad-hoc} methods have been developed in the literature for specific problems. For example, in the American put problem the required monotonicity is proven by showing that $x\mapsto v(t,x)$ is convex and strictly positive with $v(t,0)=K$ (here $K>0$ is the strike price of the option; see, e.g., \cite{jacka1991optimal}). Many more examples are available in the literature (see, e.g., \cite{peskir2006optimal} for an overview) but a full review of existing results is outside the scope of the present paper, not least because
the nature of our assumptions and of our arguments of proof is completely {\em local}. Therefore we may as well assume that there are multiple separated optimal boundaries, e.g., $t\mapsto b_1(t)$ and $t\mapsto b_2(t)$ with $b_1(t) < b_2(t)$ for $t\in[0,T)$ and $\cC=\big\{(t,x)\in[0,T)\times\cI: b_1(t)<x<b_2(t)\big\}$. The choice of a single boundary is purely motivated by clarity of exposition. 
\end{remark}

\begin{remark}
We are only concerned with portions of the optimal boundary that lie in the open interval $\cI$, i.e., $b(t)\in\cI$. Going beyond this setup would require to specify the boundary behaviour of $\mathbb X$ and it would lead to a much more complicated analysis. Moreover, we assume that $\sigma(x)>0$ for $x\in\cI$ and we will use this fact for a change of probability measure and for an application of a Lamperti change of scale (cf.\ \cite{lamperti1964simple}). That is also why we prefer to work with a time-homogeneous diffusion coefficient, in order to keep the exposition as simple as possible. In specific cases, our methods can be applied to time-dependent diffusion coefficients $\sigma(t,x)$ but at the cost of an additional term in the drift of the dynamics arising from the Lamperti change of scale (cf.\ \eqref{eq:lamperti}). In particular, the function $\gamma(t,y)$ in that dynamics would read
\[
\gamma(t,y)=\frac{\mu\big(t,f^{-1}(t,y)\big)}{\sigma \big(t,f^{-1}(t,y)\big)}-\tfrac12\partial_x\sigma\big(t,f^{-1}(t,y)\big)-\int_c^{f^{-1}(t,y)}\frac{\partial_t \sigma(t,z)}{\sigma^2(t,z)}\ud z,
\]
where $f(t,x)=\int^x_c 1/\sigma(t,z)\ud z$ for an arbitrary constant $c\in\cI$ and $f^{-1}(t,y)$ denotes the inverse function of $f(t,\cdot)$ for fixed $t\in[0,T]$.
\end{remark}
 
Given two random times $\sigma\le \tau$ we use notations $[\![\sigma,\tau]\!]$ and $(\!(\sigma,\tau)\!)$ for the closed and open stochastic interval, respectively. Given an open set $A\subset[0,T)\times\cI$, we use $\overline A$ to indicate its closure relatively to $[0,T]\times\cI$. We say that $f\in C^{j,k}(A)$, $j,k\in\N$, if $f:A\to \R$ is continuously differentiable $j$-times in $t$ and $k$-times in $x$ in the set $A$; moreover, if $f$ admits continuous extension of all of its derivatives to the closure of $A$ we denote it $f\in C^{j,k}(\overline A)$. We use the simplified notations $f\in C^j(A)$ and $f\in C^j(\overline A)$ to indicate $f\in C^{j,j}(A)$ and $f\in C^{j,j}(\overline A)$, respectively. We use $C^\infty_b(A)$ to indicate continuous bounded functions with infinitely many continuous bounded derivatives in all variables in $A$. For a function $f(t,x)$ we use $\dot f(t,x)$ to indicate its time derivative and $f_x(t,x)$ to indicate its spatial derivative. Moreover, given two metric spaces $M$ and $N$ we use the notation $f\in C(M;N)$ for a continuous function $f:M\to N$. Finally, the infinitesimal generator of the process $\mathbb X$ is denoted by $\cL$ and it is defined by its action on smooth functions $\varphi\in C^\infty_b([0,T]\times\cI)$ as 
\[
(\cL \varphi)(t,x):=\frac{\sigma^2(x)}{2}\varphi_{xx}(t,x)+\mu(t,x)\varphi_x(t,x).
\]
 
We make the following standing assumptions:
\begin{assumption}\label{ass:main}
For each $t_0\in[0,T)$ there exists an open rectangle $\cR=\cR(t_0)$ with closure $\overline\cR\subset [0,T)\times \cI$ and $(t_0,b(t_0))\in\cR$ such that:
\begin{itemize}
\item[(i)] The value function $v$ is continuously differentiable in time and space in $\overline\cR$ (i.e., also across the boundary $b$); 
\item[(ii)] The gain function $g$ belongs to $C^{1,2}(\overline \cR)$ with $\dot g\in C^{1,2}(\overline\cR)$ and 
\begin{align}\label{eq:h<0}
h(t,x)\coloneqq \dot g(t,x)+(\cL g)(t,x)-r(t,x)g(t,x)<0,\quad\text{for $(t,x)\in\overline\cR$};
\end{align}
\item[(iii)] In the set $\overline \cR$ the functions $\mu$ and $r$ are continuously differentiable with $\alpha$-H\"older continuous derivatives for some $\alpha\in(0,1)$; 
\item[(iv)] In the set $\overline\cR$ the function $\sigma$ is continuously differentiable\footnote{We slightly abuse notation and consider $\sigma$ as function of $(t,x)$ for the ease of exposition.} with $\sigma_x$ that is Lipschitz continuous and $\sigma(x)\ge c_\cR>0$ for $(t,x)\in\overline\cR$.
\end{itemize}
\end{assumption}

The above assumptions are in fact verified in a very large class of concrete problems. Sufficient conditions for continuous differentiability of the value function (i) are known in both the probabilistic literature (see \cite{de2020global}) and the PDE literature (e.g., \cite{bensoussan2011applications} or \cite{friedmanSDE} for classical references and \cite{pascucci2011pde} for a more modern account). Sufficient conditions that make results from the existing literature applicable are $g$, $\mu$, $\sigma$, $r$ bounded, continuously differentiable with bounded derivatives, $\sigma(x)>0$ and, for example, $b$ non-decreasing (cf.\ \eqref{eq:bmon} below). Those sufficient conditions are far from being necessary and the reader may refer to \cite{pascucci2011pde} for more precise PDE results. Also \cite{de2020global} offers more detailed probabilistic conditions and it proves, for example, that the value of the American put is continuously differentiable in time and space (although $g$ is not smooth and $\mu$ and $\sigma$ are unbounded). Local regularity of the gain function (ii) is not restrictive and it holds in all the main examples in the literature (even when $g$ is not smooth in the whole space as in the case of the American call/put options). We notice that the condition \eqref{eq:h<0} is also not restrictive because it is a standard result in optimal stopping that $h(t,x)>0\implies(t,x)\in\cC$ (i.e., it is never optimal to stop when the discounted gain process behaves as a submartingale). Likewise, local smoothness of the discount rate and of the coefficients in the SDE (iii)--(iv) is a common feature in optimal stopping. 

We now state the main result of the paper. Its proof will be obtained in a series of intermediate steps presented in the next sections. The formal proof is stated at the end of Section \ref{sec:proofmain}.
\begin{theorem}\label{thm:main}
Let $T_1\in(0,T)$. Under Assumption \ref{ass:main}, if $t\mapsto b(t)$ is Lipschitz continuous in $[0,T_1]$, then it is continuously differentiable on $[0,T_1)$. 
\end{theorem}
On a technical level, the Lipschitz continuity of the optimal boundary is needed to perform a change of probability measure and Girsanov transformation that simplify the analysis (cf.\ Section \ref{sec:Girs}). Sufficient conditions for the Lipschitz property of the optimal boundary can be found for example in \cite[Sec.\ 4.1]{de2019lipschitz} under some smoothness requirements on $g$ (but with no assumption on the continuity of the derivatives of $v$). In Section \ref{sec:lip} we will show that in our setup the boundary is locally Lipschitz as soon as the functions $g$ and $r$ are time-homogeneous, by leveraging on the resulting monotonicity of the stopping boundary. 

\begin{remark}
Notice that if $t\mapsto b(t)$ is locally Lipschitz on $[0,T)$ the theorem above tells us that $b\in C^1([0,T))$. We cannot expect $b\in C^1([0,T])$ in our setting because it is well-known that, for example, the optimal boundary of the American call/put option problem is not differentiable at the maturity $T$ (cf., \cite{barles1995critical}, \cite{evans2002american}, \cite{lamberton2003critical}).
\end{remark}

As part of the problem setting, we recall a well-known fact from optimal stopping theory (cf.\ \cite[Thm.\ 2.7.7]{karatzas1998methods} and \cite[Prop.\ 2.6]{jacka1991optimal} for the American put problem and \cite[Ch.\ III.7]{peskir2006optimal} for a general overview of the method). We let $\cR$ be an open rectangle as in Assumption \ref{ass:main}. With no loss of generality we assume that $\cR=(0,T_1)\times (x_1,x_2)$ and $x_1<b(t)<x_2$ for all $t\in[0,T_1]$. Setting $u\coloneqq v-g$ it holds
\begin{align}\label{eq:h}
\begin{aligned}
\dot u(t,x)+(\cL u)(t,x)-r(t,x)u(t,x)= -h(t,x),\quad (t,x)\in\cR\cap\cC.
\end{aligned}
\end{align}
Then, under Assumption \ref{ass:main}, we have $\dot u\in C^{1,2}(\cC\cap\cR)$ (cf.\ \cite[Thm.\ 3.5.11]{friedman2008partial}) and we obtain a PDE for $\dot u$ in the form
\begin{align}
\begin{aligned}
\ddot u(t,x)+(\cL \dot u)(t,x)-r(t,x)\dot u(t,x)= -H(t,x),\quad (t,x)\in\cR\cap\cC,
\end{aligned}
\end{align}
where $H(t,x)\coloneqq\dot h(t,x)+\dot \mu(t,x) u_x(t,x)-\dot r(t,x)u(t,x)$ and the boundary condition $\dot u(t,b(t))=0$ holds for all $t\in[0,T_1]$. Set $\tau^{t,x}_{x_2}=\inf\{s\in[0,\infty):X^{t,x}_{t+s}\ge x_2\}$ and let $\sigma^{t,x}_*=\tau^{t,x}_*\wedge (T_1-t)$. Notice that $\sigma_*^{t,x}\wedge \tau_{x_2}^{t,x}$ is the exit time of $(t+s,X^{t,x}_{t+s})$ from $\cR\cap\cC$. Then the process
\[
\lambda\mapsto D_{t,\sigma_*\wedge\tau_{x_2}\wedge \lambda}\dot u\big(t\!+\!\sigma_*\wedge\tau_{x_2}\wedge \lambda,X_{t+\sigma_*\wedge\tau_{x_2}\wedge \lambda}\big)\!+\!\int_0^{\sigma_*\wedge\tau_{x_2}\wedge \lambda}\!\!D_{t,s}H\big(t\!+\!s,X_{t+s}\big)\ud s,
\]
is a bounded local martingale under $\P_{t,x}$. Hence, taking expectation
\begin{align}\label{eq:udot}
\begin{aligned}
\dot u(t,x)&=\E_{t,x}\Big[D_{t,\sigma_*\wedge\tau_{x_2}}\dot u\big(t+\sigma_*\wedge\tau_{x_2},X_{t+\sigma_*\wedge\tau_{x_2}}\big)+\int_0^{\sigma_*\wedge\tau_{x_2}}D_{t,s}H\big(t+s,X_{t+s}\big)\ud s\Big]\\
&=\E_{t,x}\Big[\1_{\{\sigma_*=T_1-t,\ \sigma_*\le \tau_{x_2}\}}D_{t,T_1-t}\dot u\big(T_1,X_{T_1}\big)+\1_{\{\tau_{x_2}<\sigma_*\}}D_{t,\tau_{x_2}}\dot u\big(t+\tau_{x_2},x_2\big)\\
&\qquad\qquad+\int_0^{\sigma_*\wedge\tau_{x_2}}D_{t,s}H\big(t+s,X_{t+s}\big)\ud s\Big],
\end{aligned}
\end{align}
where we used $\dot u(t+\sigma_*,X_{t+\sigma_*})\1_{\{\sigma_*<T_1-t\}}=\dot u(t+\tau_*,b(t+\tau_*))\1_{\{\tau_*<T_1-t\}}=0$ for the second equality. The probabilistic representation in \eqref{eq:udot} is the starting point of our analysis. From now on we tacitly assume Assumption \ref{ass:main}.

\subsection{Simple sufficient conditions for a Lipschitz boundary}\label{sec:lip}
In this section we consider the special case when $T<\infty$, $g(t,x)=g(x)$, $\mu(t,x)=\mu(x)$ and $r(t,x)=r(x)\!\ge\! 0$. For any Borel set $A\subset\cI$ and any $s<t$ we have $\P_{s,x}(X_t\in A)=\P_{0,x}(X_{t-s}\in A)$. It is convenient to set $\P_x=\P_{0,x}$ and $\E_x=\E_{0,x}$. Then, for $s<t$ we deduce
\[
v(t,x)=\sup_{0\le \tau\le T-t}\E_x\Big[\e^{-\int_0^\tau r(X_u)\ud u}g\big(X_\tau\big)\Big]\le \sup_{0\le \tau\le T-s}\E_x\Big[\e^{-\int_0^\tau r(X_u)\ud u}g\big(X_\tau\big)\Big]=v(s,x), 
\]
hence $t\mapsto v(t,x)$ is non-increasing and $\dot v\le 0$ in $\cC$. That implies 
\begin{align}\label{eq:bmon}
\text{$t\mapsto b(t)$ is non-decreasing on $[0,T]$.}
\end{align}
Such monotonicity of the boundary, jointly with $\sigma(b(t))>0$ for $t\in[0,T)$, implies that any boundary point $(t,b(t))$ is regular in the sense of diffusions for the interior of the stopping set $\cS$ (cf.\ \cite[Lemma 4 and Example 7]{de2020global}), provided that also $x\mapsto X^x_t(\omega)$ is continuous. As a result, the value function $v$ is continuously differentiable across the boundary $b$ as soon as the functions $g$ and $r$ and the stochastic flow $x\mapsto X^x_t(\omega)$ are differentiable, and mild integrability conditions hold (cf.\ \cite[Thm.\ 10 and Thm.\ 15]{de2020global}). In conclusion, Assumption \ref{ass:main} is often easily verified in specific problems covered by the set-up of this section.

Notice that $H(t,x)=0$ in \eqref{eq:udot}. Then, by time-homogeneity of $\mathbb{X}$ and $g$ we can rewrite \eqref{eq:udot} as 
\begin{align}\label{eq:udotto}
\begin{aligned}
\dot u(t,x)&=\E_{x}\Big[\1_{\{\sigma_*=T_1-t,\ \sigma_*\le \tau_{x_2}\}}\e^{-\int_0^{T_1-t} r(X_s)\ud s}\dot u\big(T_1,X_{T_1-t}\big)\\
&\quad\qquad+\!\1_{\{\tau_{x_2}<\sigma_*\}}\e^{-\int_0^{\tau_{x_2}} r(X_s)\ud s}\dot u\big(t\!+\!\tau_{x_2},x_2\big)\Big],
\end{aligned}
\end{align}
with $\sigma_*=\tau_*\wedge(T_1-t)$ and
$\tau_*=\inf\{s\in[0,T-t):X_{s}\le b(t+s)\}\wedge (T-t)$. The next proposition provides a rate at which $\dot v$ vanishes as $(t,x)$ approaches the boundary of the set $\cC$. 
\begin{proposition}\label{prop:dotv}
There is $c=c(T_1,x_2)>0$ such that 
\begin{align}\label{eq:budot}
\big|\dot u(t,x)\big|\le c\big(x-b(t)\big)\Big(1+\frac{1}{(T_1-t)^{1/2}}\Big),
\end{align}
for all $x>b(t)$ and $t\in[0,T_1)$.
\end{proposition}
\begin{proof}
Thanks to \eqref{eq:bmon}, setting $\nu_b=\inf\{s\in[0,\infty):X_s\le b\}$ for $b\in\cI$, we have 
\begin{align*}
\big\{\sigma_*=T_1-t,\ \sigma_*\le \tau_{x_2}\big\}&=\big\{b(t+s)<X_s<x_2,\,\forall s\in[0,T_1-t)\big\}\\
&\subseteq\big\{b(t)<X_s<x_2,\,\forall s\in[0,T_1-t)\big\}=\big\{\tau_{x_2}\wedge \nu_{b(t)}\ge T_1-t\big\}.
\end{align*}
Since $\sigma_*\le \nu_{b(t)}\wedge(T_1-t)$, then 
$\{\tau_{x_2}<\sigma_*\}\subseteq \{\tau_{x_2}<\nu_{b(t)}\}$.
From those two inclusions and \eqref{eq:udotto}, it follows
\begin{align}\label{eq:boundudot0}
\begin{aligned}
\big|\dot u(t,x)\big|&\le \sup_{b(t)\le z\le x_2}\big|\dot u(T_1,z)\big|\P_x\big(\tau_{x_2}\wedge \nu_{b(t)}\ge T_1-t\big)+\sup_{t\le s\le T_1}\big|\dot u(s,x_2)\big|\P_x\big(\tau_{x_2}<\nu_{b(t)}\big)\\
&\le c\Big(\P_x\big(\tau_{x_2}\wedge \nu_{b(t)}\ge T_1-t\big)+\P_x\big(\tau_{x_2}<\nu_{b(t)}\big)\Big)\\
&\le c\Big(\P_x\big(\tau_{x_2}> \nu_{b(t)}\ge T_1-t\big)+2\P_x\big(\tau_{x_2}<\nu_{b(t)}\big)\Big),
\end{aligned}
\end{align}
where $c=c(T_1,x_2)>0$ is a constant. We estimate the two terms above separately.

Let $S:\cI\to (0,\infty)$ be the scale function associated to the process $\mathbb X$. Then, from the standard theory for one-dimensional, time-homogeneous diffusions we know that 
\begin{align}\label{eq:bP0}
\P_x\big(\tau_{x_2}<\nu_{b(t)}\big)=\frac{S(x)-S\big(b(t)\big)}{S(x_2)-S\big(b(t)\big)}\le \frac{S(x)-S\big(b(t)\big)}{S(x_2)-S\big(b(T_1)\big)}\le c\big(x-b(t)\big),
\end{align}
where $S(b(T_1))\ge S(b(t))$ because $b(T_1)\ge b(t)$ and the final inequality holds with $c=c(T_1,x_2)>0$ because under our assumptions $S'$ is continuous and $x_2>b(T_1)$. 

For the other term on the right-hand side of \eqref{eq:boundudot0} we simplify the notation by setting $\eta=T_1-t$. 
On the event $\{\tau_{x_2}> \nu_{b(t)}\ge \eta\}$ we have $X_s=\bar X_s$, for all $s\in[0,\eta]$, $\P_x$-a.s., where $(\bar X_s)_{s\ge 0}$ is the solution of 
\begin{align*}
\bar X_{s}=x+\int_0^s\bar \mu(\bar X_\lambda)\ud \lambda+\int_0^s\bar \sigma(\bar X_\lambda)\ud B_\lambda,\quad s\in[0,\infty),
\end{align*}
with $\bar \mu$ and $\bar \sigma$ that coincide with $\mu$ and $\sigma$ on $[b(t),x_2]$, respectively, and are extended continuously to be constant outside of that interval. Then $\nu_{b(t)}=\bar \nu_{b(t)}=\inf\{s\in[0,\infty):\bar X_s\le b(t)\}$ on the event $\{\tau_{x_2}> \nu_{b(t)}\}$ and
\begin{align*}
\begin{aligned}
\P_x\big(\tau_{x_2}> \nu_{b(t)},\nu_{b(t)} \ge \eta\big)=\P_x\big(\tau_{x_2}> \nu_{b(t)},\bar\nu_{b(t)} \ge \eta\big)\le \P_x\big(\bar\nu_{b(t)} \ge \eta\big).
\end{aligned}
\end{align*}
In order to estimate $\P_x(\bar\nu_{b(t)} \ge \eta)$ it is convenient to introduce a time change. Notice that $\bar \sigma(x)\ge\inf_{x\in[b(t),x_2]}\sigma(x)\eqqcolon \sigma_0>0$ and $\bar \mu(x)\le \sup_{x\in[b(t),x_2]}|\mu(x)|\eqqcolon \mu_0$ for $x\in\cI$. Then 
\begin{align*}
\bar X_s&\le x+\mu_0 s+\int_0^s\bar \sigma(\bar X_u)\ud B_u\\
&=x+\mu_0 s+W_{\int_0^s\bar \sigma^2(\bar X_u)\ud u}\le  x+\frac{\mu_0}{\sigma^2_0}\chi_s+W_{\chi_s},
\end{align*}
where $\chi_s\coloneqq\int_0^s\bar \sigma^2(\bar X_u)\ud u$ and $W$ is a Brownian motion for the time-changed filtration $\cG_s=\cF_{\alpha_s}$ with $\alpha_s\coloneqq\inf\{u\ge 0:\chi_u>s\}$.
It follows that
\begin{align*}
\begin{aligned}
\P_x\big(\bar\nu_{b(t)} \ge \eta\big)&=\P_x\big(\bar\nu_{b(t)} > \eta\big)=\P_x\Big(\inf_{s\in[0,\eta]}\bar X_s>b(t)\Big)\\
&\le \P\Big(\inf_{s\in[0,\eta]}\big((\mu_0/\sigma^2_0)\chi_s+W_{\chi_s}\big)>b(t)-x\Big)\\
&= \P\Big(\inf_{r\in[0,\chi_\eta]}\big((\mu_0/\sigma^2_0)r+W_{r}\big)>b(t)-x\Big)\\
&\le \P\Big(\inf_{r\in[0,\sigma^2_0\eta]}\big((\mu_0/\sigma^2_0)r+W_{r}\big)>b(t)-x\Big)\\
&= \P\Big(\sup_{r\in[0,\sigma^2_0\eta]}\big(-(\mu_0/\sigma^2_0)r+\tilde W_{r}\big)<x-b(t)\Big),
\end{aligned}
\end{align*}
where in the second inequality we used $\chi_\eta\ge \sigma^2_0\eta$ and in the final expression $\tilde W_s=-W_s$ is again a Brownian motion. The last probability can be calculated explicitly using that 
\[
\Big\{\sup_{r\in[0,\sigma^2_0\eta]}\big(-(\mu_0/\sigma^2_0)r+\tilde W_{r}\big)<x-b(t)\Big\}=\big\{\zeta_{x-b(t)}>\sigma^2_0\eta\big\},
\]
with $\zeta_{z}\coloneqq \inf\{s\in[0,\infty):-(\mu_0/\sigma^2_0)s+\tilde W_{s}\ge z\}$. Since 
\[
\P(\zeta_z\in\ud s)=\frac{|z|}{\sqrt{2\pi s^3}}\exp\Big(-\frac{\big[z+(\mu_0/\sigma^2_0)s\big]^2}{2s}\Big)\ud s,
\]
(cf.\ \cite[Eq.\ (5.12), Sec.\ 3.5.C]{KS}) then
\begin{align*}
\begin{aligned}
\P_x\big(\bar\nu_{b(t)} \ge \eta\big)&\le \big(x-b(t)\big)\int_{\sigma^2_0\eta}^\infty \frac{1}{\sqrt{2\pi s^3}}\exp\Big(-\frac{\big[(x-b(t))+(\mu_0/\sigma^2_0)s\big]^2}{2s}\Big)\ud s\\
&\le \big(x-b(t)\big)\int_{\sigma^2_0\eta}^\infty \frac{1}{\sqrt{2\pi s^3}}\ud s\le \frac{x-b(t)}{\sigma_0(T_1-t)^{1/2}},
\end{aligned}
\end{align*}
where in the final expression we recall $\eta=T_1-t$.

Combining the latter bound with \eqref{eq:bP0} and \eqref{eq:boundudot0} and redefining the constant $c=c(T_1,x_2)>0$ we arrive at \eqref{eq:budot} as needed.
\end{proof}

Next we are going to show that in this setup the boundary $b$ is locally Lipschitz.
\begin{proposition}
The mapping $t\mapsto b(t)$ is Lipschitz on $[0,T_1]$ for any $T_1\in[0,T)$.
\end{proposition}
\begin{proof}
Since we want to use the formulae and notations developed above, we are actually going to show that $t\mapsto b(t)$ is Lipschitz on $[0,T_2]$ for any $T_2\in[0,T_1)$. However, $T_1\in[0,T)$ is arbitrary and therefore there is no loss of generality.

For $0\le s<t\le T_2$ we have $u(s,b(s))=u(t,b(t))=0$ and by continuous differentiability of $u$ we can write (recall $b(s)\le b(t)$)
\begin{align*}
\begin{aligned}
u(t,b(t))-u(s,b(s))&=u(t,b(t))-u(s,b(t))+u(s,b(t))-u(s,b(s))\\
&=\int_s^t\dot u(\lambda,b(t))\ud \lambda+\int_{b(s)}^{b(t)}u_x(s,z)\ud z\\
&=\int_s^t\dot u(\lambda,b(t))\ud \lambda+\int_{b(s)}^{b(t)}\int_{b(s)}^z u_{xx}(s,y)\ud y\,\ud z\\
&=\int_s^t\dot u(\lambda,b(t))\ud \lambda+\int_{b(s)}^{b(t)}\big(b(t)-y\big) u_{xx}(s,y)\ud y,
\end{aligned}
\end{align*}
where the third equality holds because $u_x(s,b(s))=0$ and the final one is by swapping the order of integration. Then we have
\begin{align}\label{eq:lipb1}
-\int_s^t\dot u(\lambda,b(t))\ud \lambda=\int_{b(s)}^{b(t)}\big(b(t)-y\big) u_{xx}(s,y)\ud y.
\end{align}
Recalling $\dot u\le 0$ and Proposition \ref{prop:dotv}
\begin{align}\label{eq:lipb2}
\begin{aligned}
-\int_s^t\dot u(\lambda,b(t))\ud \lambda&\le c\Big(1+\frac{1}{(T_1-t)^{1/2}}\Big)\int_s^t\big(b(t)-b(\lambda)\big)\ud \lambda\\
&\le c\Big(1+\frac{1}{(T_1-t)^{1/2}}\Big)\big(b(t)-b(s)\big)(t-s),
\end{aligned}
\end{align}
where we also used monotonicity of $b$.

We notice that \eqref{eq:h} in this setup reads $\dot u(t,x)+(\cL u)(t,x)-ru(t,x)=-h(x)$, with $h(x)=(\cL g)(x)-r g(x)$, for $(t,x)\in\cR\cap\cC$. Thanks to continuous differentiability of $u$ and the boundary conditions $u(s,b(s))=u_x(s,b(s))=\dot u(s,b(s))=0$, we obtain  
\[
\lim_{\cC\ni(r,y)\to (s,b(s))}u_{xx}(r,y)=-\frac{2}{\sigma^2(b(s))}h(b(s))>0,
\]
where the inequality holds because of \eqref{eq:h<0}.
That shows continuity of $u_{xx}$ in the set $\cB_{s,t}\coloneqq\{(r,y):r\in[s,t],\, y\ge b(r)\}$. Then, for $s$ and $t$ sufficiently close, by continuity of $r\mapsto b(r)$ and uniform continuity of $u_{xx}$ in the closed set $\overline{\cR\cap\cB_{s,t}}$ we have the lower bound
\begin{align}\label{eq:lipb3}
\begin{aligned}
\int_{b(s)}^{b(t)}\big(b(t)-y\big) u_{xx}(s,y)\ud y&\ge -\frac{h(b(s))}{\sigma^2(b(s))}\int_{b(s)}^{b(t)}\big(b(t)-y\big)\ud y\\
&\ge -\frac{h(b(s))}{2\sigma^2(b(s))} \big(b(t)-b(s)\big)^2.
\end{aligned}
\end{align}

Combining \eqref{eq:lipb1}, \eqref{eq:lipb2} and \eqref{eq:lipb3} yields
\begin{align*}
0\le \frac{\big(b(t)-b(s)\big)}{t-s}\le 2c\Big(1+\frac{1}{(T_1-T_2)^{1/2}}\Big)\sup_{s\in[0,T_2]}\frac{\sigma^2(b(s))}{|h(b(s))|} ,
\end{align*}
which concludes our proof.
\end{proof}

\subsection{An outline of the proof of Theorem \ref{thm:main}}\label{sec:outline}
The proof of Theorem \ref{thm:main} requires numerous fine estimates which will take up much of the space in the following sections. However, the general idea is rather simple and we would like to illustrate it here before going into the technical details. Since $u(t,b(t))=0$ for all $t\in[0,T]$, a formal differentiation with respect to $t$ suggests that we should seek an equation for $\dot b(t)$ of the form
\[
\dot b(t)=-\frac{\dot u(t,b(t))}{u_x(t,b(t))},\quad t\in[0,T).
\]
The smooth-fit principle in optimal stopping prescribes continuity of $\dot u$ and $u_x$ across the optimal boundary. When smooth-fit holds (cf.\ \cite{de2020global}) both numerator and denominator in this expression vanish. From a first order expansion in the spatial variable we then expect that the correct equation should read
\[
\dot b(t)=-\frac{\dot u_x(t,b(t))}{u_{xx}(t,b(t))},\quad t\in[0,T).
\]
The denominator is in fact explicit and it can be derived from \eqref{eq:h} by imposing $u(t,b(t))=u_x(t,b(t))=\dot u(t,b(t))=0$ and solving for $u_{xx}$. The main task is then to calculate $\dot u_x(t,b(t))$. We do it starting from the probabilistic representation of $\dot u$ provided in \eqref{eq:udot} and obtaining an estimate of the form $\dot u(t,b(t)+\delta)=\delta \Lambda (t)+o(\delta)$ so that we can identify the mixed derivative $\dot u_x(t,b(t))$ with the function $\Lambda(t)$ as
\[
\dot u_x(t,b(t))\coloneqq\lim_{\delta\downarrow 0}\frac{\dot u(t,b(t)+\delta)-\dot u(t,b(t))}{\delta}=\Lambda(t).
\]
Then we show that $t\mapsto \Lambda(t)$ is continuous. Now we provide a quick tour of our argument of proof.

Since the law of $\mathbb X$, starting from $X_t=x$, is not known explicitly, it is difficult to study directly $\dot u$ in the form of \eqref{eq:udot}. We must instead reduce the problem to the case of time-space Brownian motion $(t,W)$ for which more powerful tools are available. For that we first change the scale via a function $f(x)$ in order to reduce the SDE for $\mathbb X$ to a SDE with unit diffusion coefficient (the so-called Lamperti transform \cite{lamperti1964simple}). Then we leverage on the Lipschitz property of the optimal boundary $b$ to implement a measure change and Girsanov transformation that have two effects: they remove the drift from the SDE and change the stopping time $\sigma_*$ into the first passage time of a Brownian motion at a constant threshold (depending on the distance $h=f(x)-f(b(t))$ of the initial point from the boundary). These two steps yield a probabilistic representation of $\dot u(t,x)$ as a functional of the triple $(W,\bar W,f(b))$ where $(W_s)_{s\ge 0}$ is a Brownian motion, $(\bar W_s)_{s\ge 0}$ its running maximum and $s\mapsto f(b(s))$ the Lamperti transform of the optimal boundary. Pitman's theorem establishes the equivalence in law $(W,\bar W)=(2J-\rho,J)$ with $(\rho_s)_{s\ge 0}$ a 3-dimensional Bessel process and $(J_s)_{s\ge 0}$ its minimum over the remaining time (i.e., on $[s,\infty)$). We use that to obtain a more convenient representation of $\dot u(t,x)$ of the form 
\[
\dot u(t,x)=\E\big[\Theta_h\big(\rho,J, f(b)\big)\big],
\] 
where $(\Theta_h)_{h>0}$ is a family of non-linear functionals of the paths of the triple $(\rho,J, f(b)\big)$ (cf.\ Proposition \ref{prop:vpitman}) parametrised by the distance from the boundary $h=f(x)-f(b(t))$. The key step in the proof of Theorem \ref{thm:main} is represented by Theorem \ref{thm:expansion} which, via several technical lemmas (cf.\ Lemmata \ref{lem:delta}--\ref{lem:V2h}), yields an expansion of the form  
\[
\E\big[\Theta_h\big(\rho,J, f(b)\big)\big]=h\E\big[\Pi\big(\rho,J, f(b)\big)\big]+o(h),
\] 
for another non-linear functional $\Pi$ of the paths of $(\rho,J, f(b)\big)$. It is precisely in this expansion that we appreciate the advantage of working with the pair of processes $(\rho,J)$ rather than $(W,\bar W)$. Indeed, in numerous instances we can perform explicit calculations using that the conditional law of $J_t$ given the path $(\rho_s)_{s\le t}$ is uniform on $[0,\rho_t]$ (cf.\ Section \ref{sec:pitman} for details). Returning to the original coordinates (i.e., accounting for the relation between $h=f(x)-f(b(t))$ and $\delta=x-b(t)$) from such an expansion we finally find the form of the function $\Lambda(t)$ mentioned above and then the expression for $\dot b$ (cf.\ Proposition \ref{prop:dotb}). Continuity of the mapping $t\mapsto \Lambda(t)$ requires some additional work that we perform in Lemma \ref{lem:lambdacont}.

\section{Lamperti transform, a change of measure and Pitman's theorem}\label{sec:3}

In this section, we obtain a more tractable expression for $\dot u$ in three steps: first, we obtain a process with unit diffusion coefficient by adopting the Lamperti transform; second, we use a change of measure and Girsanov theorem to remove the drift from the new diffusion; third, we employ Pitman's theorem to rewrite our formulae in terms of a 3-dimensional Bessel process and its minimum over the remaining time. 

\subsection{Lamperti transform} 
Here we follow \cite{lamperti1964simple} for the change of scale. Given an arbitrary point $c\in\cI$ we let $f(x):=\int_c^x\sigma^{-1}(z)\ud z$ (notice that $f_x>0$) and define the process $\mathbb Y\coloneqq f(\mathbb X)$ as $Y^{t,y}_s= f(X^{t,x}_s)$ with $y=f(x)$. We notice that rescaling the spatial coordinates according to $f$ allows us to redefine our state-space $[0,T]\times \cI$ as $[0,T]\times \cJ$, with $\cJ=(f(a),f(b))$. With no loss of generality we let $\cR=(0,T_1)\times(x_1,x_2)$ and we rescale the set via $f$ to obtain $\cR_f=(0,T_1)\times(y_1,y_2)$ with $y_1=f(x_1)$ and $y_2=f(x_2)$.

Setting $\tau_{\cR}\coloneqq\inf\{s\in[0,\infty): (t\!+\!s,X_{t+s})\notin \cR\}=\inf\{s\in[0,\infty): (t\!+\!s,Y_{t+s})\notin \cR_f\}$, the dynamics of $\mathbb Y$ on the random time interval $[\![0,\tau_{\cR}]\!]$ is explicitly given by
\begin{align}\label{eq:lamperti}
Y^{t,y}_{t+s}=y+\int_0^s\gamma\big(t\!+\!r,Y^{t,y}_{t+r}\big)\ud r+B_{t+s}-B_t,\quad s\in[\![0,\tau_{\cR}]\!],
\end{align}
where 
\begin{align}\label{eq:gamma}
\gamma(t,y)\coloneqq\frac{\mu\big(t,f^{-1}(y)\big)}{\sigma(f^{-1}(y))}-\tfrac12\sigma_x\big(f^{-1}(y)\big).
\end{align}

We can write the function $\dot u$ from \eqref{eq:udot} in terms of the process $\mathbb Y$.  For that we introduce the notations 
\begin{align}\label{eq:not}
w(t,y)\coloneqq u(t,f^{-1}(y)),\quad c(t)\coloneqq f(b(t)),\quad R(t,y)\coloneqq r(t,f^{-1}(y)),\quad F(t,y)\coloneqq H(t,f^{-1}(y)).
\end{align}
The stopping times $\tau_{x_2}$ and $\sigma_*$ read $\tau_{x_2}=\tau^{t,y}_{2}\coloneqq\inf\{s\in[0,\infty): Y^{t,y}_{t+s}\ge y_2\}$ and $\sigma_*=\sigma_*^{t,y}\coloneqq\inf\{s\in[0,\infty):Y^{t,y}_{t+s}\le c(t+s)\}\wedge(T-t)$, respectively. When no confusion shall arise we also use the simpler notation $\tau_2=\tau^{t,y}_{2}=\tau_{x_2}$.

Since $\P(Y^{t,y}_s=z)=0$ for any $s\in[t,T]$ and $z\in\cJ$, we obtain $\P$-a.s. 
\begin{align*}
&D_{t,s}=D^y_{t,s}=\exp\Big(-\int_0^s R(t+v,Y^{t,y}_{t+v})\ud v\Big),\quad s\in[0,T_1-t],\\
&A^{t,y}_1\coloneqq \big\{c(t+s)<Y^{t,y}_{t+s}<y_2,\ \forall s\in[0,T_1-t]\big\}=\{\sigma_*=T_1-t,\ \sigma_*<\tau_{2}\},\\
&A^{t,y}_2\coloneqq\{\tau_{2}^{t,y}<T_1-t\}\cap\big\{Y^{t,y}_{t+s}>c(t+s),\ \forall s\in[0,\tau_{2}^{t,y}]\big\}=\{\tau_{2}<\sigma_*\},\\
&A^{t,y}_s\coloneqq\big\{c(t+v)<Y^{t,y}_{t+v}<y_2,\ \forall v\in[0,s]\big\}=\{s<\sigma_*\wedge \tau_{2}\},\quad s\in[0,T_1-t).
\end{align*}
Noting $\dot u(t,x)=\dot w(t,y)$ we rewrite \eqref{eq:udot} in the new parametrisation as
\begin{align}\label{eq:wdot}
\begin{aligned}
\dot w(t,y)&=\E\Big[\1_{A^{t,y}_1}D^y_{t,T_1-t}\dot w\big(T_1,Y^{t,y}_{T_1}\big)\\
&\qquad+\!\1_{A^{t,y}_2}D^y_{t,\tau_{2}}\dot w\big(t\!+\!\tau_{2},y_2\big)\!+\!\int_0^{T_1-t}\!\!\!\!\1_{A^{t,y}_s}D^y_{t,s}\,F\big(t\!+\!s,Y^{t,y}_{t+s}\big)\ud s\Big].
\end{aligned}
\end{align}
Throughout the paper we avoid using superscript $(t,y)$ on stopping times when the initial condition of the process is clear from the context. This allows us to keep a simple notation when no confusion shall arise.

\subsection{Change of measure and Girsanov theorem}\label{sec:Girs}
The assumption of local Lipschitz continuity of $b$ is key to developing an approach based on a change of probability measure, which transfers the dependence on $b$ from the stopping time $\sigma_*$ to a Dol\'eans-Dade exponential. In particular, we are going to use the shifted process $Y^{t,y}_{t+s}-c(t+s)$ which can be reduced to a Brownian motion, thanks to the change of measure. As a consequence, the optimal stopping time $\sigma_*$ reduces to the first time such Brownian motion hits zero.

We start by observing that $\dot c(t)=\dot b(t)/\sigma(b(t))$ is well-defined and bounded for a.e.\ $t\in[0,T_1]$, under the assumptions of Theorem \ref{thm:main}. Then, for $s\in[\![0,\tau_\cR]\!]$
\begin{align*}
\begin{aligned}
Y^{t,y}_{t+s}-c(t+s)&=y-c(t)+\int_0^s \big(\gamma(t+v,Y^{t,y}_{t+v})-\dot c(t+v)\big)\ud v +B_{t+s}-B_t\\
&=y-c(t)+W_{t+s}-W_t,
\end{aligned}
\end{align*}
where $W_{t+s}-W_t\coloneqq B_{t+s}-B_t+\int_0^s \big(\gamma(t+v,Y^{t,y}_{t+v})-\dot c(t+v)\big)\ud v$, for $s\in[\![0,\tau_\cR]\!]$, is a Brownian motion under the measure $\PP=\PP_{t,y}$ defined via the Radon-Nikodym derivative
\begin{align*}
&Z_{t,\tau_\cR}\coloneqq\frac{\ud \PP}{\ud \P}\Big|_{\cF_{\tau_\cR}}\\
&=\exp\Big(-\!\int_0^{\tau_\cR}\!\!\big(\gamma(t\!+\!s,Y^{t,y}_{t+s})\!-\!\dot c(t\!+\!s)\big)\ud B_{t+s}\!-\!\tfrac12\int_0^{\tau_\cR}\!\!\big(\gamma(t\!+\!s,Y^{t,y}_{t+s})\!-\!\dot c(t\!+\!s)\big)^2\ud s\Big).
\end{align*}
We can now use this change of measure to rewrite the three terms in \eqref{eq:wdot} in a more convenient form. Since $t$ is fixed throughout this discussion, we prefer to simplify our notation and with a slight abuse of notation but with no loss of generality we relabel $W_{t+s}-W_t$ as $-W_s$, where the change of sign is allowed by symmetry of Brownian motion. Moreover, we set $h\coloneqq y-c(t)>0$.

For the first term in \eqref{eq:wdot} we have
\begin{align*}
\begin{aligned}
&\E\big[\1_{A^{t,y}_1}D^y_{t,T_1-t}\dot w\big(T_1,Y^{t,y}_{T_1}\big)\big]=\EE\big[\1_{A^{t,y}_1}D^y_{t,T_1-t}\dot w\big(T_1,h+c(T_1)-W_{T_1-t}\big)\widetilde L^{h}_{t,T_1-t}\big],
\end{aligned}
\end{align*}
where 
\begin{align}\label{eq:Lh}
\widetilde L^{h}_{t,s}=\frac{1}{Z_{t,s}}=\exp\Big(-\int_0^{s}\!\!\gamma_{t,v}(h-W_v)\ud W_{v}\!-\!\tfrac12\int_0^{s}\!\!\big(\gamma_{t,v}(h-W_{v})\big)^2\ud v\Big),
\end{align}
is well-defined for $s\in [\![0,\tau_\cR]\!]$, with
\begin{align}\label{eq:gammatv}
\gamma_{t,v}(\xi)\coloneqq \gamma(t\!+\!v,c(t\!+\!v)\!+\!\xi)\!-\!\dot c(t\!+\!v).
\end{align}
 Notice that the discount factor reads explicitly as
\begin{align*}
D^y_{t,s}=\exp\Big(-\int_0^{s}R_{t,v}(h-W_v)\ud v\Big)\eqqcolon \widetilde D^h_{t,s},\quad s\in[0,T_1-t],
\end{align*}
with $R_{t,v}(\xi)\coloneqq R(t\!+\!v,c(t\!+\!v)\!+\!\xi)$. Somewhat more importantly, we can express the event $A^{t,y}_1$ in terms of $h-W$ as follows
\begin{align*}
A^{t,y}_1=\big\{0<h-W_s<y_2-c(t+s), \forall s\in[0,T_1-t]\big\}.
\end{align*}
Since $t$ is fixed, in order to keep track of the dependence on $h$ we use the notation $\widetilde A^{h}_1=A^{t,y}_1$.

For the second term in \eqref{eq:wdot} we use analogous arguments to obtain
\begin{align*}
\begin{aligned}
\E\big[\1_{A^{t,y}_2}D^y_{t,\tau_{2}}\dot w\big(t\!+\!\tau_{2},y_2\big)\big]=\EE\big[\1_{\widetilde A^{h}_2}\widetilde D^h_{t,\tau^h_{2}}\dot w\big(t\!+\!\tau^h_{2},y_2\big)\widetilde L^{h}_{t,\tau^h_2}\big],
\end{aligned}
\end{align*}
where $\widetilde L^{h}_{\tau^h_2}$ is the same as in \eqref{eq:Lh} with 
\begin{align}\label{eq:tau2h}
\tau^h_2\coloneqq \inf\{s\in[0,\infty): h-W_s+c(t+s)\ge y_2\},
\end{align}
and we rewrite the event $A^{t,y}_2$ as
\[
A^{t,y}_2=\big\{\tau^h_2<T_1-t\big\}\cap\{h-W_s>0,\,\forall s\in[0,\tau^h_2]\}=\widetilde A^h_2.
\]

Finally, for the integral term in \eqref{eq:wdot} we have
\begin{align*}
\int_0^{T_1-t}\!\!\!\!\E\big[\1_{A^{t,y}_s}D^y_{t,s}\,F\big(t\!+\!s,Y^{t,y}_{t+s}\big)\big]\ud s=\int_0^{T_1-t}\!\!\EE\big[\1_{\widetilde A^{h}_s}\widetilde D^h_{t,s}\,F_{t,s}\big(h\!-\!W_s\big)\widetilde L^{h}_{t,s}\big]\ud s,
\end{align*}
where $F_{t,s}(\xi)=F(t\!+\!s,c(t\!+\!s)+\xi)$, 
\[
A^{t,y}_s=\big\{0<h-W_v<y_2-c(t+v),\ \forall v\in[0,s]\big\}=\widetilde A^h_s,\quad s\in[0,T_1-t],
\]
and $\widetilde L^{h}_s$ is the same as in \eqref{eq:Lh} but with $T_1-t$ replaced by $s$.

Combining the expressions above and plugging them back into \eqref{eq:wdot} yields
\begin{align}\label{eq:dot.w.1}
\begin{aligned}
\dot w(t,y)&=\EE\big[\1_{\widetilde A^{h}_1}\widetilde L^{h}_{t,T_1-t}\widetilde D^h_{t,T_1-t}\dot w\big(T_1,h+c(T_1)-W_{T_1-t}\big)+\1_{\widetilde A^{h}_2}\widetilde L^{h}_{t,\tau^h_2}\widetilde D^h_{t,\tau^h_{2}}\dot w\big(t\!+\!\tau^h_{2},y_2\big)\big]\\
&\quad+\int_0^{T_1-t}\!\!\EE\big[\1_{\widetilde A^{h}_s}\widetilde L^{h}_{t,s}\widetilde D^h_{t,s}F_{t,s}\big(h\!-\!W_s\big)\big]\ud s.
\end{aligned}
\end{align}
Since $\dot w(t,c(t))=0$ and $\dot w$ is continuous, then $\dot w(t,y) = \dot w(t,h+c(t))$ vanishes when $h\to 0$ (i.e., $y\downarrow c(t)$). In the rest of our study we denote 
\begin{align}\label{eq:Vh}
V_h(t)\coloneqq\dot w(t,h+c(t)).
\end{align}
\begin{remark}\label{rem:P1}
Recall $\cR=(0,T_1)\times (x_1,x_2)$ and the rescaled $\cR_f$. With no loss of generality we are assuming $x_1<b(t)<x_2$ for all $t\in[0,T_1]$. Then $y_1<c(t)<y_2$ for all $t\in[0,T_1]$. Notice that
we can write 
$\sigma_*\wedge\tau_2=\sigma^h_*\wedge\tau^h_2$, with $\tau^h_2$ as in \eqref{eq:tau2h} and
$\sigma^h_*=\inf\{s\in [0,T_1-t]: h-W_s\le 0\}$.
Then,  $0\le h-W_s\le y_2-c(t+s)$ for all $s\in[\![0,\sigma^h_*\wedge\tau^h_2]\!]$. It follows that 
\[
(t+s,c(t+s)+h-W_s)\in\overline\cR_f,\quad\text{for $s\in[\![0,\sigma^h_*\wedge\tau^h_2]\!]$}.
\] 
\end{remark}

\subsection{A representation of $\dot w$ via Pitman's theorem}\label{sec:pitman} 
We start with a review of useful facts about the 3-dimensional Bessel process and Pitman's theorem. Denoting $\bar W_t=\sup_{0\le s\le t} W_s$ Pitman's theorem (cf.\ \cite[Thm.\ VI.3.5]{revuz2013continuous}) tells us that the process $\rho_t=2\bar W_t-W_t$, for $t\ge 0$, is a 3-dimensional Bessel process starting at zero. Then $(\rho_t)_{t\ge 0}$ is described by the SDE 
\begin{align}\label{eq:SDErho}
\rho_t=\int_0^t\rho^{-1}_s\ud s+\beta^0_t, \quad t\ge 0,
\end{align}
where $(\beta^0_t)_{t\ge 0}$ is a Brownian motion. Moreover, Pitman's theorem also tells us that defining $J_t=\inf_{s\ge t}\rho_s$, the following equality in law holds
\begin{align}\label{eq:pitman}
\big(\rho_t,J_t\big)_{t\ge 0}=\big(2\bar W_t-W_t,\bar W_t\big)_{t\ge 0}.
\end{align}
We emphasise that $\rho_t>0$ for all $t>0$. Moreover, $\rho_t=|\widehat B_t|$ where $\widehat B$ is a 3-dimensional Brownian motion and here $|\cdot|$ denotes the Euclidean norm in $\R^3$. Then, writing the integral in spherical coordinates, it is not hard to verify that 
\begin{align}\label{eq:rhoint}
\begin{aligned}
\E\Big[\Big(\frac{1}{\rho_1}\Big)^p\Big]=\frac{2}{(2\pi)^{\frac12}}\int_0^\infty r^{2-p}\e^{-\frac{r^2}{2}}\ud r<\infty,\quad\text{for $p\in[1,3)$}.
\end{aligned}
\end{align}
Since $\rho_t=\sqrt{t}\rho_1$ in law, the result above yields $\E[\rho_t^{-p}]<\infty$ for $p\in[1,3)$.

Denoting $\cF^\rho_t:=\sigma(\rho_s,\,s\le t)$ it can be shown that $\P(J_t\le u|\cF^\rho_t)=(u\wedge\rho_t)/\rho_t$, for any $t>0$ (cf.\ \cite[Cor.\ VI.3.6]{revuz2013continuous}), i.e., the law of $J_t$ conditional upon $\cF^\rho_t$ is uniform on $[\![0,\rho_t]\!]$. This result extends to stopping times for the filtration $(\cF^\rho_t)_{t\ge 0}$ as follows: let $\tau$ be a bounded $(\cF^\rho_t)_{t\ge 0}$-stopping time and let $(\tau_n)_{n\in\N}$ be a sequence of $(\cF^\rho_t)_{t\ge 0}$-stopping times, taking finitely many values $\tau_n\in\{t_0,t_1,\ldots, t_n\}$ and decreasing to $\tau$ when $n\to\infty$; then, $J_{\tau_n}\downarrow J_\tau$ by continuity of $t\mapsto \rho_t$ and, for any $A\in\cF^\rho_\tau$
\begin{align*}
\E\big[1_A \1_{\{J_\tau\ge u\}}\big]&=\lim_{n\to\infty} \E\big[1_A \1_{\{J_{\tau_n}\ge u\}}\big]=\lim_{n\to\infty} \sum_{j=0}^{n}\E\big[1_A \E[\1_{\{J_{t_j}\ge u\}}\big|\cF^\rho_{t_j}\big]\1_{\{\tau_n=t_j\}}\big]\\
&=\lim_{n\to\infty} \sum_{j=0}^{n}\E\big[1_A\big(1-u/\rho_{t_j}\big)^+\1_{\{\tau_n=t_j\}}\big]=\lim_{n\to\infty} \E\big[1_A\big(1-u/\rho_{\tau_n}\big)^+\big]\\
&=\E\big[1_A\big(1-u/\rho_{\tau}\big)^+\big];
\end{align*}
the latter implies 
\begin{align}\label{eq:unifJ}
\P(J_\tau\ge u|\cF^\rho_\tau)= (1-u/\rho_{\tau})^+,\quad u\geq 0,
\end{align}
for any bounded $(\cF^\rho_t)_{t\ge 0}$-stopping time $\tau$, as needed.
  
With the notation introduced above, we can now use Pitman's theorem to obtain the desired representation of $V_h(t)$ (cf.\ \eqref{eq:Vh}).
\begin{proposition}\label{prop:vpitman}
Let $\cR_f=(0,T_1)\times(y_1,y_2)$ be such that $y_1<c(t)<y_2$ for all $t\in[0,T_1]$. 
Then, 
\begin{align}\label{eq:dot.w.2}
\begin{aligned}
V_h(t)&=\E\big[\1_{B^{h}_1}L^{h}_{t,T_1-t}D^h_{t,T_1-t}\dot w\big(T_1,h\!+\!c(T_1)\!+\!\rho_{T_1-t}\!-\!2J_{T_1-t}\big)\!+\!\1_{B^{h}_2}L^{h}_{t,\theta_h}D^h_{t,\theta_h}\dot w\big(t\!+\!\theta_h,y_2\big)\big]\\
&\quad+\int_0^{T_1-t}\!\!\E\big[\1_{B^{h}_s}  L^{h}_{t,s}D^h_{t,s}F_{t,s}\big(h\!+\!\rho_s-\!2J_s\big)\big]\ud s,
\end{aligned}
\end{align}
where $\theta_h=\inf\{s\in[0,\infty):h+\rho_s-2J_s+c((t+s)\wedge T_1)\ge y_2\}$, the events $B^h_1$, $B^h_2$ and $B^h_s$ are defined as 
\begin{align*}
\begin{aligned}
B^h_1\!=\!\big\{J_{T_1-t}\!\le \!h\big\}\!\cap\!\big\{\theta_h\!\ge\! T_1-t\big\},\quad B^h_2\!=\!\big\{J_{\theta_h}\!\le\! h\big\}\!\cap\!\big\{\theta_h\!<\!T_1-t\big\},\quad B^h_s\!=\!\big\{J_{s}\!\le \!h\big\}\!\cap\!\big\{\theta_h\!\ge\! s\big\},
\end{aligned}
\end{align*}
and the processes $s\mapsto L^{h}_{t,s}$ and $s\mapsto D^{h}_{t,s}$ are defined as
\begin{align*}
\begin{aligned}
&L^{h}_{t,s}=\exp\Big(\int_0^{s}\!\!\gamma_{t,v}(h+\rho_v-2J_v)\ud (\rho_v-2J_v)\!-\!\tfrac12\int_0^{s}\!\!\big(\gamma_{t,v}(h+\rho_v-2J_v)\big)^2\ud v\Big),\\
&D^h_{t,s}=\exp\Big(-\int_0^s R_{t,v}(h+\rho_v-2J_v)\ud v\Big),
\end{aligned}
\end{align*} 
for $s\in[\![0,\theta_h\wedge\vartheta^*_h]\!]$, with
$\vartheta^*_h\coloneqq\inf\{s\in[0,T_1-t): h+\rho_s-2J_s \le 0\}\wedge(T_1-t)$.
\end{proposition}
\begin{proof}
Let us start by noticing that the events $\widetilde A^h_1$, $\widetilde A^h_2$ and $\widetilde A^h_s$ can be rewritten $\P$-a.s.\ in terms of the pair $(W_t,\bar W_t)$ as 
\begin{align*}
\begin{aligned}
&\widetilde A^h_1=\big\{\bar W_{T_1-t}\le h\big\}\cap\big\{h-W_s+c(t+s)<y_2,\,\forall s\in[0,T_1-t]\big\},\\
&\widetilde A^h_2=\big\{\bar W_{\tau^h_2}\le h\big\}\cap\big\{\tau^h_2<T_1-t\big\},\\
&\widetilde A^h_s=\big\{\bar W_{s}\le h\big\}\cap\big\{h-W_v+c(t+v)<y_2,\,\forall v\in[0,s]\big\}.
\end{aligned}
\end{align*}
From \eqref{eq:pitman} we know that $(W,\bar W)=(2J-\rho,J)$ in law. Then, it is clear that the events $B^h_1$, $B^h_2$ and $B^h_s$ are the analogue of the ones above but for the pair $(2J-\rho,J)$. Likewise the stopping times $\theta_h$ and $\vartheta^*_h$ are the analogue of $\tau^h_2$ and $\sigma^h_*$ (cf.\ Remark \ref{rem:P1}) and the processes $s\mapsto L^h_{t,s},D^h_{t,s}$ are the analogue of the processes $s\mapsto \widetilde L^h_{t,s},\widetilde D^h_{t,s}$.

Invoking \eqref{eq:pitman} we can rewrite \eqref{eq:dot.w.1} as in \eqref{eq:dot.w.2}.
The use of the measure $\P$ is with no loss of generality as long as the Bessel process is defined on the original space $(\Omega,\cF,\P)$.
\end{proof}

For the ease of exposition, in what follows we rewrite the expression for \eqref{eq:dot.w.2} as 
\begin{align}\label{eq:Vht.2}
V_h(t)=V^h_1(t)+V^h_2(t)+\int_0^{T_1-t} V^h_s(t)\ud s,
\end{align} 
with
\begin{align}\label{eq:V1h}
\begin{aligned}
&V_1^h(t)\coloneqq\E\big[\1_{B^{h}_1}L^{h}_{t,T_1-t}D^h_{t,T_1-t}\dot w\big(T_1,c(T_1)\!+\!h\!+\!\rho_{T_1-t}\!-\!2J_{T_1-t}\big)\big],\\
&V_2^h(t)\coloneqq\E\big[\1_{B^{h}_2}L^{h}_{t,\theta_h}D^h_{t,\theta_h}\dot w\big(t\!+\!\theta_h,y_2\big)\big], \\
&V_s^h(t)\coloneqq\E\big[\1_{B^{h}_s} L^{h}_{t,s} D^h_{t,s}F_{t,s}\big(h\!+\!\rho_s-\!2J_s\big)\big].
\end{aligned}
\end{align}
One of the key steps in the proof of Theorem \ref{thm:main} is to obtain an expansion up to order $o(h)$ for each term on the right-hand side of \eqref{eq:Vht.2}. In order to do that we need some preliminary estimates on the process $(\rho_t)_{t\ge 0}$ and some related stopping times. However, before proceeding further the next remark simplifies substantially our analysis.
\begin{remark}\label{rem:gammaR}
Let $\tau\in[0,T_1-t]$ be a random time. On the event $\{J_\tau\leq h\}\cap\{\tau\leq \theta_h\}$, we have $c(t+v)-h\leq h+\rho_v-2J_v+c(t+v)<y_2$, for $v\in[\![0,\tau)\!)$. Moreover, we can assume $c(t+v)-h\geq y_1$ because $y_1<c(s)$ for all $s\in[0,T_1]$ and we are interested in the limit as $h\downarrow 0$. Thus, with no loss of generality we can replace the functions $\gamma(t,y)$, $R(t,y)$,
 $F(t,y)$ with bounded and Lipschitz-continuous functions on $\R^2$ which coincide with the original ones on $\overline \cR_f$. Such replacement will be tacitly assumed throughout the paper, i.e., we assume with no loss of generality that $\gamma(t,y)$, $R(t,y)$,
 $F(t,y)$ are bounded and Lipschitz-continuous. For the same reason we can assume $\dot w(t,y)$ to be bounded and uniformly continuous.
\end{remark}

\subsection{Key estimates for 3-dimensional Bessel processes}
Here we prove a few simple but important technical lemmas. Throughout this section $(\rho_t)_{t\ge 0}$ will always denote a 3-dimensional Bessel process starting from zero and $J_t=\inf_{s\ge t}\rho_s$.
\begin{lemma}\label{lem:3dBp1}
Let $\varphi:[0,\infty)\to\R$ be a continuous function which is bounded from below, with $\varphi(0)<0$ and $\varphi(t)-\varphi(s)\ge -c_\varphi (t-s)$ for $0\le s<t$ and some constant $c_\varphi>0$. For $h\in(0,-\varphi(0))$ let
\begin{align}\label{eq:tauphi}
\tau^{\pm h}_\varphi\coloneqq\inf\{s\ge 0:\rho_s+\varphi(s)\ge\pm h \}.
\end{align}
Then, there exists a Brownian motion $(\beta_t)_{t\ge 0}$, independent of $\cF^\rho_{\tau^{-h}_\varphi}$ such that $\tau^{+h}_{\varphi}\le \tau^{-h}_\varphi+\sigma^{\beta}_{2h}$ on the event $\{\tau^{-h}_\varphi<\infty\}$ with 
\begin{align}\label{eq:sigbeta}
\sigma^{\beta}_{2h}\coloneqq\inf\{s\ge 0:\beta_s\ge 2h+c_\varphi s\}.
\end{align}
\end{lemma}
\begin{proof}
By continuity of $t\mapsto(\rho_t,\varphi(t))$ we have $\P(\tau^{-h}_\varphi>0)=1$, because $h\in(0,-\varphi(0))$. Moreover, $\P(\tau^{-h}_\varphi< \tau^{+h}_\varphi)=1$ and we can write
\begin{align}\label{eq:tauh+}
\tau^{+h}_\varphi=\tau^{-h}_\varphi+\inf\big\{s\ge 0:\rho_{\tau^{-h}_\varphi+s}+\varphi(\tau^{-h}_\varphi+s)\ge h\big\}.
\end{align}
Upon noticing also that $\rho_{\tau^{-h}_\varphi}+\varphi(\tau^{-h}_\varphi)=-h$ on the event $\{\tau^{-h}_\varphi<\infty\}$, from the SDE for the Bessel process (cf.\ \eqref{eq:SDErho}) we obtain
\begin{align}\label{eq:rhodyn}
\begin{aligned}
\rho_{\tau^{-h}_\varphi+s}&=\rho_{\tau^{-h}_\varphi}+\int_{\tau^{-h}_\varphi}^{\tau^{-h}_\varphi+s}\rho^{-1}_u\ud u+\beta^0_{\tau^{-h}_\varphi+s}-\beta^0_{\tau^{-h}_\varphi}\ge -h-\varphi(\tau^{-h}_\varphi)+\beta_s,
\end{aligned}
\end{align}
with $\beta_s\coloneqq \beta^0_{\tau^{-h}_\varphi+s}-\beta^0_{\tau^{-h}_\varphi}$. It is clear that $(\beta_s)_{s\ge 0}$ is a $(\cF^\rho_{\tau^{-h}_\varphi+s})_{s\ge 0}$-Brownian motion, independent of $\cF^\rho_{\tau^{-h}_\varphi}$. Plugging \eqref{eq:rhodyn} into \eqref{eq:tauh+} yields
\begin{align*}
\tau^{+h}_\varphi&\le\tau^{-h}_\varphi+\inf\big\{s\ge 0:-h-\varphi(\tau^{-h}_\varphi)+\beta_s+\varphi(\tau^{-h}_\varphi+s)\ge h\big\}\\
&\le\tau^{-h}_\varphi+\inf\big\{s\ge 0:\beta_s\ge 2 h+c_\varphi s\big\},
\end{align*}
where in the second inequality we use $\varphi(\tau^{-h}_\varphi+s)-\varphi(\tau^{-h}_\varphi)\ge -c_\varphi s$.
\end{proof}

\begin{lemma}\label{lem:convth}
With the same notation and assumptions as in Lemma \ref{lem:3dBp1} we have, for any (deterministic) $S\in(0,\infty)$ and any sequence $(h_n)_{n\in\N}\subset(0,\infty)$ such that $h_n\downarrow 0$ as $n\to\infty$
\begin{align*}
\lim_{n\to \infty}\tau_\varphi^{-h_n}=\tau^0_\varphi\quad\text{and}\quad \lim_{n\to \infty}\big(\tau^{-h_n}_\varphi \1_{\{\tau_\varphi^{-h_n}<S\}}\big)=\tau^{0}_\varphi \1_{\{\tau_\varphi^{0}<S\}},\quad \P-a.s.,
\end{align*}
where $\tau^0_\varphi\coloneqq\inf\{s\ge 0:\rho_s+\varphi(s)\ge 0\}$.
\end{lemma}
\begin{proof}
First we notice that $\P(\tau^0_\varphi<\infty)=1$, because $\rho_t\to\infty$ as $t\to\infty$, $\P$-a.s. By definition we have $\tau^{-h}_\varphi\le \tau^0_\varphi$, for every $h\in(0,|\varphi(0)|)$, $\P$-a.s. Moreover, for $h'\le h$ we have $\tau^{-h}_\varphi\le \tau^{-h'}_\varphi$, $\P$-a.s. Then, the limit $\tau_0\coloneqq\lim_{n\to\infty}\tau^{-h_n}_\varphi$ is well-defined and $\tau_0\le \tau^0_\varphi$, $\P$-a.s.

For any $n\in\N$, it holds $\rho_{\tau^{-h_n}_\varphi}+\varphi(\tau^{-h_n}_\varphi)\ge -h_n$. Then by continuity of $t\mapsto(\rho_t,\varphi(t))$, letting $n\to\infty$ yields $\rho_{\tau_0}+\varphi(\tau_0)\ge 0$. Thus proving $\tau_0\ge \tau^0_\varphi$. This shows the first limit in the statement of the lemma.

For the second limit we notice that $\{\tau^{-h_n}_\varphi<S\}\downarrow\{\tau^0_\varphi<S\}\cup\{\tau^0_\varphi=S\}$ as $n\to\infty$. However, $\P(\tau^0_\varphi=S)\le \P(\rho_S=\varphi(S))=0$. Then, the second limit in the statement of the lemma follows. 
\end{proof}
For the next lemma we observe that if $\varphi$ is Lipschitz we can write
\[
\tau^{-h}_\varphi=\inf\Big\{s\ge 0:\rho_s+\int_0^s\dot\varphi(u)\ud u\ge -h-\varphi(0)\Big\}.
\]
Then we can obtain a bound on the distribution of $\tau^{-h}_\varphi$, uniformly with respect to $h+\varphi(0)$.
\begin{lemma}\label{lem:pbdtauh}
With the same notation as in Lemma \ref{lem:3dBp1}, assume that $\varphi$ is Lipschitz continuous with $\|\dot\varphi \|_\infty\leq c_\varphi$ for some $c_\varphi>0$ and fix $T_1>0$.
There exists a constant $C$ depending only on $T_1$ and $c_\varphi$ (and not depending on $h+\varphi(0)$) such that,
for any $0< t_1<t_2<T_1$, we have
\begin{align}\label{eq:lim1}
\P\left(\tau^{-h}_\varphi\in [t_1,t_2]\right)\leq \frac{C}{\sqrt{t_1}}\sqrt{t_2-t_1}.
\end{align}
\end{lemma}
\begin{proof}
For notational convenience, let $\delta=t_2-t_1$ and $h_\varphi=h+\varphi(0)$.
On the event $\{\tau^{-h}_\varphi\in [t_1,t_2]\}$, we have $\rho_{t_1}+\int_0^{t_1}\dot\varphi(u)\ud u\leq -h_\varphi$ and
there exists $s\in [0,\delta]$ such that $\rho_{t_1+s}+\int_0^{t_1+s}\dot\varphi(u)\ud u\geq -h_{\varphi}$. That is equivalent to
\[
\rho_{t_1}\le -\tilde h_{\varphi}\quad\text{and}\quad\rho_{t_1+s}-\rho_{t_1}+\int_{t_1}^{t_1+s}\dot\varphi(u)\ud u\geq -\tilde h_\varphi -\rho_{t_1},
\]
where $\tilde h_\varphi= h_\varphi+\int_0^{t_1}\dot\varphi(u)\ud u$. Recall that $\rho_s=|\widehat B_s|$ where $\widehat B$ is a three dimensional Brownian motion. Then, $\rho_{t_1+s}-\rho_{t_1}=|\widehat B_{t_1+s}|-|\widehat B_{t_1}|\leq |\widehat B_{t_1+s}-\widehat B_{t_1}|\eqqcolon \tilde \rho_s$.
It thus follows
\begin{align*}
\P\left(\tau^{-h}_\varphi\in [t_1,t_2]\right)&\leq\P\left(\rho_{t_1}\leq -\tilde h_\varphi, \sup_{0\le s\le \delta}
\left(\tilde \rho_s+\int_{t_1}^{t_1+s}\dot\varphi(u)\ud u\right)\geq -\tilde h_\varphi-\rho_{t_1}\right)\\
&\le \P\left(\rho_{t_1}\leq -\tilde h_\varphi, \sup_{0\le s\le \delta}
\left(\tilde \rho_s+c_\varphi s\right)\geq -\tilde h_\varphi-\rho_{t_1}\right),
\end{align*}
where we used $|\dot\varphi(u)|\leq c_\varphi$. Since  $(\tilde\rho_s)_{s\ge 0}$ is another Bessel process independent of $\cF^\rho_{t_1}$, then by tower property
\begin{align*}
\P\left(\tau^{-h}_\varphi\in [t_1,t_2]\right)&\leq\E\Big[\1_{\{\rho_{t_1}\leq -\tilde h_\varphi\}}\P\Big(\sup_{0\le s\le \delta}
\left(\tilde \rho_s+c_\varphi s\right)\geq -\tilde h_\varphi-\rho_{t_1}\Big|\cF^\rho_{t_1}\Big)\Big]\\
     &=\int_0^{(-\tilde h_\varphi)^+}\P(\rho_{t_1}\in \ud y)\P\left(\sup_{0\leq s \leq \delta}\left(\tilde \rho_s+c_\varphi s\right)\geq (-\tilde h_\varphi)^+-y\right)\\
     &=\int_0^{(-\tilde h_\varphi)^+}\sqrt{\frac{2}{\pi t_1}}\frac{y^2}{t_1}\e^{-\frac{y^2}{2t_1}}\P\left(\sup_{0\leq s \leq \delta}\left(\tilde \rho_s+c_\varphi s\right)\geq (-\tilde h_\varphi)^+-y\right)\ud y,
\end{align*}
where in the second line we used that the upper bound is trivially zero if $\tilde h_\varphi\ge 0$ and in the final line we used the density of $\rho_{t_1}$ (see
\cite[Section VI.3]{revuz2013continuous}). Since $x\e^{-x}\leq \e^{-1}<1$ for $x>0$, Fubini's theorem and a simple change of variable of integration yield
\begin{align*}
\P\Big(\tau^{-h}_\varphi\in [t_1,t_2]\Big)&\leq\sqrt{\frac{8}{\pi t_1}}\E\Big[\int_0^{(-\tilde h_\varphi)^+}\!\!\! \1_{\big\{y\leq \sup_{0\leq s\leq \delta}\big(\tilde\rho_s+c_\varphi s\big)\big\}}\ud y\Big]\\
&\leq \sqrt{\frac{8}{\pi t_1}}\E\Big[\sup_{0\leq s\leq \delta}\big(\tilde\rho_s+c_\varphi s\big) \Big]
\le\sqrt{\frac{8}{\pi t_1}}\Big(\sqrt{\delta}\E\Big[\sup_{0\leq s\leq 1}\tilde\rho_s\Big]+c_\varphi \delta\Big),
\end{align*}
where in the last expression we used the scaling property $\tilde\rho_s=\sqrt{\delta}\tilde\rho_{s/\delta}$.
\end{proof}

As pointed out in the proof of \cite[Thm.\ VI.3.5]{revuz2013continuous}, for every $s\le t$ we have $J_s=J_t\wedge \inf_{s\le u\le t}\rho_u$, therefore implying that for every $s\le t$, $J_s$ is measurable with respect to $\cF^\rho_t\vee\sigma(J_t)$. Hence, the filtration $(\cF^{\rho,J}_t)_{t\ge 0}$ generated by the pair $(\rho_t,J_t)_{t\ge 0}$ coincides with the filtration $\big(\cF^\rho_t\vee\sigma(J_t)\big)_{t\ge 0}$. It is shown in the same proof (cf.\ also \cite[Cor.\ VI.3.7]{revuz2013continuous}) that $X_t\coloneqq 2J_t-\rho_t$ is a Brownian motion with respect to the filtration $\big(\cF^\rho_t\vee\sigma(J_t)\big)_{t\ge 0}$ and actually $\cF^X_t=\cF^{\rho,J}_t$, where $\cF^X\coloneqq\sigma(X_s,s\le t)$. Moreover, $J_t=\sup_{0\le s\le t}X_s$. Equipped with these notions we can prove our next lemma. 

\begin{lemma}\label{lem:expbd}
Let $(\Gamma_t)_{t\ge 0}$ be a process adapted to the filtration $(\cF^{\rho,J}_t)_{t\ge 0}$. Let $\sigma\le \tau$ be stopping times for such filtration bounded by a deterministic $S>0$. If $\sup_{t\in[0,S]}|\Gamma_{t}|\le K$, then
\begin{align*}
\E\Big[\e^{\int_\sigma^{\tau} \Gamma_s\ud \rho_s}\Big|\cF^{\rho,J}_\sigma\Big]\le \sqrt{2}\e^{5 K^2 S}.
\end{align*}
\end{lemma}
\begin{proof}
Using the equality $X_t=2J_t-\rho_t$ we can write
\begin{align*}
\begin{aligned}
\int_\sigma^{\tau} \Gamma_s\ud \rho_s&=-\int_\sigma^{\tau} \Gamma_s\ud X_s+2\int_\sigma^{\tau} \Gamma_s\ud J_s\\
&\le-\int_\sigma^{\tau} \Gamma_s\ud X_s+2 K (J_{\tau}-J_\sigma)\le-\int_\sigma^{\tau} \Gamma_s\ud X_s+2 K (J_{S}-J_\sigma), 
\end{aligned}
\end{align*}
where the final inequality uses that $J_\tau-J_\sigma\le J_S-J_\sigma$, because $(J_t)_{t\ge 0}$ is non-decreasing.
Substituting into the exponential function we obtain
\begin{align}\label{eq:Gbd0}
\begin{aligned}
\E\Big[\e^{\int_\sigma^{\tau} \Gamma_s\ud \rho_s}\Big|\cF^X_\sigma\Big]&\le\E\Big[\e^{-\int_\sigma^{\tau} \Gamma_s\ud X_s}\e^{2K (J_S-J_\sigma)}\Big|\cF^X_\sigma\Big]\\
&\le\E\Big[\e^{-2\int_\sigma^{\tau} \Gamma_s\ud X_s}\Big|\cF^X_\sigma\Big]^{\frac12}\E\Big[\e^{4K (J_S-J_\sigma)}\Big|\cF^X_\sigma\Big]^{\frac12},  
\end{aligned}
\end{align}

Boundedness of $\Gamma_{s}$ implies that $\e^{-2\int_0^{t} \Gamma_s\ud X_s-2\int_0^{t}\Gamma^2_s\ud s}$ is a martingale on $[0,S]$ and therefore, by optional sampling
\begin{align}\label{eq:Gmart}
\E\Big[\e^{-2\int_\sigma^{\tau} \Gamma_s\ud X_s}\Big|\cF^X_\sigma\Big]=\E\Big[\e^{-2\int_\sigma^{\tau} \Gamma_s\ud X_s-2\int_\sigma^{\tau}\Gamma^2_s\ud s}\e^{2\int_\sigma^{\tau} \Gamma^2_s\ud s}\Big|\cF^X_\sigma\Big]\le \e^{2K^2 S}.
\end{align}
Recalling that $J_t=\sup_{0\le s\le t}X_s$ we have 
\begin{align*}
\begin{aligned}
J_S&=J_\sigma\vee\sup_{\sigma\le t\le S}X_t=J_\sigma\vee\big(\sup_{\sigma\le t\le S}(X_t-X_\sigma)+X_\sigma\big)\le J_\sigma+\sup_{0\le t\le S}\bar X_t,
\end{aligned}
\end{align*}
where $\bar X_t=X_{\sigma+t}-X_\sigma$ is a $(\cF^X_{\sigma+t})$-Brownian motion independent of $\cF^X_\sigma$. Then
\begin{align*}
\E\Big[\e^{4K (J_S-J_\sigma)}\Big|\cF^X_\sigma\Big]
\le \E\Big[\e^{4K \sup_{0\le t\le S} \bar X_t}\Big].
\end{align*}
By the equivalence in law $\sup_{0\le t\le S}\bar X_s=|X_S|=\sqrt{S}|X_1|$ we have
\begin{align}\label{eq:Gbd}
\begin{aligned}
\E\Big[\e^{4K \sup_{0\le t\le S} \bar X_t}\Big]&=2\int_0^{+\infty}\e^{4K\sqrt{S}x}\frac{1}{\sqrt{2\pi}}\e^{-\frac{x^2}{2}}\ud x\\
&=2\e^{8K^2 S}\int_0^{+\infty}\frac{1}{\sqrt{2\pi}}\e^{-\frac{(x-4K\sqrt{S})^2}{2}}\ud x\le 2\e^{8K^2 S}.
\end{aligned}
\end{align}
Substituting \eqref{eq:Gmart} and \eqref{eq:Gbd} into the right-hand side of \eqref{eq:Gbd0} we conclude the proof.
\end{proof}

\begin{lemma}\label{lem:3dBp2}
With the same notation and assumptions as in Lemma \ref{lem:3dBp1} set 
\[
\hat \tau^{-h}_{\varphi}\coloneqq\inf\{s\ge 0:\rho_s-2J_s+\varphi(s)\ge - h\}
\]
and notice that $\hat \tau^{-h}_\varphi\ge \tau^{-h}_\varphi$.
Then, there is a Brownian motion $(\beta_t)_{t\ge 0}$ independent of $\cF^{\rho,J}_{\hat \tau^{-h}_\varphi}$ such that, for $\sigma^\beta_{2h}$ defined as in \eqref{eq:sigbeta},
\begin{align*}
\tau^{-h}_{\varphi}\le \hat \tau^{-h}_\varphi\le \tau^{-h}_{\varphi}+\sigma^\beta_{2h},\quad\text{on the event $\{J_{\tau^{-h}_\varphi}\le h\}\cap\{\tau^{-h}_\varphi<\infty\}$}.
\end{align*}
Moreover, $\hat \tau^{-h}_\varphi\le \tau^{+h}_\varphi$ on the event $\{J_{\tau^{+h}_\varphi}\le h\}$ (notice that $J_\infty$ is well-defined by monotonicity and the claim trivially holds on $\{\tau^{+h}_\varphi=\infty\}$). 
\end{lemma}
\begin{proof}
The proof of the first claim is similar to the one of Lemma \ref{lem:3dBp1}. First we notice that since $(J_t)_{t\ge 0}$ is positive then $\hat \tau^{-h}_{\varphi}\ge \tau^{-h}_{\varphi}$ by definition of the stopping times. On the event $\{\tau^{-h}_\varphi<\infty\}$ we have $\rho_{\tau^{-h}_\varphi}+\varphi(\tau^{-h}_\varphi)=-h$. Using this equality and recalling that 
\[
\beta_s\coloneqq-(X_{\tau^{-h}_\varphi+s}-X_{\tau^{-h}_\varphi})=(\rho_{\tau^{-h}_\varphi+s}-\rho_{\tau^{-h}_\varphi})-2(J_{\tau^{-h}_\varphi+s}-J_{\tau^{-h}_\varphi})
\] 
is a $\cF^{\rho,J}_{\tau^{-h}_\varphi+s}$-Brownian motion we can write
\begin{align*}
\hat\tau^{-h}_{\varphi}&=\tau^{-h}_\varphi+\inf\{s\ge 0:\rho_{\tau^{-h}_\varphi+s}-2J_{\tau^{-h}_\varphi+s}+\varphi(\tau^{-h}_\varphi+s)\ge -h\}\\
&=\tau^{-h}_\varphi+\inf\big\{s\ge 0:\beta_s\ge 2J_{\tau^{-h}_\varphi} -\big(\varphi(\tau^{-h}_\varphi+s)-\varphi(\tau^{-h}_\varphi)\big)\big\}\le \tau^{-h}_\varphi+\sigma^\beta_{2h},
\end{align*}
where the final inequality uses that $J_{\tau^{-h}_\varphi}\le h$ and $\varphi(\tau^{-h}_\varphi+s)-\varphi(\tau^{-h}_\varphi)\ge -c_\varphi s$.

For the second claim it is sufficient to notice that on $\{\tau^{+h}_\varphi<\infty\}\cap\{J_{\tau^{+h}_\varphi}\le h\}$ we have 
\[
\rho_{\tau^{+h}_\varphi}+\varphi(\tau^{+h}_\varphi)=h\ge 2J_{\tau^{+h}_\varphi}-h\implies \rho_{\tau^{+h}_\varphi}-2J_{\tau^{+h}_\varphi}+\varphi(\tau^{+h}_\varphi)\ge -h,
\]
hence $\tau^{+h}_\varphi\ge \hat \tau^{-h}_\varphi$ as needed.
\end{proof}

\begin{remark}\label{rem:sigma}
In the next section we will often use that, given any sequence $(h_n)_{n\in\N}\subset(0,1)$ with $h_n\to 0$ as $n\to \infty$, it holds 
\[
\P\Big(\lim_{n\to\infty}\sigma^\beta_{2h_n}=0\Big)=1.
\]
The result is easily deduced by the law of iterated logarithm for Brownian motion.
\end{remark}

\section{Expansion of the time-derivative of the value function near the boundary}\label{sec:expansion}

The main result of this section is presented in the next theorem, whose proof is given formally at the end of the section and it makes use of  Lemmata
\ref{lem:barVh} and \ref{lem:V2h}. For simplicity but with no loss of generality, throughout the section we work with $\cR=(0,T_1)\times(x_1,x_2)$ such that $x_1<b(t)<x_2$ for all $t\in[0,T_1]$. We recall the notations $\dot u(t,x)=\dot w(t,y)$, $y=f(x)$, $c(t)=f(b(t))$, $\cR_f=(0,T_1)\times(y_1,y_2)$ and $V_h(t)=\dot w(t,c(t)+h)$. 
\begin{theorem}\label{thm:expansion}
For any $T_2\in[0,T_1)$, 
\begin{align}
\lim_{h\downarrow 0}\sup_{0\le t\le T_2}\Big|h^{-1}V_h(t)-\Big(V_1(t)+V_2(t)+\int_0^{T_1-t}V_s(t)\ud s\Big)\Big|=0,
\end{align}
where
\begin{align}
\begin{aligned}
&V_1(t)\coloneqq \E\Big[\frac{L_{t,T_1-t}}{\rho_{T_1-t}}D_{t,T_1-t}\dot w\big(T_1,c(T_1)+\rho_{T_1-t}\big)\1_{\{\theta= T_1-t\}}\Big],\\
&V_2(t)\coloneqq \E\Big[\frac{L_{t, \theta}}{\rho_{\theta}}D_{t,\theta}\dot w\big(t+\theta,y_2\big)\1_{\{\theta<T_1-t\}}\Big],\\
&V_s(t)\coloneqq \E\Big[\1_{\{s<\theta\}} \frac{L_{t,s}}{\rho_s} D_{t,s}F_{t,s}\big(\rho_s\big)\Big],
\end{aligned}
\end{align}
with $\theta=\theta_{t,y_2}\coloneqq\inf\{s\in[0,\infty):\rho_s+c(t+s)\ge y_2\}\wedge(T_1-t)$ and, for any $s\in[\![0,\theta]\!]$,
\begin{align*}
&L_{t,s}\coloneqq\exp\Big(\int_0^s\gamma_{t,v}(\rho_v)\ud \rho_v-\tfrac12\int_0^s\big(\gamma_{t,v}(\rho_v)\big)^2\ud v\Big)\quad\text{and}\quad D_{t,s}\coloneqq\exp\Big(-\int_0^s R_{t,v}(\rho_v)\ud v\Big).
\end{align*}
\end{theorem}
\begin{remark}\label{rem:welldef}
The proof of the theorem shows that $V_1(t)$, $V_2(t)$ and $V_s(t)$ are finite.  That is essentially a consequence of \eqref{eq:rhoint} and Lemma \ref{lem:expbd}. We also notice that on $\{\theta<T_1-t\}$ we have $\rho_\theta=y_2-c(t+\theta)\ge \bar c>0$, where $\bar c=\min_{0\le s\le T_1}(y_2-c(s))$, because $y_1<c(s)<y_2$ for all $s\in[0,T_1]$. Then $V_2(t)$ is well-defined. Due to the scaling $\rho_s=\sqrt{s}\rho_1$ in law, the function $s\mapsto V_s(t)$ may explode at most as $1/\sqrt{s}$ when $s\to 0$. Therefore it is integrable and the third term in the expansion is well-defined. Finally, we observe that $\gamma_{t,v}(\rho_v)$, $R_{t,v}(\rho_v)$ and $F_{t,v}(\rho_v)$ are bounded for $v\in[\![0,\theta]\!]$ because of Remark \ref{rem:gammaR}.
\end{remark}

\begin{remark}\label{rem:V1+V2}
It is worth observing for future reference that 
\begin{align*}
V_1(t)+V_2(t)=\E\Big[\frac{L_{t,\theta}}{\rho_\theta}D_{t,\theta}\dot w\big(t+\theta,c(t+\theta)+\rho_\theta\big)\Big].
\end{align*}
\end{remark}

We will study the behaviour of $V_h(t)$ from the decomposition given in \eqref{eq:Vht.2}--\eqref{eq:V1h}. In fact, it will be convenient to express the
last two factors in $V^h_1(t)$, $V^h_2(t)$ and $V^h_s(t)$ as functionals $\delta^{(1)}$, $\delta^{(2)}$ and $\delta^{(3)}$
 of the paths of the process $t\mapsto h+\rho_t-2J_t$. More precisely, given $\xi\in C([0,T_1];[y_1,y_2])$, for any $t\in[0,T_1]$ and $s\in[0,T_1-t]$ we set
  \[
  \delta^{(1)}_{t,T_1-t}(\xi)=e^{-\int_0^{T_1-t}R_{t,v}(\xi_v)dv}\dot w\big(T_1,c(T_1)\!+\!\xi_{T_1-t}\big), \quad
   \delta^{(2)}_{t,s}(\xi)=e^{-\int_0^{s}R_{t,v}(\xi_v)dv}\dot w\big(t+s,y_2\big)
  \]
 and
 $
 \delta^{(3)}_{t,s}(\xi)=e^{-\int_0^sR_{t,v}(\xi_v)dv}F_{t,s}(\xi_s).
 $
Then,  
 \[
 V^h_1(t)=\E\big[\1_{B^{h}_1}L^{h}_{t,T_1-t}\delta^{(1)}_{t,T_1-t}(h+\rho-2J)\big],\quad V^h_2(t)=\E\big[\1_{B^{h}_2}L^{h}_{t,\theta_h}\delta^{(2)}_{t,\theta_h}
 (h+\rho-2J)\big]
 \]
 and
 \[
 V^h_s(t)=\E\big[\1_{B^{h}_s} L^{h}_{t,s} \delta^{(3)}_{t,s}\big(h\!+\!\rho-\!2J\big)\big].
 \]
In these notations, the expressions $V_1(t)$, $V_2(t)$ and $V_s(t)$ in the statement of Theorem \ref{thm:expansion} read as
 \begin{equation*}
 \begin{aligned}
 &V_1(t)=\E\Big[\1_{\{\theta=T_1-t\}}\frac{L_{t,T_1-t}}{\rho_{T_1-t}}\delta^{(1)}_{t,T_1-t}(\rho)\Big], \quad V_2(t)=\E\Big[\1_{\{\theta<T_1-t\}}\frac{L_{t,\theta}}{\rho_{\theta}}\delta^{(2)}_{t,\theta}(\rho)\Big],\\
&\qquad\qquad\qquad\qquad\qquad V_s(t)=\E\Big[\1_{\{s<\theta\}}\frac{L_{t,s}}{\rho_s}\delta^{(3)}_{t,s}(\rho)\Big].
 \end{aligned}
 \end{equation*}

For any $\cF$-measurable function $X:\Omega\to\R$ we introduce the notations $\|X\|_p\coloneqq\E[|X|^p]^{\frac1p}$, for $p\in[1,\infty)$. Thanks to Remark \ref{rem:gammaR}
we can assume that the functionals $\delta^{(1)}$, $\delta^{(2)}$, $\delta^{(3)}$ are uniformly bounded by a constant $M_\delta$, i.e.,
\begin{align}\label{eq:deltabound}
\sup_{\substack{(t,s)\in \Delta_{T_1},\\ \xi\in C([0,T_1];[y_1,y_2])}}\big|\delta^{(i)}_{t,s}(\xi)\big|\le M_\delta,\quad i=1,2,3,
\end{align}
where $\Delta_{T_1}=\{(s,t)\in [0,T_1]\times [0,T_1] : s+t\leq T_1\}$.
Moreover, $(t,s,\xi)\mapsto\delta^{(i)}_{t,s}(\xi)$, $i=1,2,3$, can be taken jointly uniformly continuous when the space $C([0,T];[y_1,y_2])$ is equipped with the supremum norm. Finally, we observe that for any random time $\tau$, on the event $\{J_\tau\le h\}$ the process $J$ can be replaced with $J^h\coloneqq J\wedge h$ on $[\![0,\tau]\!]$.

We can now state the following lemma, the proof of which follows easily from the uniform continuity properties of the $\delta^{(i)}$'s and it is therefore omitted.
\begin{lemma}\label{lem:delta}
There exist non-decreasing functions $h\mapsto \varepsilon_0(h)$ and $\lambda\mapsto \varepsilon_1(\lambda)$,  with $\lim_{h\downarrow 0}\varepsilon_0(h)=0$ and $\lim_{\lambda\downarrow 0}\varepsilon_1(\lambda)=0$,
satisfying the following properties:
\begin{enumerate}
\item With probability one for all $t\in[0,T_1]$, 
\[
\left|\delta^{(1)}_{t,T_1-t}\big(h+\rho-2(J\wedge h)\big)-\delta^{(1)}_{t,T_1-t}\big(\rho\big)\right|+
\sum_{i=2}^3\sup_{s\in [0,T_1-t]}\left( \left|\delta^{(i)}_{t,s}\big(h+\rho-2(J\wedge h)\big)-\delta^{(i)}_{t,s}\big(\rho\big)\right|
\right)\leq \varepsilon_0(h).
\]
\item For all $(t,s)\in\Delta_{T_1}$ we have for all $\lambda\in[0,T_1-t-s]$ 
\[
\sup_{\substack{\eta\in[0,\lambda] \\ \xi\in C\left([0,T_1];[y_1,y_2]\right)}}\left|\delta^{(2)}_{t,s+\eta}(\xi)-\delta^{(2)}_{t,s}(\xi)\right|\leq \varepsilon_1(\lambda).
\]
\end{enumerate}
\end{lemma}
The next lemma deals with the asymptotic behavior of $L^{h}_{t,s}$ as $h$ goes to zero.
\begin{lemma}\label{lem:Lh}
The next two properties hold:
\begin{enumerate}
\item Let $h>0$, $t\in [0, T_1]$ and $\tau\in[0,T_1-t]$ be a $(\cF^{\rho,J}_s)_{s\geq 0}$-stopping time. For any $p\in(1,\infty)$, there exists a constant $C_{p,\cR}>0$ (depending only on $p$ and $\cR$) such that
\[
\left\| \1_{\{J_\tau\leq h\}}\left(L^{h}_{t,\tau}-L_{t,\tau}\right)\right\|_p\leq C_{p,\cR} h.
\]
\item Let $t\in [0,T_1]$ and let $\tau_1$ and $\tau_2$ with $0\leq \tau_1\leq \tau_2\leq T_1-t$ be $(\cF^{\rho,J}_s)_{s\geq 0}$-stopping times.
There exists a constant $C_{\cR}>0$ (depending only on $\cR$) such that
\[
\E\Big[\1_{\{J_{\tau_2}\leq h\}}\big| L_{t,\tau_2}-L_{t,\tau_1} \big|\Big| \cF^{\rho,J}_{\tau_1}\Big]\leq C_{\cR}L_{t,\tau_1}\left(h+\big(\E\big[(\tau_2-\tau_1)\big| \cF^{\rho,J}_{\tau_1}\big]\big)^{1/2}+\big(\E\big[(\tau_2-\tau_1)^2\big| \cF^{\rho,J}_{\tau_1}\big]\big)^{1/2}\right).
\]
\end{enumerate}
\end{lemma}
\begin{proof}
In this proof the constants $C_\cR>0$ and $C_{p,\cR}>0$ may vary from line to line but they are independent of $t$ and $h$. We start by proving the first claim.

{\bf Proof of 1}. Using the inequality $\left| \e^x-\e^y\right|\leq \left(\e^x+\e^y\right)\left|x-y\right|$, we have
\[
\big|L^h_{t,\tau}-L_{t,\tau}\big|\leq \left(L^h_{t,\tau}+L_{t,\tau}\right)\left(I^{(1)}(\tau)+\frac{1}{2}I^{(2)}(\tau) \right),
\]
with 
\[
I^{(1)}(\tau)=\left|\int_0^\tau \gamma_{t,v}\left(h+\rho_v-2J_v\right)\ud\left(\rho_v-2J_v\right)-\int_0^\tau\gamma_{t,v}(\rho_v)\ud\rho_v\right|
\]
and
\[
I^{(2)}(\tau)=\int_0^\tau \left| \big(\gamma_{t,v}\left(h+\rho_v-2J_v\right)\big)^2-\big(\gamma_{t,v}(\rho_v)\big)^2\right|\ud v.
\]
Then, by Cauchy-Schwarz inequality and Lemma \ref{lem:expbd} we have, for $p>1$ and a constant $C_{p,\cR}>0$
\begin{align}\label{eq:I12}
\Big\|\1_{\{J_\tau\le h\}}\big(L^h_{t,\tau}-L_{t,\tau}\big)\Big\|_p\le C_{p,\cR} \Big(\big\|\1_{\{J_\tau\le h\}}I^{(1)}(\tau)\big\|_{2p}+\big\|\1_{\{J_\tau\le h\}}I^{(2)}(\tau)\big\|_{2p}\Big).
\end{align}
Now we estimate the two terms on the right-hand side of the expression above.

On $\{J_\tau\leq h\}$ we may replace $J$ with $J^h= J\wedge h$. Using that $|\gamma(t,y)|\le C_\cR$, for some $C_\cR>0$ (cf.\ Remark \ref{rem:gammaR}), we have
\[
\1_{\{J_{\tau}\leq h\}}I^{(1)}(\tau)\leq \left|\int_0^\tau \gamma_{t,v}\Big(\big(h+\rho_v-2J^h_v\big)-\gamma_{t,v}(\rho_v)\Big)\ud \big(\rho_v-2J_v\big)\right|
+2C_{\cR}h.
\]
Recall that $X=2J-\rho$ is a Brownian motion. Then Burkholder-Davis-Gundy inequalities and the Lipschitz continuity of $y\mapsto \gamma(t,y)$ yield
\begin{align*}
\Big\| \1_{\{J_\tau\leq h\}}I^{(1)}(\tau)\Big\|_{2p}&\leq C_{p,\cR}\Big\|\Big(\int_0^\tau \Big(
       \gamma_{t,v}\big(h+\rho_v-2J^h_v\big)-\gamma_{t,v}(\rho_v)\Big)^2\ud v\Big)^{1/2}\Big\|_{2p}+2C_{\cR}h\\
       &\leq C_{p,\cR}h\|\sqrt{\tau}\|_{2p}+2C_{\cR}h.
\end{align*}
We also have
\begin{align*}
\1_{\{J_{\tau}\leq h\}}I^{(2)}(\tau)& \leq \int_0^\tau \left| \gamma_{t,v}\big(h\!+\!\rho_v\!-\!2J^h_v\big)\!-\!\gamma_{t,v}(\rho_v)\right|\cdot
      \left| \gamma_{t,v}\big(h\!+\!\rho_v\!-\!2J^h_v\big)\!+\!\gamma_{t,v}\rho_v)\right|\ud v \leq 2C^2_{\cR}h\tau.
\end{align*}
Since $\tau\in[0,T_1]$, combining the two bounds above with \eqref{eq:I12} completes the proof.
\medskip

{\bf Proof of 2}. Consider stopping times $\tau_1$, $\tau_2$, with $0\leq \tau_1\leq \tau_2\leq T_1-t$. Noticing that 
\begin{align*}
\Big|\int_{\tau_1}^{\tau_2}\gamma_{t,v}(\rho_v)\ud \rho_v\Big|\le \Big|\int_{\tau_1}^{\tau_2}\gamma_{t,v}(\rho_v)\ud (\rho_v-2J_v)\Big|+2\int_{\tau_1}^{\tau_2}\big|\gamma_{t,v}(\rho_v)\big|\ud J_v,
\end{align*}
we easily deduce 
\[
\1_{\{J_{\tau_2}\leq h\}}\left| L_{t,\tau_2}-L_{t,\tau_1}\right|\leq \left(L_{t,\tau_1}+L_{t,\tau_2}\right)\left(I^{(1)}(\tau_1,\tau_2)+\frac{1}{2}I^{(2)}(\tau_1,\tau_2)\right),
\]
with
\begin{align*}
I^{(1)}(\tau_1,\tau_2)&=\left| \int_{\tau_1}^{\tau_2}\gamma_{t,v}(\rho_v)\ud \left(\rho_v-2J_v\right)\right|
+2C_{\cR}h,\\
I^{(2)}(\tau_1,\tau_2)&=\int_{\tau_1}^{\tau_2}\left(\gamma_{t,v}(\rho_v)\right)^2\ud v \leq C_{\cR}^2\left(\tau_2-\tau_1\right).
\end{align*}
Here, we have used $\int_{\tau_1}^{\tau_2}|\gamma_{t,v}(\rho_v)|\ud J_v\leq C_{\cR}h$ on $\{J_{\tau_2}\leq h\}$. Since $X=\rho-2J$ is a Brownian motion with respect to the filtration $(\cF^{\rho,J}_t)_{t\ge 0}$, by It\^o isometry we obtain
\begin{align*}
\E\Big[\Big(\int_{\tau_1}^{\tau_2}\gamma_{t,v}(\rho_v)\ud \left(\rho_v-2J_v\right)\Big)^2\Big|\cF^{\rho,J}_{\tau_1}\Big]=\E\Big[\int_{\tau_1}^{\tau_2}\big(\gamma_{t,v}(\rho_v)\big)^2\ud v\Big|\cF^{\rho,J}_{\tau_1}\Big]\le C^2_\cR\E\big[(\tau_2-\tau_1)\big|\cF^{\rho,J}_{\tau_1}\big]
\end{align*}
Moreover, Cauchy-Schwarz inequality yields
\begin{align*}
\begin{aligned}
&\E\Big[\1_{\{J_{\tau_2}\leq h\}}\big| L_{t,\tau_2}-L_{t,\tau_1}\big|\Big|\cF^{\rho,J}_{\tau_1}\Big]\\
&\le C_\cR L_{t,\tau_1}\E\Big[\big(1\!+\!L_{t,\tau_2}/L_{t,\tau_1}\big)^2\Big|\cF^{\rho,J}_{\tau_1}\Big]^\frac12\Big(h\!+\!\E\big[\left(\tau_2\!-\!\tau_1\right)\big|\cF^{\rho,J}_{\tau_1}\big]^\frac12\!+\!\E\big[\left(\tau_2\!-\!\tau_1\right)^2\big|\cF^{\rho,J}_{\tau_1}\big]^\frac12\Big).
\end{aligned}
\end{align*}
Since $\E[(L_{t,\tau_2}/L_{t,\tau_1})^2|\cF^{\rho,J}_{\tau_1}]\le C_{\cR}$, thanks to Lemma \ref{lem:expbd}, the proof is complete.
\end{proof}
In the next lemma, we establish  an important intermediate estimate for $V^h_1(t)$ and $V^h_s(t)$
in which we replace all terms depending on $h$, but for the sets $B^h_1$ and $B^h_s$, with their limits as $h$ goes to 0.
The term $V^h_2(t)$ will be treated in a separate statement.
\begin{lemma}\label{lem:barVh}
For $T_2\in[0,T_1)$, we have
\begin{align}\label{eq:barVh.0}
\lim_{h\to 0}\sup_{0\le t\le T_2}h^{-1}\big|V^h_1(t)-\bar V_1^h(t)\big|=0,\quad \lim_{h\to 0}
                                                \sup_{0\le t\le T_2}h^{-1}\int_0^{T_1-t}\big|V^h_s(t)-\bar V^h_s(t)\big|\ud s=0,
\end{align}
with  $\bar V_1^h(t)\coloneqq\E\big[\1_{B_1^h}L_{t,T_1-t}\delta^{(1)}_{t,T_1-t}(\rho)\big]$ and
$\bar V_s^h(t)\coloneqq\E\big[\1_{B^{h}_s} L_{t,s} \delta^{(3)}_{t,s}(\rho)\big]$, for $s\in[0,T_1-t]$.
\end{lemma}
\begin{proof}
We first deal with $V^h_1(t)$. We have
\begin{align*}
V^h_1(t)\!-\!\bar V^h_1(t) =\E\left[\1_{B^h_1}\left(L^h_{t,T_1-t}\delta^{(1)}_{t,T_1-t}(h\!+\!\rho\!-\!2J)\!-\!
      L_{t,T_1-t}\delta^{(1)}_{t,T_1-t}(\rho) \right) \right]=I_1(h)\!+\!I_2(h),
\end{align*}
with
\[
I_1(h)=\E\left[\1_{B^h_1}\left(L^h_{t,T_1-t}-L_{t,T_1-t}\right)\delta^{(1)}_{t,T_1-t}(h+\rho-2J)\right],
\]
and
\[
I_2(h)=\E\left[
     \1_{B^h_1} L_{t,T_1-t}\left(\delta^{(1)}_{t,T_1-t}(h+\rho-2J)-\delta^{(1)}_{t,T_1-t}(\rho) \right) \right].
\]
Recall $B^h_1=\big\{J_{T_1-t}\le h\big\}\cap\big\{\theta_h\ge T_1-t\big\}\subset 
             \big\{J_{T_1-t}\leq h\big\}$.  
Using \eqref{eq:deltabound}, H\"older's inequality and Lemma \ref{lem:Lh}-(1) with $\tau=T_1-t$, we have
\begin{align*}
|I_1(h)|&\leq\E\left[\1_{B^h_1}\left|L^h_{t,T_1-t}-L_{t,T_1-t}\right|\cdot\left|\delta^{(1)}_{t,T_1-t}(h+\rho-2J)\right|\right]\\
       &\leq M_\delta\big(\P(B^h_1)\big)^{1-1/p}\left\| \1_{\{J_{T_1-t}\leq h\}}\big|L^h_{t,T_1-t}-L_{t,T_1-t}\big|\right\|_p\\
       &\leq M_\delta C_{p,\cR}h\left(\P\left(J_{T_1-t}\leq h\right)\right)^{1-1/p}\leq h M_\delta C_{p,\cR}\left(\P\left(J_{T_1-T_2}\leq h\right)\right)^{1-1/p},
\end{align*}
where the final inequality holds because $s\mapsto J_s$ is non-decreasing.

For the remaining term, using Lemma \ref{lem:delta}, tower property and \eqref{eq:unifJ} we obtain
\begin{align*}
|I_2(h)|&\leq \E\Big[
      L_{t,T_1-t}\big|\delta^{(1)}_{t,T_1-t}(h\!+\!\rho\!-\!2J)\!-\!\delta^{(1)}_{t,T_1-t}(\rho) \big|\1_{\{J_{T_1-t}\le h\}} \Big]\leq \varepsilon_0(h)\E\big[\1_{\{J_{T_1-t}\leq h\}}L_{t,T_1-t}\big]\\
      &=\varepsilon_0(h)\E\big[\P\big(J_{T_1-t}\leq h\big|\cF^\rho_{T_1-t}\big)L_{t,T_1-t}\big]\le\varepsilon_0(h)h\E\big[(\rho_{T_1-t})^{-1}L_{t,T_1-t}\big].
\end{align*}
Using H\"older inequality and the scaling in law $\rho_{T_1-t}=\sqrt{T_1-t}\rho_1$, we have for $p\in(1,3)$
\[
\E\Big[\frac{L_{t,T_1-t}}{\rho_{T_1-t}}\Big]\leq \left\|\frac{1}{\rho_{T_1-t}}\right\|_p\big\| L_{t,T_1-t}\big\|_{\frac{p}{p-1}}
\leq \frac{1}{\sqrt{T_1-T_2}}\left\|\frac{1}{\rho_1}\right\|_p\big\| L_{t,T_1-t}\big\|_{\frac{p}{p-1}}\le \frac{C_{p,\cR}}{\sqrt{T_1-T_2}}.
\]
for a constant $C_{p,\cR}>0$ and where the final inequality follows by \eqref{eq:rhoint} and Lemma \ref{lem:expbd}.

Combining the bounds on $I_1(h)$ and $I_2(h)$ we obtain for some other constant $C_{p,\cR}>0$
\[
\sup_{0\le t\le T_2}h^{-1}\big|V^h_1(t)-\bar V_1^h(t)\big|\leq C_{p,\cR}\Big(
     \big(\P(J_{T_1-T_2}\leq h)\big)^{1-1/p}+\frac{\varepsilon_0(h)}{\sqrt{T_1-T_2}}\Big).
\]
Since \eqref{eq:unifJ} and \eqref{eq:rhoint} imply $\P(J_{T_1-T_2}\leq h)\le h\E[(\rho_{T_1-T_2})^{-1}]\to 0$ as $h\to 0$, the first claim in the lemma holds.
\medskip

Next, we use the same approach to study $V^h_s(t)$. Since $B^h_s\subset \{J_s\leq h\}$, we have
\begin{align*}
\big| V^h_s(t)-\bar V^h_s(t)\big|&\leq
M_\delta\E\Big[\1_{B^h_s}\big| L^h_{t,s}-L_{t,s}\big|\Big]
+\E\Big[\1_{B^h_s}L_{t,s}\big|\delta^{(3)}_{t,s}(h+\rho-2J)- \delta^{(3)}_{t,s}(\rho)\big|\Big]\\
&\leq
      M_\delta \E\Big[\1_{\{J_s\leq h\}}\big| L^h_{t,s}-L_{t,s}\big|\Big]
+\varepsilon_0(h)\E\big[\1_{\{J_s\leq h\}}L_{t,s}\big] \\
&\leq
  M_\delta\big(\P(J_s\leq h)\big)^{1-1/p}\left\| \1_{\{J_s\le h\}}\big(L^h_{t,s}-L_{t,s}\big)\right\|_p+
     \varepsilon_0(h)h\E\Big[(\rho_s)^{-1}L_{t,s}\Big],
\end{align*}
where in the final inequality we used H\"older's inequality, tower property with $\cF^\rho_s$ and \eqref{eq:rhoint}.
Noticing that $\P(J_s\leq h)\leq h\E[\rho_s^{-1}]=h s^{-1/2}\E[\rho_1^{-1}]$ and $\E[L_{t,s}/\rho_s]\leq s^{-1/2}\|L_{t,s}\|_{(p-1)/p}\|\rho_1^{-1}\|_p$, for $p\in(1,3)$, the second claim in the lemma follows easily by Lemma \ref{lem:expbd} and Lemma \ref{lem:Lh}.
\end{proof}

The next lemma gives a delicate preliminary estimate on the behaviour of $V_2^h(t)$ from \eqref{eq:V1h}. 
\begin{lemma}\label{lem:V2h}
For $T_2\in[0,T_1)$, we have
\begin{align}\label{eq:V2h.0}
\lim_{h\downarrow 0}\sup_{0\leq t\leq T_2}h^{-1}\big|V^h_2(t)-\bar V_2^h(t)\big|=0,
\end{align}
with
\[
\bar V_2^h(t)\coloneqq\E\Big[L_{t,\tau^{-h}_\varphi}\delta^{(2)}_{t,\tau^{-h}_\varphi}(\rho)\1_{\{\tau^{-h}_\varphi<T_1-t\}}
\1_{\{J_{\tau^{-h}_\varphi}\le h\}}\Big],
\]
where $\tau_\varphi^{-h}$ is defined in \eqref{eq:tauphi} with $\varphi(s)=\varphi_t(s)\coloneqq c((t+s)\wedge T_1)-y_2$.
\end{lemma}
\begin{proof}
Recall that $B^h_2=\{\theta_h<T_1-t\}\cap\{J_{\theta_h}\le h\}$ and
\[
\theta_h=\inf\{s\in[0,\infty):h+\rho_s-2J_s+c((t+s)\wedge T_1)\ge y_2\}.
\] 
Then, in the notations of Lemma \ref{lem:3dBp2}, $\theta_h=\hat{\tau}^{-h}_\varphi$ with $\varphi(s)=\varphi_t(s)$, and ${\tau}^{-h}_\varphi\leq  \theta_h\le \tau^{-h}_{\varphi}+\sigma^\beta_{2h}$ on the event $\{\tau^{-h}_\varphi<T_1-t\}\cap\{J_{\tau^{-h}_\varphi}\le h\}$. 
Let us introduce an auxiliary function:
\begin{align}
\begin{aligned}
U^h_2(t)\coloneqq\E\big[L_{t,\theta_h}\delta^{(2)}_{t,\theta_h}(\rho)\1_{\{\theta_h<T_1-t\}}\1_{\{J_{\theta_h}\le h\}}\big]=\E\big[L_{t,\theta_h}\delta^{(2)}_{t,\theta_h}(\rho)\1_{B^h_2}\big].
\end{aligned}
\end{align}
Now we argue in two steps.

{\bf Step 1}. Following similar arguments as those in the proof of Lemma \ref{lem:barVh}, we show that 
\begin{align}
\begin{aligned}\label{eq:limUhVh}
\lim_{h\downarrow 0}\sup_{0\le t\le T_2}h^{-1}\big|V^h_2(t)-U^h_2(t)\big|=0.
\end{aligned}
\end{align}
Using triangular inequality, \eqref{eq:deltabound} and Lemma \ref{lem:delta}-(2) we have
\begin{align}\label{eq:VhUh}
\begin{aligned}
&\big|V^h_2(t)-U^h_2(t)\big|\\
&\le \E\big[\1_{B^h_2}\big|L^h_{t,\theta_h}\!-\!L_{t,\theta_h}\big|\cdot\big|\delta^{(2)}_{t,\theta_h}(h\!+\!\rho\!-\!2J)\big|\big]\!+\!\E\big[\1_{B^h_2}L_{t,\theta_h}\big|\delta^{(2)}_{t,\theta_h}(h\!+\!\rho\!-\!2J)\!-\!
  \delta^{(2)}_{t,\theta_h}(\rho)\big|\big]\\
&\le M_\delta\E\big[\1_{B^h_2}\big|L^h_{t,\theta_h}\!-\!L_{t,\theta_h}\big|\big]\!+\!\eps_0(h)\E\big[\1_{B^h_2}L_{t,\theta_h}\big].
\end{aligned}
\end{align}
For the first term on the right-hand side above we use H\"older's inequality with $p\in(1,\infty)$ and Lemma \ref{lem:Lh}-(1) with $\tau=\theta_h\wedge(T_1-t)$, jointly with $\theta_h\geq \tau^{-h}_{\varphi}$ and $J_{\theta_h}\ge J_{\tau^{-h}_\varphi}$ to obtain
\begin{align*}
 \E\big[\1_{B^h_2}\big|L^h_{t,\theta_h}-L_{t,\theta_h}\big|\big]&\le C_{p,\cR}\left\|\1_{\{J_{\theta_h}\le h\}}\big(L^h_{t,\theta_h}-L_{t,\theta_h}\big)\right\|_p\P\big(B^h_2\big)^{1-1/p}\\
&\leq C_{p,\cR}h\left(\P\big( \tau^{-h}_{\varphi}<T_1\!-\!t,J_{ \tau^{-h}_{\varphi}}\leq h\big)\right)^{1-1/p}\\
&\leq C_{p,\cR}h\Big(\E\big[\1_{\{\tau^{-h}_{\varphi}<T_1-t\}}\E[\1_{\{J_{\tau^{-h}_\varphi}\le h\}}|\cF^\rho_{\tau^{-h}_\varphi}]\big]\Big)^{1-1/p}\\ 
&\leq C_{p,\cR}h^{2-\frac1p}\E\big[\1_{\{\tau^{-h}_{\varphi}<T_1-t\}}(\rho_{ \tau^{-h}_{\varphi}})^{-1}\big]^{1-1/p},
\end{align*}
where $C_{p,\cR}>0$ is a constant and in the final inequality we used also \eqref{eq:unifJ}. It is worth noticing that on $\{\tau^{-h}_{\varphi}<T_1-t\}$ we have $\rho_{\tau^{-h}_\varphi}=-\varphi(\tau^{-h}_\varphi)-h$. Since $\varphi(s)=c((t+s)\wedge T_1)-y_2<0$ for all $s\ge 0$ and we can take $h>0$ arbitrarily small, then there is no loss of generality in assuming that 
\begin{align}\label{eq:boundrho}
\rho_{\tau^{-h}_\varphi}\ge \bar c>0,\quad\text{on $\{\tau^{-h}_\varphi<T_1-t\}$}, 
\end{align}
for some constant $\bar c>0$ independent of $h$ and $t$ (cf.\ Remark \ref{rem:welldef}). Thus 
\[
\E\big[\1_{B^h_2}\big|L^h_{t,\theta_h}-L_{t,\theta_h}\big|\big]\le (C_{p,\cR}/\bar c) h^{2-\frac1p}.
\]
For the second term on the right-hand side of \eqref{eq:VhUh}, recalling $\theta_h\ge \tau^{-h}_\varphi$ we obtain
\begin{align*}
\begin{aligned}
\E\big[\1_{B^h_2}L_{t,\theta_h}\big]&\le \E\big[\1_{\{\tau^{-h}_\varphi<T_1-t\}\cap\{J_{\tau^{-h}_\varphi}\le h\}}L_{t,\theta_h\wedge(T_1-t)}\big]\\
&\le\E\bigg[\1_{\{\tau^{-h}_\varphi<T_1-t\}\cap\{J_{\tau^{-h}_\varphi}\le h\}}L_{t,\tau^{-h}_\varphi}\E\Big[\exp\Big(\int_{\tau^{-h}_\varphi}^{\theta_h\wedge(T_1-t)}\gamma_{t,v}(\rho_v)\ud \rho_v\Big)\Big|\cF^{\rho,J}_{\tau^{-h}_\varphi}\Big]\bigg]\\
&\le C_\cR\E\big[\1_{\{\tau^{-h}_\varphi<T_1-t\}\cap\{J_{\tau^{-h}_\varphi}\le h\}}L_{t,\tau^{-h}_\varphi}\big]\\
&= C_\cR\E\big[\1_{\{\tau^{-h}_\varphi<T_1-t\}}L_{t,\tau^{-h}_\varphi}\P(J_{\tau^{-h}_\varphi}\le h|\cF^\rho_{\tau^{-h}_\varphi})\big]\\
&\le C_\cR h\E\big[\1_{\{\tau^{-h}_\varphi<T_1-t\}}L_{t,\tau^{-h}_\varphi}(\rho_{\tau^{-h}_\varphi})^{-1}\big]\le (C_\cR/\bar c) h \E\big[L_{t,\tau^{-h}_\varphi}\big],
\end{aligned}
\end{align*}
where we used Lemma \ref{lem:expbd} for the second inequality, \eqref{eq:unifJ} for the third one and \eqref{eq:boundrho} for the final one.
Combining the above estimates we obtain
\[
\big|V^h_2(t)-U^h_2(t)\big|\le C_{p,\cR}h\big(h^{1-\frac1p}+\eps_0(h)),
\]
for some other constant that we denote again by $C_{p,\cR}>0$. That proves \eqref{eq:limUhVh}.
\medskip

{\bf Step 2}. Now we show that 
\begin{align}\label{eq:limUhbarVh}
\begin{aligned}
\lim_{h\to 0}\sup_{0\le t\le T_2}h^{-1}\big|U^h_2(t)-\bar V^h_2(t)\big|=0.
\end{aligned}
\end{align}
We have
\begin{align}\label{UhVhbarbis}
\begin{aligned}
U^h_2(t)-\bar V^h_2(t)=\E\Big[L_{t,\theta_h}\delta^{(2)}_{t,\theta_h}(\rho)
                            \1_{\{\theta_h<T_1-t\}}\1_{\{J_{\theta_h}\le h\}}\Big]
                           - \E\Big[L_{t,\tau^{-h}_\varphi}\delta^{(2)}_{t,\tau^{-h}_\varphi}(\rho)
                            \1_{\{\tau^{-h}_\varphi<T_1-t\}}\1_{\{J_{\tau^{-h}_\varphi}\le h\}}  \Big],
\end{aligned}
\end{align}
and using triangular inequality we obtain
\begin{align}\label{eq:UhVhbar}
\begin{aligned}
\big|U^h_2(t)-\bar V^h_2(t)\big|
 \le&\E\left[\left|L_{t,\theta_h}\delta^{(2)}_{t,\theta_h}(\rho)-L_{t,\tau^{-h}_\varphi}\delta^{(2)}_{t,\tau^{-h}_\varphi}(\rho)\right|\1_{\{\theta_h<T_1-t\}\cap\{J_{\theta_h}\le h\}}\right]\\
&+\E\left[L_{t,\tau^{-h}_\varphi}\big|\delta^{(2)}_{t,\tau^{-h}_\varphi}(\rho)\big|\left|\1_{\{\theta_h<T_1-t\}\cap\{J_{\theta_h}\le h\}}-\1_{\{\tau^{-h}_\varphi<T_1-t\}\cap\{J_{\tau^{-h}_\varphi}\le h\}} \right|\right]\\
&\leq M_\delta\E\left[\left|L_{t,\theta_h}-L_{t,\tau^{-h}_\varphi}\right|\1_{\{\theta_h<T_1-t\}\cap\{J_{\theta_h}\le h\}}\right]\\
&\quad +\E\left[L_{t,\tau^{-h}_\varphi}\left|\delta^{(2)}_{t,\theta_h}(\rho)-\delta^{(2)}_{t,\tau^{-h}_\varphi}(\rho) \right| \1_{\{\theta_h<T_1-t\}\cap\{J_{\theta_h}\le h\}}\right]\\
&\quad+M_\delta \E\left[L_{t,\tau^{-h}_\varphi}\left|\1_{\{\theta_h<T_1-t\}\cap\{J_{\theta_h}\le h\}}-\1_{\{\tau^{-h}_\varphi<T_1-t\}\cap\{J_{\tau^{-h}_\varphi}\le h\}} \right|\right].
\end{aligned}
\end{align}

We study each term individually.
Noticing that $\{J_{\theta_h}\le h\}\subset\{J_{\tau^{-h}_\varphi}\le h\}$ and $\tau^{-h}_\varphi\le \theta_h$ we have
\begin{align*}
\begin{aligned}
&\E\left[\left|L_{t,\theta_h}-L_{t,\tau^{-h}_\varphi}\right|\1_{\{\theta_h<T_1-t\}\cap\{J_{\theta_h}\le h\}}\right]\\
&\leq
\E\left[\left|L_{t,\theta_h\wedge(T_1-t)}-L_{t,\tau^{-h}_\varphi}\right|\1_{\{\tau^{-h}_\varphi<T_1-t\}\cap\{J_{\theta_h\wedge(T_1-t)}\leq h\}}\right]\\
&=
\E\left[\E\Big[\big|L_{t,\theta_h\wedge(T_1-t)}-L_{t,\tau^{-h}_\varphi}\big|\1_{\{J_{\theta_h\wedge(T_1-t)}\leq h\}}\Big|\cF^{\rho,J}_{\tau^{-h}_\varphi}\Big]\1_{\{\tau^{-h}_\varphi<T_1-t\}\cap\{J_{\tau^{-h}_\varphi}\le h\}}\right].
\end{aligned}
\end{align*}
To simplify the notation we set $\tilde \theta_h=\theta_h\wedge(T_1-t)$ and $\tilde \tau^{-h}_\varphi=\tau^{-h}_\varphi\wedge(T_1-t)$. Then, using Lemma \ref{lem:Lh}-(2) with $\tau_1=\tilde\tau^{-h}_\varphi\le \tilde \theta_h=\tau_2$ and Lemma \ref{lem:3dBp2} we obtain
\begin{align*}
\begin{aligned}
&\E\left[\E\Big[\big|L_{t,\tilde\theta_h}-L_{t,\tau^{-h}_\varphi}\big|\1_{\{J_{\tilde\theta_h}\leq h\}}\Big|\cF^{\rho,J}_{\tau^{-h}_\varphi}\Big]\1_{\{\tau^{-h}_\varphi<T_1-t\}\cap\{J_{\tau^{-h}_\varphi}\le h\}}\right]\\
&\leq C_{\cR} \E\left[L_{t,\tau^{-h}_\varphi}\Big(h\!+\!\E\left[\big(\tilde\theta_h\!-\!\tilde\tau^{-h}_\varphi \big)\big| \cF^{\rho,J}_{\tau^{-h}_\varphi}\right]^{\frac12}\!+\!\E\left[\big(\tilde\theta_h\!-\!\tilde\tau^{-h}_\varphi \big)^2\big| \cF^{\rho,J}_{\tau^{-h}_\varphi}\right]^{\frac12}\Big)\1_{\{\tau^{-h}_\varphi<T_1-t\}\cap\{J_{\tau^{-h}_\varphi}\leq h\}}\right]\\
&\leq C_{\cR} \E\left[L_{t,\tau^{-h}_\varphi}\Big(h\!+\!\E\left[\sigma^\beta_{2h}\wedge T_1\big| \cF^{\rho,J}_{\tau^{-h}_\varphi}\right]^{\frac12}\!+\!\E\left[\big(\sigma^\beta_{2h}\wedge T_1\big)^2\big| \cF^{\rho,J}_{\tau^{-h}_\varphi}\right]^{\frac12}\Big)\1_{\{\tau^{-h}_\varphi<T_1-t\}\cap\{J_{\tau^{-h}_\varphi}\leq h\}}\right]\\
&=C_{\cR}\Big(h\!+\!\E\left[\sigma^\beta_{2h}\wedge T_1\right]^{\frac12}\!+\!\E\left[\big(\sigma^\beta_{2h}\wedge T_1\big)^2\right]^{\frac12}\Big) \E\Big[L_{t,\tau^{-h}_\varphi}\1_{\{\tau^{-h}_\varphi<T_1-t\}\cap\{J_{\tau^{-h}_\varphi}\leq h\}}\Big],
\end{aligned}
\end{align*}
where the equality holds by independence of $\sigma^\beta_{2h}$ from $\cF^{\rho,J}_{\tau^{-h}_\varphi}$. Using now \eqref{eq:unifJ}, \eqref{eq:boundrho} and Lemma \ref{lem:expbd} we also have
\begin{align*}
\E\Big[L_{t,\tau^{-h}_\varphi}\1_{\{\tau^{-h}_\varphi<T_1-t\}\cap\{J_{\tau^{-h}_\varphi}\leq h\}}\Big]&=\E\Big[L_{t,\tau^{-h}_\varphi}\1_{\{\tau^{-h}_\varphi<T_1-t\}}\P\big(J_{\tau^{-h}_\varphi}\leq h\big|\cF^\rho_{\tau^{-h}_\varphi}\big)\Big]\\
&\le h\E\big[L_{t,\tau^{-h}_\varphi}(\rho_{\tau^{-h}_\varphi})^{-1}\1_{\{\tau^{-h}_\varphi<T_1-t\}}\big]\le (C_\cR/\bar c)h,
\end{align*}
so that we conclude
\begin{align}\label{eq:line1}
\begin{aligned}
\qquad\E\left[\left|L_{t,\theta_h}\!-\!L_{t,\tau^{-h}_\varphi}\right|\1_{\{\theta_h<T_1-t\}\cap\{J_{\theta_h}\le h\}}\right]\le (C^2_{\cR}/\bar c)h \Big(h\!+\!\E\left[\sigma^\beta_{2h}\wedge T_1\right]^{\frac12}\!+\!\E\left[\big(\sigma^\beta_{2h}\wedge T_1\big)^2\right]^{\frac12}\Big).
\end{aligned}
\end{align}
Using Lemma \ref{lem:delta}, we have by similar arguments
\begin{align}\label{eq:line2}
\begin{aligned}
&\E\left[L_{t,\tau^{-h}_\varphi}\left|\delta^{(2)}_{t,\theta_h}(\rho)\!-\!\delta^{(2)}_{t,\tau^{-h}_\varphi}(\rho) \right| \1_{\{\theta_h<T_1-t\}\cap\{J_{\theta_h}\le h\}}\right]\\
&\leq \E\left[L_{t,\tau^{-h}_\varphi}\varepsilon_1\big(\tilde\theta_h-\tau^{-h}_\varphi\big)\1_{\{\tau^{-h}_\varphi<T_1-t\}}\1_{\{J_{\tau^{-h}_\varphi}<T_1-t\}}\right]\\
&\leq \E\left[L_{t,\tau^{-h}_\varphi}\varepsilon_1\big(\sigma^{\beta}_{2h}\wedge (T_1-t)\big) \1_{\{\tau^{-h}_\varphi<T_1-t\}}\1_{\{J_{\tau^{-h}_\varphi}<T_1-t\}}\right]\leq (C_{\cR}/\bar c) h\E\left[\varepsilon_1(\sigma^{\beta}_{2h}\wedge T_1)\right],
\end{aligned}
\end{align}
for the second term on the right-hand side of \eqref{eq:UhVhbar}.

Finally, we look at the third term on the right-hand side of \eqref{eq:UhVhbar}.
Since $\tau^{-h}_\varphi\leq \theta_h$, then $B^h_2\subset\{\tau^{-h}_\varphi<T_1-t\}\cap\{J_{\tau^{-h}_\varphi}\leq h \}$ and
\begin{align*}
\1_{\{\tau^{-h}_\varphi<T_1-t\}\cap\{J_{\tau^{-h}_\varphi}\leq h \}} &=  \1_{B^h_2}+\1_{\{\tau^{-h}_\varphi<T_1-t\}\cap\{J_{\tau^{-h}_\varphi}\leq h \}\cap (B^h_2)^c}\\
&=\1_{B^h_2}+\1_{\{\tau^{-h}_\varphi<T_1-t\}\cap\{J_{\tau^{-h}_\varphi}\leq h \}}\left(\1_{\{\theta_h\geq T_1-t\}}+\1_{\{\theta_h< T_1-t\}\cap\{J_{\theta_h}>h\}}\right).
\end{align*}
It follows that the last term in \eqref{eq:UhVhbar} reads
\begin{align*}
\E\left[L_{t,\tau^{-h}_\varphi}\left|\1_{B^h_2}-\1_{\{\tau^{-h}_\varphi<T_1-t\}\cap\{J_{\tau^{-h}_\varphi}\le h\}} \right|\right]&=\bar I_1(t,h)+\bar I_2(t,h),
\end{align*}
with
\[
\bar I_1(t,h)=\E\left[L_{t,\tau^{-h}_\varphi}\1_{\{\tau^{-h}_\varphi<T_1-t\}\cap\{J_{\tau^{-h}_\varphi}\leq h \}}\1_{\{\theta_h\geq T_1-t\}}\right]
\]
and
\[
\bar I_2(t,h)=\E\left[L_{t,\tau^{-h}_\varphi}\1_{\{\tau^{-h}_\varphi<T_1-t\}\cap\{J_{\tau^{-h}_\varphi}\leq h \}}\1_{\{\theta_h< T_1-t\}\cap\{J_{\theta_h}>h\}} \right].
\]
We first study $\bar I_1(t,h)$. For fixed $\delta\in (0, (T_1-T_2)/2)$ we have
\begin{align*}
\bar I_1(t,h)&=\E\left[L_{t,\tau^{-h}_\varphi}\1_{\{\tau^{-h}_\varphi<T_1-t-\delta\}\cap\{J_{\tau^{-h}_\varphi}\leq h \}}
\1_{\{\theta_h\geq  T_1-t\}}\right]\\
&\quad +\E\left[L_{t,\tau^{-h}_\varphi}\1_{\{T_1-t-\delta\leq\tau^{-h}_\varphi<T_1-t\}\cap\{J_{\tau^{-h}_\varphi}\leq h \}}
\1_{\{\theta_h\geq T_1-t\}}\right]
\end{align*}
On the set $\{\tau^{-h}_\varphi<T_1-t-\delta\}\cap\{\theta_h\geq T_1-t\}$, we have $\theta_h-\tau^{-h}_\varphi\geq \delta$.
By Lemma \ref{lem:3dBp2}, $\theta_h=\hat\tau^{-h}_{\varphi}\leq \tau^{-h}_\varphi+\sigma^\beta_{2h}$
on the event $\{J_{\tau^{-h}_\varphi}\leq h\}\cap\{\tau^{-h}_\varphi<T_1-t\}$. Then, by independence of $\beta$ and $\cF^{\rho,J}_{\tau^{-h}_\varphi}$,
\begin{align}\label{eq:I1-1}
\begin{aligned}
&\E\left[L_{t,\tau^{-h}_\varphi}\1_{\{\tau^{-h}_\varphi<T_1-t-\delta\}\cap\{J_{\tau^{-h}_\varphi}\leq h \}}\1_{\{\theta_h\geq T_1-t\}}\right]\\
&\leq\E\left[L_{t,\tau^{-h}_\varphi}\1_{\{\tau^{-h}_\varphi<T_1-t-\delta\}\cap\{J_{\tau^{-h}_\varphi}\leq h \}}\1_{\{\sigma^\beta_{2h}\geq \delta\}}   \right]\\
&=\E\left[L_{t,\tau^{-h}_\varphi}\1_{\{\tau^{-h}_\varphi<T_1-t-\delta\}}\P\left(J_{\tau^{-h}_\varphi}\leq h \big|\cF^{\rho}_{\tau^{-h}_\varphi}\right)\right]\P\left(\sigma^\beta_{2h}\geq \delta\right)\\
&\leq \frac{h}{\bar c}\E\left[L_{t,\tau^{-h}_\varphi}\1_{\{\tau^{-h}_\varphi<T_1-t-\delta\}}\right]\P\left(\sigma^\beta_{2h}\geq \delta\right)\leq(C_{\cR}/\bar c) h\P\left(\sigma^\beta_{2h}\geq \delta\right),
\end{aligned}
\end{align}
where we used \eqref{eq:boundrho} in the second inequality and Lemma \ref{lem:expbd} for the final one.
Recalling that $\delta\in(0,[T_1-T_2]/2)$, for the second term in the expression of $\bar I_1(t,h)$ we have, for $t\in[0,T_2]$,
\begin{align}\label{eq:I1-2}
\begin{aligned}
&\E\left[L_{t,\tau^{-h}_\varphi}\1_{\{T_1-t-\delta\leq\tau^{-h}_\varphi<T_1-t\}\cap\{J_{\tau^{-h}_\varphi}\leq h \}}
\1_{\{\theta_h\geq T_1-t\}}\right]\\
&\leq\E\left[L_{t,\tau^{-h}_\varphi}\1_{\{T_1-t-\delta\leq\tau^{-h}_\varphi<T_1-t\}}\P\big(J_{\tau^{-h}_\varphi}\leq h \big|\cF^\rho_{\tau^{-h}_\varphi}\big)\right]\leq\frac{h}{\bar c}\E\left[L_{t,\tau^{-h}_\varphi}\1_{\{T_1-t-\delta\leq\tau^{-h}_\varphi<T_1-t\}}\right]\\
&\leq\frac{h}{\bar c}C_{p,\cR}\left[\P\left(T_1-t-\delta\leq\tau^{-h}_\varphi<T_1-t \right)\right]^{1/p}\\
&\leq h\frac{C_{p,\cR,\|\dot \varphi\|}}{\bar c}\left(\frac{\sqrt{\delta}}{\sqrt{T_1-t-\delta}}\right)^{1/p}\leq h\frac{C_{p,\cR,\|\dot \varphi\|}}{\bar c}\left(\frac{\sqrt{2\delta}}{\sqrt{T_1-T_2}}\right)^{1/p},
\end{aligned}
\end{align}
where the penultimate inequality is by Lemma \ref{lem:pbdtauh}.
Therefore, combining \eqref{eq:I1-1} and \eqref{eq:I1-2} we have for some $C>0$ independent of $h$ and $t\in[0,T_2]$,
\begin{align}\label{eq:I1}
\big|\bar I_1(t,h)\big|\le C h \Big(\P\big(\sigma^\beta_{2h}\ge \delta\big)+\delta^{\frac{1}{2p}}\Big).
\end{align}

We now examine $\bar I_2(t,h)$. Recall that the process $X$, defined by  $X_s=2J_s-\rho_s$ is a $(\cF^{\rho,J}_s)$-Brownian motion 
and $J_s=\sup_{u\in[0,s]}X_{u}$. 
Therefore, recalling $\theta_h\ge \tau^{-h}_\varphi$ we have 
\[
J_{\theta_h}=\max\Big(J_{\tau^{-h}_\varphi},\sup_{\tau^{-h}_\varphi\le s\le\theta_h}X_s\Big).
\] 
On the event $\{J_{\tau^{-h}_\varphi}\leq h<J_{\theta_h}\}$, 
\[
\sup_{\tau^{-h}_\varphi\le s\le \theta_h}X_s>h\iff\sup_{\tau^{-h}_\varphi\le s \le \theta_h}\left(X_s-X_{\tau^{-h}_\varphi}\right)>h-X_{\tau^{-h}_\varphi}=h+\rho_{\tau^{-h}_\varphi}-2J_{\tau^{-h}_\varphi}.
\] 
On the event $\{\tau^{-h}_\varphi<T_1-t\}\cap\{J_{\tau^{-h}_\varphi}\leq h\}$ using \eqref{eq:boundrho} in the above yields 
\[
\sup_{\tau^{-h}_\varphi\le s\le \theta_h}\left(X_s-X_{\tau^{-h}_\varphi}\right)>\bar c-h\ge \bar c/2,
\]
where the final inequality holds for $h\leq \bar c/2$.
The process $(\psi_s)_{s\ge 0}=(X_{\tau^{-h}_\varphi+s}-X_{\tau^{-h}_\varphi})_{s\ge 0}$ is a Brownian motion independent of $\cF^X_{\tau^{-h}_\varphi}$. Then, setting 
$S_\lambda=\sup_{0\leq s\leq \lambda} \psi_s$, we have for all $\delta>0$
\begin{align*}
\bar I_2(t,h)&\leq \E\left[L_{t,\tau^{-h}_\varphi}\1_{\{\tau^{-h}_\varphi<T_1-t\}\cap\{J_{\tau^{-h}_\varphi}\leq h \}}
\1_{\{ S_{ \theta_h-\tau^{-h}_\varphi}\geq \bar c/2\}}\right]\\
&\leq \E\left[L_{t,\tau^{-h}_\varphi}\1_{\{\tau^{-h}_\varphi<T_1-t\}\cap\{J_{\tau^{-h}_\varphi}\leq h \}}
\left(\1_{\{ S_{ \delta}\geq \bar c/2\}} +\1_{\{\theta_h-\tau^{-h}_\varphi>\delta\}} \right)\right]\\
&\leq\E\left[L_{t,\tau^{-h}_\varphi}\1_{\{\tau^{-h}_\varphi<T_1-t\}\cap\{J_{\tau^{-h}_\varphi}\leq h \}}
\left(\1_{\{ S_{ \delta}\geq \bar c/2\}}+\1_{\{\sigma^\beta_{2h}>\delta\}} \right)\right],
\end{align*}
where we used also $\theta_h\le \tau^{-h}_\varphi+\sigma^\beta_{2 h}$ (cf.\ Lemma \ref{lem:3dBp2}).
Using the independence of the pair $(S,\beta)$ from $\cF^{\rho,J}_{\tau^{-h}_\varphi}$ and \eqref{eq:unifJ} yields
\begin{align}\label{eq:I2}
\begin{aligned}
\bar I_2(t,h)&\leq \E\left[L_{t,\tau^{-h}_\varphi}\1_{\{\tau^{-h}_\varphi<T_1-t\}}\P\left(J_{\tau^{-h}_\varphi}\leq h\big|\cF^{\rho}_{\tau^{-h}_\varphi}\right)\right] \left(\P\left(S_{\delta}\geq \bar c/2 \right)          +
\P\left(\sigma^\beta_{2h}>\delta\right)\right)\\
&\leq\frac{h}{\bar c}\E\left[L_{t,\tau^{-h}_\varphi}\1_{\{\tau^{-h}_\varphi<T_1-t\}}\right]
\left(\P\left(S_{ \delta}\geq \bar c/2 \right)          +
\P\left(\sigma^\beta_{2h}>\delta\right)\right)\\
&\leq (C_{\cR}/\bar c) h\left(\P\big(S_{ \delta}\geq \bar c/2 \big)          +
\P\big(\sigma^\beta_{2h}>\delta\big)\right),
\end{aligned}
\end{align}
for some  constant $C_{\cR}>0$ coming from Lemma \ref{lem:expbd}. 

Therefore, combining \eqref{eq:I1} and \eqref{eq:I2} we obtain
\begin{align}\label{eq:line3}
\begin{aligned}
&\E\left[L_{t,\tau^{-h}_\varphi}\left|\1_{B^h_2}-\1_{\{\tau^{-h}_\varphi<T_1-t\}\cap\{J_{\tau^{-h}_\varphi}\le h\}} \right|\right]\le C h \Big(\P\big(\sigma^\beta_{2h}\ge \delta\big)+\P\big(S_\delta\ge \bar c/2\big)+\delta^{\frac{1}{2p}}\Big),
\end{aligned}
\end{align}
for a constant $C>0$ independent of $h$ and $t\in[0,T_2]$.
Now, combining \eqref{eq:line1}, \eqref{eq:line2} and \eqref{eq:line3}, we obtain 
\begin{align*}
\lim_{h\to 0}\sup_{0\le t\le T_2}h^{-1}\big|U^h_2(t)-\bar V^h_2(t)\big|\le C\Big(\P(S_\delta\ge \bar c/2)+\delta^{\frac{1}{2p}}\Big),
\end{align*}
by using that $\sigma^\beta_{2h}\to 0$ a.s.\ (Remark \ref{rem:sigma}). Since $\delta>0$ was arbitrary, letting $\delta\downarrow 0$ we obtain \eqref{eq:limUhbarVh}.

Finally, combining \eqref{eq:limUhVh} and \eqref{eq:limUhbarVh} we conclude.
\end{proof}

Equipped with the results from Lemmata \ref{lem:barVh} and \ref{lem:V2h}, we now proceed to prove Theorem \ref{thm:expansion}.
\begin{proof}[\bf Proof of Theorem \ref{thm:expansion}]
We adopt the same notations as in the proofs of the previous lemmas. In the estimates below $C_\cR>0$ is a constant that may vary from line to line but it is independent of $h$ and $t$.
Setting
\[
\bar V_h(t)\coloneqq \bar V^h_1(t)+\bar V^h_2(t)+\int_0^{T_1-t}\bar V^h_s(t)\ud s,
\]
thanks to Lemmata \ref{lem:barVh} and  \ref{lem:V2h}  it suffices to show that 
\begin{align}\label{eq:limV}
\begin{aligned}
\lim_{h\downarrow 0}\sup_{0\le t\le T_2}\Big|h^{-1}\bar V_h(t)-V_1(t)-V_2(t)-\int_0^{T_1-t}V_s(t)\ud s\Big|=0.
\end{aligned}
\end{align}

{\bf Step 1}. We first study $h^{-1}\bar V^h_1(t)-V_1(t)$. 
In the notation of Lemma \ref{lem:3dBp2}, we have $\tau^{-h}_\varphi\le \theta_h=\hat{\tau}^{-h}_\varphi\le \tau^{-h}_\varphi+\sigma^\beta_{2h}$, with
$\varphi(s)=\varphi_t(s)=c((t+s)\wedge T_1)-y_2$, and
\begin{align*}
\{J_{T_1-t}\le h\}\cap\{\theta_h\ge T_1-t\}&= \{J_{T_1-t}\le h\}\cap\{\rho_s-2J_s+\varphi(s)< -  h,\forall s\in [0,T_1-t)\}\\
           &\subset  \{J_{T_1-t}\le h\}\cap\{\tau^{+h}_\varphi \ge T_1-t\}.
\end{align*}
Therefore
\[
\{J_{T_1-t}\le h\}\cap\{\tau^{-h}_\varphi\ge T_1-t\} \subset \{J_{T_1-t}\le h\}\cap\{\theta_h\ge T_1-t\}\subset 
                    \{J_{T_1-t}\le h\}\cap\{\tau^{+h}_\varphi\ge T_1-t\}. 
\]
Similarly, using $\{\theta = T_1-t\}=\{\rho_s+\varphi(s)<0,\forall s\in [0,T_1-t)\}$, we also have
\begin{align*}
\{\tau^{-h}_\varphi\ge T_1-t\}\subset 
\{\theta = T_1-t\}\subset\{\tau^{+h}_\varphi\ge T_1-t\}.
\end{align*}
Combining the above set inclusions we deduce
\[
\left|\1_{\{J_{T_1-t}\le h\}\cap\{\theta_h\ge T_1-t\}}-\1_{\{J_{T_1-t}\le h\}\cap\{\theta = T_1-t\}} \right|\leq \1_{\{J_{T_1-t}\le h\}}
    \1_{\{\tau^{-h}_\varphi<T_1-t\leq \tau^{+h}_\varphi\}},
\]
and therefore
\begin{align}\label{eq:in1}
\begin{aligned}
&\left|\bar V^h_1(t)-\E\left[L_{t,T_1-t}\delta^{(1)}_{t,T_1-t}(\rho)\1_{\{J_{T_1-t}<h\}\cap\{\theta= T_1-t\}}\right]\right|\\
&\le
\E\left[\left|L_{t,T_1-t}\delta^{(1)}_{t,T_1-t}(\rho)\right|\P\left(J_{T_1-t}\le h\big|\cF^\rho_{T_1-t}\right)\1_{\{\tau^{-h}_\varphi<T_1-t\leq \tau^{+h}_\varphi\}}\right]\\
&\leq h M_\delta \E\left[\frac{1}{\rho_{T_1-t}}L_{t,T_1-t}\1_{\{\tau^{-h}_\varphi<T_1-t\leq \tau^{+h}_\varphi\}}\right],
\end{aligned}
\end{align}
where we used \eqref{eq:unifJ} and \eqref{eq:deltabound}. 
By analogous arguments and using the equality $\eta\wedge 1=\eta-(\eta-1)^+$ we also obtain
\begin{align}\label{eq:in2}
\begin{aligned}
&\E\left[L_{t,T_1-t}\delta^{(1)}_{t,T_1-t}(\rho)\1_{\{J_{T_1-t}\le h\}\cap\{\theta= T_1-t\}}\right]\\
&=\E\left[L_{t,T_1-t}\delta^{(1)}_{t,T_1-t}(\rho)\P\left(J_{T_1-t}\le h\big|\cF^\rho_{T_1-t}\right)\1_{\{\theta= T_1-t\}}\right]\\
&=\E\left[L_{t,T_1-t}\delta^{(1)}_{t,T_1-t}(\rho)\left(\frac{h}{\rho_{T_1-t}}\wedge 1\right)\1_{\{\theta= T_1-t\}}\right]\\
&=h \E\left[L_{t,T_1-t}\delta^{(1)}_{t,T_1-t}(\rho)\frac{1}{\rho_{T_1-t}}\1_{\{\theta= T_1-t\}}\right]-
\E\left[L_{t,T_1-t}\delta^{(1)}_{t,T_1-t}(\rho)\left(\frac{h}{\rho_{T_1-t}}-1\right)^+\1_{\{\theta= T_1-t\}}\right]\\
&=hV_1(t)-\E\left[L_{t,T_1-t}\delta^{(1)}_{t,T_1-t}(\rho)\left(\frac{h}{\rho_{T_1-t}}-1\right)^+\1_{\{\theta= T_1-t\}}\right].
\end{aligned}
\end{align}

From \eqref{eq:in1} and \eqref{eq:in2} it follows
\begin{align}\label{eq:s2}
h^{-1}\left|\bar V^h_1(t)\!-\! hV_1(t)\right|&\leq M_\delta\Big(\E\left[\frac{L_{T_1-t}}{\rho_{T_1-t}}\1_{\{\tau^{-h}_\varphi<T_1-t\leq \tau^{+h}_\varphi\}}\!+\!\frac{L_{T_1-t}}{\rho_{T_1-t}}\1_{\{\rho_{T_1-t}\le h\}}\1_{\{\theta= T_1-t\}}\right]\Big).
\end{align}
For the second term in the expression above, by H\"older's inequality with $p\in(1,3)$ and $p^{-1}+q^{-1}+1/2=1$, and scaling of the Bessel process, we have
\begin{align}\label{eq:s}
\begin{aligned}
\E\left[\frac{L_{T_1-t}}{\rho_{T_1-t}}\1_{\{\rho_{T_1-t}\le h\}}\1_{\{\theta= T_1-t\}}\right]&\leq 
\frac{1}{\sqrt{T_1-t}}\left\|\frac{1}{\rho_{1}}\right\|_p\| L_{T_1-t}\|_q\P\big(\rho_{T_1-t}\leq h\big)^{\frac12}\\
&\leq
C_\cR\P\left(\rho_{1}\leq \frac{h}{\sqrt{T_1-T_2}}\right)^{\frac12},
\end{aligned}
\end{align}
where Lemma \ref{lem:expbd} yields the second inequality.
For the first term in \eqref{eq:s2} we have
\begin{align*}
&\E\left[\frac{L_{T_1-t}}{\rho_{T_1-t}}\1_{\{\tau^{-h}_\varphi<T_1-t\leq \tau^{+h}_\varphi\}}\right]\\
&\leq \E\left[\frac{L_{T_1-t}}{\rho_{T_1-t}}\1_{\{\tau^{-h}_\varphi<T_1-t-\delta\}\cap\{T_1-t\leq \tau^{+h}_\varphi\}}\right]
+\E\left[\frac{L_{T_1-t}}{\rho_{T_1-t}}\1_{\{T_1-t-\delta\leq \tau^{-h}_\varphi<T_1-t\}}\right]\\
&\leq\E\left[\frac{L_{T_1-t}}{\rho_{T_1-t}}\1_{\{\sigma^{\beta}_{2h}\geq \delta\}}\right]
+\E\left[\frac{L_{T_1-t}}{\rho_{T_1-t}}\1_{\{T_1-t-\delta\leq \tau^{-h}_\varphi<T_1-t\}}\right],
\end{align*}
for any $\delta\in (0,(T_1-T_2)/2]$, where we have used Lemma \ref{lem:3dBp1} for the final inequality. Continuing from the last expression as in \eqref{eq:s}, using H\"older's inequality with $p^{-1}+q^{-1}+1/2=1$ we obtain
\begin{align*}
\E\left[\frac{L_{T_1-t}}{\rho_{T_1-t}}\1_{\{\tau^{-h}_\varphi<T_1-t\leq \tau^{+h}_\varphi\}}\right]
&\leq C_{\cR}\left[\P\left(\sigma^{\beta}_{2h}\geq \delta\right)^{\frac12}+\P\left(T_1-t-\delta\leq \tau^{-h}_\varphi<T_1-t\right)^{\frac12}\right]\\
&\leq C_{\cR}\left[\P\left(\sigma^{\beta}_{2h}\geq \delta\right)^{\frac12}+\left(\frac{\sqrt{2\delta}}{\sqrt{T_1-T_2}}\right)^{\frac12}\right],
\end{align*}
where, for the last inequality, we have used Lemma~\ref{lem:pbdtauh}. Combining the estimates above with \eqref{eq:s2} and letting $h\to 0$ we arrive at
\[
\limsup_{h\downarrow 0}\sup_{0\le t\le T_2}h^{-1}\left|\bar V^h_1(t)- hV_1(t)\right|\le C_\cR\delta^{1/4}.
\]
Since $\delta>0$ was arbitrary, $h^{-1}\left|\bar V^h_1(t)- hV_1(t)\right|\to 0$ as $h\to 0$, uniformly for $t\in[0,T_2]$.
\medskip

{\bf Step 2}. Next we consider the difference $h^{-1}\bar V^h_2(t)-V_2(t)$. By tower property and \eqref{eq:unifJ} we have 
\begin{align*}
\bar V^h_2(t)&=\E\left[L_{t,\tau^{-h}_\varphi}\delta^{(2)}_{t,\tau^{-h}_\varphi}(\rho)\1_{\{\tau^{-h}_\varphi<T_1-t\}}\P\left(J_{\tau^{-h}_\varphi}\le h\big|\cF^\rho_{\tau^{-h}_\varphi}\right)\right] \\
&=\E\left[L_{t,\tau^{-h}_\varphi}\left(\frac{h}{\rho_{\tau^{-h}_\varphi}}\wedge 1\right)\delta^{(2)}_{t,\tau^{-h}_\varphi}(\rho)\1_{\{\tau^{-h}_\varphi<T_1-t\}}\right] \\
&=h \E\left[\frac{L_{t,\tau^{-h}_\varphi}}{\rho_{\tau^{-h}_\varphi}}\delta^{(2)}_{t,\tau^{-h}_\varphi}(\rho)\1_{\{\tau^{-h}_\varphi<T_1-t\}}\right]-\E\left[L_{t,\tau^{-h}_\varphi}\left(\frac{h}{\rho_{\tau^{-h}_\varphi}} -1\right)^+\delta^{(2)}_{t,\tau^{-h}_\varphi}(\rho)\1_{\{\tau^{-h}_\varphi<T_1-t\}}\right],
\end{align*}
where we used $\eta\wedge 1=\eta-(\eta-1)^+$.
Since $\rho_{\tau^{-h}_\varphi}\geq \bar c>0$ by \eqref{eq:boundrho}, the final term in the expression above vanishes for $h\le \bar c$ and  
\begin{align*}
\big|h^{-1}\bar V^h_2(t)-V_2(t)\big|&=\left|\E\left[\frac{L_{t,\tau^{-h}_\varphi}}{\rho_{\tau^{-h}_\varphi}}\delta^{(2)}_{t,\tau^{-h}_\varphi}(\rho)\1_{\{\tau^{-h}_\varphi<T_1-t\}}\right]-\E\left[\frac{L_{t,\theta}}{\rho_\theta}\delta^{(2)}_{t,\theta}(\rho)
                        \1_{\{\theta<T_1-t\}}\right]\right|\\
&\le\E\left[\frac{|L_{t,\theta}-L_{t,\tau^{-h}_\varphi}|}{\rho_\theta}
                  |\delta^{(2)}_{t,\theta}(\rho)|
                        \1_{\{\theta<T_1-t\}}\right]   \\
&\quad+\E\left[L_{t,\tau^{-h}_\varphi}\left|\frac{\delta^{(2)}_{t,\theta}(\rho)}{\rho_\theta}\1_{\{\theta<T_1-t\}}-\frac{ \delta^{(2)}_{t,\tau^{-h}_\varphi}(\rho)}{\rho_{\tau^{-h}_\varphi}}\1_{\{\tau^{-h}_\varphi<T_1-t\}}\right|\right].                   
\end{align*}

On $\{\theta<T_1-t\}$ we have $\theta=\tau^0_\varphi\geq \tau^{-h}_\varphi$, so that
$\1_{\{\theta<T_1-t\}}=\1_{\{\tau^{-h}_\varphi<T_1-t\}}-\1_{\{\tau^{-h}_\varphi<T_1-t\leq \theta \}}$.
Moreover, on $\{\theta<T_1-t\}$ we have $\rho_\theta=y_2-c(t+\theta)\ge \bar c>0$ and, similarly, on $\{\tau^{-h}_\varphi<T_1-t\}$
we have $\rho_{\tau^{-h}_\varphi}=y_2-c(t+\tau^{-h}_\varphi)-h\ge\bar c/2>0$, for sufficiently small $h>0$.
Therefore
\begin{align*}
\left|V_2(t)-h^{-1}\bar V^h_2(t)\right|&\leq (M_\delta/\bar c) \E\left[\big|L_{t,\theta}-L_{t,\tau^{-h}_\varphi}\big|\1_{\{\theta<T_1-t\}}\right]\\
&\quad+\E\left[ L_{t,\tau^{-h}_\varphi}\1_{\{\tau^{-h}_\varphi<T_1-t\}}\left|\frac{\delta^{(2)}_{t,\theta}(\rho)}{y_2-c(t+\theta)}-\frac{ \delta^{(2)}_{t,\tau^{-h}_\varphi}(\rho)}{y_2-c(t+\tau^{-h}_\varphi)-h}\right|\right]\\
&\quad+ \E\left[ L_{t,\tau^{-h}_\varphi}\left|\frac{\delta^{(2)}_{t,\theta}(\rho)}{y_2-c(t+\theta)}\right|\1_{\{\tau^{-h}_\varphi<T_1-t\leq \theta\}}\right]\\
&\le(M_\delta/\bar c)\left(K_{2,1}(t,h)+K_{2,2}(t,h)\right)+K_{2,3}(t,h),
\end{align*}
with 
 \[
 K_{2,1}(t,h)= \E\left[|L_{t,\theta}-L_{t,\tau^{-h}_\varphi}|\1_{\{\theta<T_1-t\}}\right],
 \quad K_{2,2}(t,h)=  \E\left[ L_{t,\tau^{-h}_\varphi}
                      \1_{\{\tau^{-h}_\varphi<T_1-t\leq \theta\}}\right],
 \]
and
   \[
   K_{2,3}(t,h)=\E\left[ L_{t,\tau^{-h}_\varphi}\1_{\{\tau^{-h}_\varphi<T_1-t\}}\left|\frac{\delta^{(2)}_{t,\theta}(\rho)}{y_2-c(t+\theta)}
                      -\frac{ \delta^{(2)}_{t,\tau^{-h}_\varphi}(\rho)}{y_2-c(t+\tau^{-h}_\varphi)-h}\right|\right].
   \]                  
   
For the study of $K_{2,1}(t,h)$, 
we recall that $\tau^{-h}_\varphi<T_1-t$ on $\{\theta<T_1-t\}$. Then, on $\{\theta<T_1-t\}$,
\begin{align*}
\left|L_{t,\theta}-L_{t,\tau^{-h}_\varphi}\right|&\leq \left(L_{t,\theta}+L_{t,\tau^{-h}_\varphi}\right)
\left(\left|\int_{\tau^{-h}_\varphi}^\theta \gamma_{t,v}(\rho_v)\ud \rho_v\right|+\frac{1}{2}\int_{\tau^{-h}_\varphi}^\theta \left(\gamma_{t,v}(\rho_v)\right)^2\ud v\right)\\
&\leq \left(L_{t,\theta}+L_{t,\tau^{-h}_\varphi}\right)\left(\left|\int_{\tau^{-h}_\varphi}^\theta \gamma_{t,v}(\rho_v)\ud \rho_v\right|+\frac{C^2_{\cR}}{2}\left(\theta-\tau^{-h}_\varphi\right)\right).
\end{align*}
It follows by Cauchy-Schwarz inequality
\begin{align}\label{eq:K21}
\begin{aligned}
K_{2,1}(t,h)\le C_\cR &\left\|\1_{\{\theta<T_1-t\}}\big(L_{t,\theta}\!+\!L_{t,\tau^{-h}_\varphi}\big)\right\|_2\cdot\\
&\cdot\bigg(\Big\|\1_{\{\theta<T_1-t\}}\int_{\tau^{-h}_\varphi}^\theta \gamma_{t,v}(\rho_v)\ud \rho_v\Big\|_2+\!\big\|\1_{\{\theta<T_1-t\}}\big(\theta\!-\!\tau^{-h}_\varphi\big)\big\|_2\bigg).
\end{aligned}
\end{align}

Recall that the Bessel process is a strong solution of the equation \eqref{eq:SDErho}, so that
\begin{align*}
\left(\int_{\tau^{-h}_\varphi}^\theta \gamma_{t,v}(\rho_v)\ud \rho_v\right)^2&\leq
2\left(\int_{\tau^{-h}_\varphi}^\theta \gamma_{t,v}(\rho_v)\ud \beta^0_v\right)^2+
2C^2_{\cR}\left(\int_{\tau^{-h}_\varphi}^\theta \frac{\ud s}{ \rho_s}\right)^2\\
&\le 2\left(\int_{\tau^{-h}_\varphi}^\theta \gamma_{t,v}(\rho_v)\ud \beta^0_v\right)^2+
4C^2_{\cR}\left( \left(\rho_{\theta}-\rho_{\tau^{-h}_\varphi}\right)^2+\left(\beta^0_{\theta}-\beta^0_{\tau^{-h}_\varphi}\right)^2\right).
\end{align*}
Setting $\tau^{-h}_\varphi\wedge(T_1-t)=\tau_1$ for the ease of notation and recalling also that $\rho_\tau=|\widehat B_\tau|$ where $\widehat B$ is a 3-dimensional Brownian motion, we have the bound
\begin{align*}
\E\Big[\1_{\{\theta<T_1-t\}}\Big|\int_{\tau^{-h}_\varphi}^\theta \gamma_{t,v}(\rho_v)\ud \rho_v\Big|^2\Big]&\leq
2\E\left[\left(\int_{\tau_1}^\theta \gamma_{t,v}(\rho_v)\ud \beta^0_v\right)^2\right] \\
&\quad+
4C^2_{\cR}\left(\E \left[\left(\rho_{\theta}-\rho_{\tau_1}\right)^2\right]+
\E \left[\left(\beta^0_{\theta}-\beta^0_{\tau_1}\right)^2 \right]\right)\leq 
   C_{\cR}\E\left[\theta-\tau_1 \right],
\end{align*}
where  the last inequality follows from the boundedness of $\gamma$ and $(\rho_{\theta}-\rho_{\tau_1})^2\le |\widehat B_\theta-\widehat B_{\tau_1}|^2$.  Substituting these bounds into \eqref{eq:K21} we obtain
\begin{align}\label{eq:K21.2}
\begin{aligned}
K_{2,1}(t,h)&\leq
  C_{\cR} \left\|\1_{\{\theta<T_1-t\}}\big( L_{t,\theta}+L_{t,\tau_1}\big)\right\|_2\big(\sqrt{\|(\theta-\tau^{-h}_\varphi)^+\|_1}+\|(\theta-\tau^{-h}_\varphi)^+\|_2\big)\\
  &\le C_\cR\big(\sqrt{\|\sigma^\beta_{2h}\wedge T_1\|_1}+\big\|\sigma^\beta_{2h}\wedge T_1\big\|_2\big),
\end{aligned}
\end{align}
where for the second inequality we used Lemma \ref{lem:expbd} and Lemma \ref{lem:3dBp1}, upon noticing that on $\{\theta<T_1-t\}$ we have 
$\theta-\tau_1\leq (\tau^{h}_\varphi-\tau^{-h}_\varphi)\wedge (T_1-t)\leq \sigma^\beta_{2h}\wedge T_1$.

In order to estimate $K_{2,2}(t,h)$, we use Cauchy-Schwarz and Lemma \ref{lem:expbd} to derive
\[
K_{2,2}(t,h)\leq \left\| L_{t,\tau^{-h}_\varphi\wedge (T_1-t)}\right\|_2\sqrt{\P(\tau^{-h}_\varphi<T_1-t\leq \theta)}
   \leq C_{\cR}\sqrt{\P\left(\tau^{-h}_\varphi<T_1-t\leq \theta\right)}.
\]
For any $\delta\in (0,(T_1-T_2)/2]$, we have
\begin{align*}
\begin{aligned}
\P\left(\tau^{-h}_\varphi<T_1-t\leq \theta\right)&\leq \P\left(\tau^{-h}_\varphi<T_1-t-\delta,\ \theta\ge T_1-t\right)+\P\left(T_1-t-\delta\leq \tau^{-h}_\varphi<T_1-t\right)\\
&\leq\P\left(\theta-\tau^{-h}_\varphi>\delta\right)+\P\left(T_1-t-\delta\leq \tau^{-h}_\varphi<T_1-t\right)\\
&\leq\P\left(\sigma^\beta_{2h}>\delta\right)+C\frac{\sqrt{\delta}}{\sqrt{T_1-t-\delta}}\leq \P\left(\sigma^\beta_{2h}>\delta\right)+C\frac{\sqrt{2\delta}}{\sqrt{T_1-T_2}},
\end{aligned}
\end{align*}
thanks to Lemma \ref{lem:3dBp1} and Lemma \ref{lem:pbdtauh}. Thus
\begin{align}\label{eq:K22.2}
K_{2,2}(t,h)\le C_\cR\sqrt{\P\left(\sigma^\beta_{2h}\ge\delta\right)+\sqrt{2\delta}}.
\end{align}

Concerning $K_{2,3}(t,h)$, we have 
\begin{align*}
K_{2,3}(t,h)&\leq\E\left[\1_{\{\tau^{-h}_\varphi<T_1-t\}}L_{t,\tau^{-h}_\varphi}\frac{\left|\delta^{(2)}_{t,\theta}(\rho)- \delta^{(2)}_{t,\tau^{-h}_\varphi}(\rho)\right|}{y_2-c(t+\theta)}
                      \right]\\
&\quad+\E\left[\1_{\{\tau^{-h}_\varphi<T_1-t\}}L_{t,\tau^{-h}_\varphi}\left|\delta^{(2)}_{t,\tau^{-h}_\varphi}(\rho)\right|\left|\frac{1}{y_2-c(t+\theta)}- \frac{1}{y_2-c(t+\tau^{-h}_\varphi)-h}\right|\right]\\
&\leq\frac{1}{\bar c}\E\left[\1_{\{\tau^{-h}_\varphi<T_1-t\}}L_{t,\tau^{-h}_\varphi}\left|\delta^{(2)}_{t,\theta}(\rho)- \delta^{(2)}_{t,\tau^{-h}_\varphi}(\rho)\right|\right]\\
&\quad+(M_\delta/\bar c^{2})\E\left[\1_{\{\tau^{-h}_\varphi<T_1-t\}}L_{t,\tau^{-h}_\varphi} \left(\left|c(t+\theta)-c(t+\tau^{-h}_\varphi)\right|   +h\right)\right]\\
&\leq C_\cR\Big(\E\left[\1_{\{\tau^{-h}_\varphi<T_1-t\}}L_{t,\tau^{-h}_\varphi}\varepsilon_1(\sigma^\beta_{2h}\wedge T_1)\right]
                           +\E\left[\1_{\{\tau^{-h}_\varphi<T_1-t\}}L_{t,\tau^{-h}_\varphi}\left(c_\varphi (\sigma^\beta_{2h}\wedge T_1)+h\right)   
                         \right]\Big),
\end{align*}
where $c_\varphi$ is the Lipschitz constant of $s\mapsto c(s)$ on $[0,T_1]$ and we used 
$\theta-\tau^{-h}_\varphi\leq\sigma^\beta_{2h}\wedge T_1$ on $\{\tau^{-h}_\varphi<T_1-t\}$ along with Lemma \ref{lem:delta}.

Then, by Cauchy-Schwarz inequality we conclude
\begin{align}\label{eq:K23.2}
K_{2,3}(t,h)\le C_\cR\left\|L_{t,\tau^{-h}_\varphi\wedge(T_1-t)}\right\|_2\Big(\left\|\varepsilon_1(\sigma^\beta_{2h}\wedge T_1)\right\|_2+\left\|\sigma^\beta_{2h}\wedge T_1\right\|_2+h\Big).
\end{align}
Now, combining \eqref{eq:K21.2}, \eqref{eq:K22.2} and \eqref{eq:K23.2} and letting $h\to 0$ we find
\[
\limsup_{h\to 0}\sup_{0\le t\le T_2}\left|V_2(t)-h^{-1}\bar V^h_2(t)\right|\le C_\cR\delta^{1/4}.
\]
Since $\delta>0$ was arbitrary, $|V_2(t)-h^{-1}\bar V^h_2(t)|$ vanishes when $h\to 0$, uniformly in $t\in[0,T_2]$.
\medskip

{\bf Step 3}. Finally, we study the integral term in \eqref{eq:limV}. For a fixed $s\in(0,T_1-t]$, we have
\[
\bar V^h_s(t)=\E\left[\1_{\{J_s\leq h\}\cap\{\theta_h\geq s\}}L_{t,s}\delta^{(3)}_{t,s}(\rho)\right] \mbox{ and } 
V_s(t)=\E\left[\1_{\{\theta\geq s\}}\frac{L_{t,s}}{\rho_s}\delta^{(3)}_{t,s}(\rho)\right].
\]
Recall $B^h_s=\{J_s\leq h\}\cap\{\theta_h\geq s\}$. Since 
\[
\{J_s\leq h\}\cap\{\tau^{-h}_\varphi\geq s\}\subset B^h_s\subset \{J_s\leq h\}\cap\{\tau^{h}_\varphi\geq s\}
\quad\text{and}\quad
\{\tau^{-h}_\varphi\geq s\}\subset \{\theta\geq s\}\subset \{\tau^{h}_\varphi\geq s\},
\]
then
\[
\left|\1_{B^h_s}-\1_{\{J_s\leq h\}\cap\{\theta\geq s\}}\right|\leq \1_{\{J_s\leq h\}\cap\{\tau^{h}_\varphi\geq s\}}-
    \1_{\{J_s\leq h\}\cap\{\tau^{-h}_\varphi\geq s\}}= \1_{\{J_s\leq h\}\cap\{\tau^{h}_\varphi\geq s> \tau^{-h}_\varphi\}}
\]
and
\begin{align*}
&\left| \bar V^h_s(t)-\E\left[\1_{\{J_s\leq h\}\cap\{\theta\geq s\}}L_{t,s}\delta^{(3)}_{t,s}(\rho)\right]\right|\\
&\leq M_\delta\E\left[\1_{\{\tau^{h}_\varphi\geq s> \tau^{-h}_\varphi\}}
L_{t,s}\P\big(J_s\leq h\big|\cF^\rho_s\big)\right]\leq  M_\delta h\E\left[\frac{1}{\rho_s}L_{t,s}\1_{\{\tau^{-h}_\varphi< s\le\tau^{h}_\varphi\}}
     \right],
\end{align*}
where we used tower property, \eqref{eq:unifJ} and \eqref{eq:deltabound}. The use of tower property and \eqref{eq:unifJ} also yield
\begin{align*}
\E\left[\1_{\{J_s\leq h\}\cap\{\theta\geq s\}}L_{t,s}\delta^{(3)}_{t,s}(\rho)\right]&=
\E\left[\left(\frac{h}{\rho_s}\wedge 1\right)\1_{\{ \theta\geq s\}}L_{t,s}\delta^{(3)}_{t,s}(\rho)\right]\\
&=hV_s(t)-\E\left[\left(\frac{h}{\rho_s}- 1\right)^+\1_{\{ \theta\geq s\}}L_{t,s}\delta^{(3)}_{t,s}(\rho)\right].
\end{align*}
Using H\"older's inequality with $p^{-1}+q^{-1}+1/2=1$ and $p\in(1,3)$ we obtain
\begin{align}\label{eq:Vs}
\begin{aligned}
h^{-1}\big|\bar V^h_s(t)-hV_s(t)\big|&\leq  M_\delta\left(
      \E\left[\frac{L_{t,s}}{\rho_s}
                 \1_{\{\tau^{-h}_\varphi< s\le\tau^{h}_\varphi\}}
     \right]+\E\left[\frac{L_{t,s}}{\rho_s} \1_{\{\rho_s\le h\}}\right]\right)\\
     &\le M_\delta\left\|L_{t,s}\right\|_q\frac{1}{\sqrt{s}}\left\|\frac{1}{\rho_1}\right\|_p\left(\P\big(\tau^{-h}_\varphi< s\le\tau^{h}_\varphi\big)^{\frac12}+\P\big(\rho_s\le h\big)^{\frac12}\right)\\
     &\le \frac{C_\cR}{\sqrt{s}}\left(\P\big(\tau^{-h}_\varphi< s\le\tau^{h}_\varphi\big)^{\frac12}+\P\big(\rho_s\le h\big)^{\frac12}\right),
\end{aligned}
\end{align}
where in the final inequality we used Lemma \ref{lem:expbd} and \eqref{eq:rhoint}.
Then
\begin{align*}
\int_0^{T_1-t}h^{-1}\left|\bar V^h_s(t)-hV_s(t)\right|\ud s
   &\leq C_\cR\left(\int_0^{T_1}
          \sqrt{\P\left(\tau^{-h}_\varphi< s\le\tau^{h}_\varphi\right)}\frac{\ud s}{\sqrt{s}}+
              \int_0^{T_1}   \sqrt{ \P\left( \rho_s\le h\right) }    \frac{\ud s}{\sqrt{s}}   \right).    
\end{align*}
By dominated convergence, we have 
$\lim_{h\downarrow 0}\int_0^{T_1}   \sqrt{ \P(\rho_s\le h) }    \frac{\ud s}{\sqrt{s}}=0$. 
For the other integral, given $\delta\in[0,T_1)$ we have 
\begin{align*}
\P\left(\tau^{-h}_\varphi< s\le\tau^{h}_\varphi\right)\leq
  \P\left(\tau^{-h}_\varphi< s-\delta,\tau^{h}_\varphi\ge s\right)\!+\!\P\left(s\!-\!\delta\leq \tau^{-h}_\varphi<s\right)\leq \P\left(\sigma^\beta_{2h}\geq \delta\right)\!+\!C\frac{\sqrt{\delta}}{\sqrt{s-\delta}},
\end{align*}
for $s\in (\delta,T_1]$, by Lemma \ref{lem:3dBp1} and Lemma \ref{lem:pbdtauh} and with a constant $C>0$ independent of $(t,s,h,\delta)$.
Integrating with respect to $s$, we get
\begin{align*}
\int_0^{T_1}
          \sqrt{\P\left(\tau^{-h}_\varphi< s\le\tau^{h}_\varphi\right)}\frac{\ud s}{\sqrt{s}}&\leq
          \int_0^\delta\frac{\ud s}{\sqrt{s}}+
          \int_\delta^{T_1}\P\left(\sigma^\beta_{2h}\geq \delta\right)^\frac12\frac{\ud s}{\sqrt{s}}+C^\frac12\delta^{\frac14}\int_\delta^{T_1}\Big(\frac{1}{\sqrt{s-\delta}}\Big)^\frac12
          \frac{\ud s}{\sqrt{s}}\\
          &\leq 2\sqrt{\delta}+2\sqrt{T_1}\P\left(\sigma^\beta_{2h}\geq \delta\right)^\frac12+C^\frac12\delta^{\frac14}\int_0^{T_1-\delta}\frac{1}{s^{\frac14}}\frac{\ud s}{\sqrt{s+\delta}}.
\end{align*}
Hence, by dominated convergence,
\[
\limsup_{h\downarrow 0}
\left(\int_0^{T_1} \sqrt{\P\left(\tau^{-h}_\varphi< s\le\tau^{h}_\varphi\right)}\frac{\ud s}{\sqrt{s}}\right)\leq C\big(
 \sqrt{\delta}+\delta^{1/4}\big).
\]
Since $\delta$ can be made arbitrarily small, we have 
$\lim_{h\downarrow 0}\sup_{0\le t\le T_1}\int_0^{T_1-t}|h^{-1}\bar V^h_s(t)-V_s(t)|\ud s=0$.

This concludes the proof of \eqref{eq:limV} and therefore Theorem \ref{thm:expansion} holds.
\end{proof} 

\section{Proof of Theorem \ref{thm:main}}\label{sec:proofmain}

Theorem \ref{thm:expansion} in the previous section shows that 
\begin{align}\label{eq:expansion}
\dot w(t,c(t)+h)=h\Big(V_1(t)+V_2(t)+\int_0^{T_1-t}V_s(t)\ud s\Big)+o(h),\quad\text{for any $t\in[0,T_2]$},
\end{align}
where $o(h)$ stands for a remainder of a smaller order than $h$, uniformly for $t\in[0,T_2]$. We will use this fact to prove Theorem \ref{thm:main} in two steps. For the ease of presentation, first we state Lemma \ref{lem:lambdacont} concerning continuity of the expression in bracket in \eqref{eq:expansion}. Then we use the lemma to derive a formula for $\dot b(t)$ for $t\in[0,T_1)$ in Proposition \ref{prop:dotv}. The proof of Lemma \ref{lem:lambdacont} is given after the proof of the proposition.

\begin{lemma}\label{lem:lambdacont}
Under Assumption \ref{ass:main} the mapping
\begin{align}\label{eq:Lambda2}
\begin{aligned}
t\mapsto\Big(V_1(t)+V_2(t)+\int_0^{T_1-t}V_s(t)\ud s\Big),
\end{aligned}
\end{align}
is continuous on $[0,T_1)$, with $V_1(t)$, $V_2(t)$ and $V_s(t)$ as in Theorem \ref{thm:expansion}.
\end{lemma}

\begin{proposition}\label{prop:dotb}
Under Assumption \ref{ass:main} we have
\begin{align}\label{eq:bdot}
\begin{aligned}
\dot b(t)=\frac{\sigma\big(b(t)\big)}{2h\big(t,b(t)\big)}\Big(V_1(t)+V_2(t)+\int_0^{T_1-t}V_s(t)\ud s\Big),\quad t\in[0,T_1),
\end{aligned}
\end{align}
with $V_1(t)$, $V_2(t)$ and $V_s(t)$ as in Theorem \ref{thm:expansion}.
\end{proposition}
\begin{proof}
Recall the set $\cR=(0,T_1)\times(x_1,x_2)$ such that $x_1<b(t)<x_2$ for all $t\in[0,T_1]$. Let us fix $t_0\in[0,T_1)$. Since $u$ is continuously differentiable in $\overline\cR$, then $u_{xx}$ is bounded and continuous in the closure of $\cC\cap\cR$ thanks to \eqref{eq:h}. Moreover, using $u(t,x)=\dot u(t,x)=u_x(t,x)=0$ for $x\le b(t)$, and taking limits as $\cC\ni(t,x)\to (t_0,b(t_0))$ we obtain from \eqref{eq:h}
\begin{align}\label{eq:limuxx}
\begin{aligned}
\lim_{(t,x)\to(t_0,b(t_0))}u_{xx}\big(t,x\big)=-\frac{2h\big(t_0,b(t_0)\big)}{\sigma^2\big(b(t_0)\big)}\eqqcolon c_0>0,
\end{aligned}
\end{align}
where the strict inequality is due to \eqref{eq:h<0}. Then,  there is $\eps_0>0$ such that, with the notation $I_0\coloneqq \big((t_0-\eps_0)^+,(t_0+\eps_0)\wedge T_1\big)$, we have
\begin{align*}
u_{xx}(t,x)\ge \tfrac{c_0}{2}>0,\quad \text{for $0< x-b(t)\le \eps_0$ and $t\in\overline I_0$}.
\end{align*}
It follows by continuity of $u_x$ and boundedness of $u_{xx}$ in $\cR\cap \cC$ that, for any $(t,x)\in\cR$ such that $0< x-b(t)\le \eps_0$ and $t\in\overline I_0$, 
\begin{align*}
u_x(t,x)=\int_{b(t)}^x u_{xx}(t,z)\ud z\ge \tfrac{c_0}{2}\big(x-b(t)\big)>0.
\end{align*}
The latter implies that for $\delta>0$ sufficiently small there is a unique solution of the equation $u(t,x)=\delta$ for $x\in(b(t),b(t)+\eps_0)$. We denote it $b_\delta(t)$, so that $u(t,b_\delta(t))=\delta$. Clearly $b_\delta\in C^1(I_0)$ by the implicit function theorem with
\begin{align}\label{eq:dotbd}
\dot b_\delta(t)=-\frac{\dot u\big(t,b_\delta(t)\big)}{u_x\big(t,b_\delta(t)\big)},\quad t\in I_0.
\end{align}
Moreover, $b_\delta(t)\downarrow b(t)$ as $\delta\to 0$, for all $t\in\overline I_0$. Since $b_\delta$ and $b$ are continuous and the convergence is monotone, by Dini's theorem the convergence is indeed uniform on $\overline I_0$. 

Recalling the Lamperti transform and denoting $c_\delta(t)=f(b_\delta(t))$ it is convenient to rewrite the ODE above as\footnote{The ODE for $c_\delta(t)$ can also be derived directly from $w(t,c_\delta(t))=w(t,f(b_\delta(t)))=u(t,b_\delta(t))=\delta$.}
\begin{align}\label{eq:cdotdelta}
\dot c_\delta(t)=-\frac{\dot w\big(t,c_\delta(t)\big)}{w_y\big(t,c_\delta(t)\big)},\quad t\in I_0,
\end{align}
and recall the relation $\dot c_\delta(t)=\dot b_\delta(t)/\sigma(b_\delta(t))$. Setting $h_\delta(t)=c_\delta(t)-c(t)>0$ we have 
\[
\lim_{\delta \to 0}\sup_{t\in\overline I_0}h_\delta(t)=0.
\]
By Theorem \ref{thm:expansion}, it then follows 
\begin{align}\label{eq:ODE0}
\dot w\big(t,c_\delta(t)\big)=h_\delta(t)\Big(V_1(t)+V_2(t)+\int_0^{T_1-t}V_s(t)\ud s\Big)+\phi\big(h_\delta(t)\big),
\end{align}
where $\phi$ is a function such that
\begin{align}\label{eq:limphi}
\lim_{\delta\to 0}\Big(\sup_{t\in\overline I_0} \frac{\big|\phi\big(h_\delta(t)\big)\big|}{h_\delta(t)}\Big)=0.
\end{align}

For the denominator in \eqref{eq:cdotdelta} we can use the expression
\begin{align}\label{eq:ODE1}
w_y(t,c_\delta(t))=\int_{c(t)}^{c_\delta(t)}w_{yy}(t,z)\ud z=h_\delta(t)\int_{0}^{1}w_{yy}\big(t,c(t)+z h_\delta(t)\big)\ud z.
\end{align}
Since $u_{xx}$ is bounded and continuous in $\overline{\cC\cap\cR}$ and
\begin{align*}
w_y(t,y)&=u_x(t,f^{-1}(y))\sigma(f^{-1}(y)),\\
w_{yy}(t,y)&=u_{xx}(t,f^{-1}(y))\sigma^2(f^{-1}(y))+\sigma_x(f^{-1}(y))\sigma(f^{-1}(y))u_x(t,f^{-1}(y)),
\end{align*}
then it is clear that $w_{yy}$ is bounded on $\cR_f\cap\cC_f$, where $\cC_f=\{(t,y):y>c(t)\}$. Moreover, taking limits as $\cC_f\ni(t,y)\to (t_0,c(t_0))$ 
\begin{align}\label{eq:limwxx}
\begin{aligned}
\lim_{(t,y)\to(t_0,y_0)}w_{yy}\big(t,y\big)=-2 h\big(t_0,f^{-1}\big(c(t_0)\big)\big)=-2h\big(t_0,b(t_0)\big)\eqqcolon c'_0>0,
\end{aligned}
\end{align}
by \eqref{eq:limuxx}. Substituting \eqref{eq:ODE0} and \eqref{eq:ODE1} into \eqref{eq:cdotdelta} and integrating over an arbitrary interval $(t_1,t_2)\subset I_0$, yields
\begin{align}\label{eq:c-c}
\begin{aligned}
c_\delta(t_2)-c_\delta(t_1)=-\int_{t_1}^{t_2}\frac{\Lambda_{T_1}(t)+h^{-1}_\delta(t)\phi\big(h_\delta(t)\big)}{\int_{0}^{1}w_{yy}\big(t,c(t)+z h_\delta(t)\big)\ud z}\ud t,
\end{aligned}
\end{align}
where we set
\begin{align}\label{eq:Lambda}
\Lambda_{T_1}(t)\coloneqq V_1(t)+V_2(t)+\int_0^{T_1-t}V_s(t)\ud s.
\end{align}
Notice that with no loss of generality we can assume 
$\int_{0}^{1}w_{yy}\big(t,c(t)+z h_\delta(t)\big)\ud z\ge \tfrac{1}{2}c'_0>0$,
thanks to \eqref{eq:limwxx}.
Taking limits as $\delta \to 0$ in \eqref{eq:c-c}, invoking dominated convergence and using \eqref{eq:limphi} and \eqref{eq:limwxx} we obtain
\begin{align}\label{eq:c-c1}
\begin{aligned}
c(t_2)-c(t_1)=\int_{t_1}^{t_2}\frac{\Lambda_{T_1}(t)}{2h(t,b(t))}\ud t.
\end{aligned}
\end{align}
Thanks to Lemma \ref{lem:lambdacont}, $t\mapsto \Lambda_{T_1}(t)$ is continuous on $[0,T_1)$. Therefore, \eqref{eq:c-c1} and the equivalence $\dot c(t)=\dot b(t)/\sigma(b(t))$ imply
\[
\dot b(t)=\frac{\sigma(b(t))}{2h(t,b(t))}\Big(V_1(t)+V_2(t)+\int_0^{T_1-t}V_s(t)\ud s\Big),\quad t\in I_0.
\]
Since the interval $I_0$ is centred around $t_0$ and $t_0\in[0,T_1)$ is arbitrary, the proof is complete.
\end{proof}

It remains to prove Lemma \ref{lem:lambdacont}. For that, it is convenient to state separately an auxiliary lemma.
We recall that $\theta=\tau^0_{\varphi}\wedge(T_1-t)$, with $\varphi(s)=\varphi_t(s)\coloneqq c((t+s)\wedge T_1)-y_2$. Since $y_2$ is fixed, in order to keep track of the (double) dependence of $\theta$ on $t$ we relabel $\theta=\theta(t)$.
\begin{lemma}\label{lem:contheta}
Fix $t\in[0,T_1)$. For any sequence $(t_n)_{n\in\N}\subset[0,T_1]$ such that $t_n\to t$ as $n\to \infty$ we have 
$\lim_{n\to\infty}\theta(t_n)=\theta(t)$, $\P$-a.s.
\end{lemma}
\begin{proof}
We prove the result in two steps. First we show $\liminf_{n\to\infty}\theta(t_n)\ge \theta(t)$, $\P$-a.s. Fix $\omega\in\Omega$ and consider $\theta(t,\omega)=\tau^0_{\varphi_t}(\omega)\wedge(T_1-t)$. The result is trivial if $\theta(t,\omega)=0$, therefore we consider $\theta(t,\omega)>0$. Let $\delta>0$ be such that $\theta(t,\omega)>\delta$. Then there exists $\eps=\eps(\delta,\omega)>0$ such that 
\[
\rho_s(\omega)\le y_2-c((t+s)\wedge T_1)-\eps,\quad \forall s\in[0,\delta].
\]
Since $c(\cdot)$ is Lipschitz on $[0,T_1]$ then $c((t_n+s)\wedge T_1)-c((t+s)\wedge T_1)\le c_\varphi|t-t_n|$ and taking $|t-t_n|\le \eps/(2c_\varphi)$ yields
\begin{align*}
\rho_s(\omega)&\le y_2-c((t_n+s)\wedge T_1)+\big(c((t_n+s)\wedge T_1)-c((t+s)\wedge T_1)\big)-\eps\\
&\le y_2-c((t_n+s)\wedge T_1)+c_\varphi|t-t_n|-\eps\\
&\le y_2-c((t_n+s)\wedge T_1)-\eps/2,\quad \forall s\in[0,\delta].
\end{align*}
Thus $\theta(t_n,\omega)>\delta$ and $\liminf_{n\to\infty}\theta(t_n,\omega)\ge \delta$. We obtain $\liminf_{n\to\infty}\theta(t_n,\omega)\ge \theta(t,\omega)$ for $\P$-a.e.\ $\omega\in\Omega$, by arbitrariness of $\delta$ and $\omega$.

Next we want to show $\limsup_{n\to\infty}\theta(t_n)\le \theta(t)$, $\P$-a.s., so that combining the latter with the result above we conclude the proof of the lemma. On the event $\{\theta(t)=T_1-t\}$, since $\theta(t_n)\le T_1-t_n$ for every $n\in\N$, we have $\limsup_{n\to \infty}\theta(t_n)\le \theta(t)$. Let us now consider the event $\{\theta(t)<T_1-t\}$. Set $\tau'_{\varphi_t}\coloneqq\inf\{u\ge 0: \rho_u+\varphi(t+u)>0\}$. Clearly $\tau'_{\varphi_t}\ge \tau^0_{\varphi_t}$. Now we show that $\theta(t)=\tau^0_{\varphi_t}=\tau'_{\varphi_t}$ on $\{\theta(t)<T_1-t\}$. That can be deduced as follows: on the event $\{\theta(t)<T_1-t\}$, using $\rho_{\tau_{\varphi_t}^0}=y_2-c(t+\varphi_s^0)$ we have
\begin{align*}
&\rho_{\tau_{\varphi_t}^0+u}+c(t+\tau_{\varphi_t}^0+u)\\
&=\rho_{\tau_{\varphi_t}^0+u}-\rho_{\tau_{\varphi_t}^0}+y_2-c(t+\tau_{\varphi_t}^0)+c(t+\tau_{\varphi_t}^0+u)\\
&=\int_{\tau^0_{\varphi_t}}^{\tau^0_{\varphi_t}+u}\frac{1}{\rho_r}\ud r+\beta^0_{\tau^0_{\varphi_t}+u}-\beta^0_{\tau^0_{\varphi_t}}+y_2-c(t+\tau^0_{\varphi_t})+c(t+\tau_{\varphi_t}^0+u)\\
&>\beta^0_{\tau^0_{\varphi_t}+u}-\beta^0_{\tau^0_{\varphi_t}}+y_2-c_\varphi u.
\end{align*}
By the law of iterated logarithm, for any $\delta>0$ there exists $u\in(0,\delta)$ such that $\beta^0_{\tau^0_{\varphi_t}+u}-\beta^0_{\tau^0_{\varphi_t}}>c_\varphi u$. Then, $\rho_{\tau_{\varphi_t}^0+u}+c(t+\tau_{\varphi_t}^0+u)>y_2$ for any such $u\in(0,\delta)$, implying $\tau'_{\varphi_t}\le \tau^0_{\varphi_t}+\delta$. By arbitrariness of $\delta>0$ we conclude that $\tau'_{\varphi_t}= \tau^0_{\varphi_t}$ on $\{\theta(t)<T_1-t\}$.

Since $\theta(t)=\tau^0_{\varphi_t}=\tau'_{\varphi_t}$ on $\{\theta(t)<T_1-t\}$, for any $\delta>0$ there is $u\in(0,\delta)$ such that $\rho_{\tau^0_{\varphi_t}+u}+c(t+\tau^0_{\varphi_t}+u)>y_2$. Thus, by continuity of the boundary we can find $n\in\N$ sufficiently large so that $\rho_{\tau^0_{\varphi_t}+u}+c(t_n+\tau^0_{\varphi_t}+u)>y_2$. Hence $\theta(t_n)\le \tau^0_{\varphi_t}+u=\theta(t)+u$ and $\limsup_{n\to\infty}\theta(t_n)\le \theta(t)+\delta$. By arbitrariness of $\delta>0$ we conclude that $\limsup_{n\to\infty}\theta(t_n)\le \theta(t)$ also on $\{\theta(t)<T_1-t\}$.  
\end{proof}

\begin{proof}[{\bf Proof of Lemma \ref{lem:lambdacont}}]
Let us study continuity of $t\mapsto V_1(t)+V_2(t)$, first.
From Remark \ref{rem:V1+V2} we can rewrite
\[
V_1(t)+V_2(t)=\E\Big[\frac{L_{t,\theta(t)}}{\rho_{\theta(t)}}D_{t,\theta(t)}\dot w\big(t+\theta(t),c(t+\theta(t))+\rho_{\theta(t)}\big)\Big]\eqqcolon \E\Big[U\big(t,\theta(t)\big)\Big].
\]
Once again, we notice that with no loss of generality we can assume that $(t+r,c(t+r)+\rho_r)\in\cR_f$ for all $r\in[\![0,\theta(t)]\!]$ by Remark \ref{rem:P1}. Therefore, there is a constant $C_\cR>0$ independent of $t\in[0,T_1]$, such that
\begin{align}\label{eq:ubound}
\sup_{0\le t\le T_1}\big|D_{t,\theta(t)}\dot w\big(t+\theta(t),c(t+\theta(t))+\rho_{\theta(t)}\big)\big|\le C_\cR,\quad \P-a.s.
\end{align}
Now, fix $T_2<T_1$ and $t\in[0,T_2)$, and consider a sequence $(t_n)_{n\in\N}$ in $[0,T_2]$, with $\lim_{n\to \infty}t_n=t$.
In order to prove $\lim_{n\to \infty}\E[U(t_n,\theta(t_n))]=\E[U(t,\theta(t))]$, it suffices
to show that the sequence $(U(t_n,\theta(t_n)))_{n\in\N}$ is uniformly integrable (indeed bounded in $L^2$) and it converges in probability
to $U(t,\theta(t))$.
 
Due to \eqref{eq:ubound} and $\rho_{\theta(t_n)}\ge \bar c>0$ on $\{\theta(t_n)<T_1-t\}$ (cf.\ \eqref{eq:boundrho}), using H\"older's inequality with $p^{-1}+q^{-1}=1$ and $p\in(1,3/2)$ and letting $C_\cR>0$ change from line to line, we have
 \begin{align*}
\begin{aligned}
\E\Big[\big|U\big(t_n,\theta(t_n)\big)\big|^2\Big]&\le 
C_\cR\E\left[\big(L_{t,\theta(t_n)}\big)^2\left( \1_{\{\theta(t_n)<T_1-t_n\}}\frac{1}{(\rho_{\theta(t_n)})^2}+\1_{\{\theta(t_n)=T_1-t_n\}}\frac{1}{(\rho_{T_1-t_n})^2}\right)\right]\\
   & \le C_\cR\Big(\big\|L_{t,\theta(t_n)}\big\|^2_2+\big\|L_{t,\theta(t_n)}\big\|^2_{2q}\frac{1}{T_1-T_2}\Big\|\frac{1}{\rho_{1}}\Big\|^2_{2p}\Big)\le C_\cR.
\end{aligned}
\end{align*}
The final inequality follows by Lemma \ref{lem:expbd} and \eqref{eq:rhoint}.

For the convergence in probability, let us start by noticing that $\rho_{\theta(t)}>0$ with probability one (because $\rho_{\theta(t)}\geq \bar c$ on $\{\theta(t)<T_1-t\}$
and $\rho_{\theta(t)}=\rho_{T_1-t}$ on $\{\theta(t)\geq T_1-t\}$ and $t<T_1$) and therefore $\lim_{n\to\infty}(1/\rho_{\theta(t_n)})=1/\rho_{\theta(t)}$, $\P$-a.s., from Lemma \ref{lem:contheta}. Combining that with continuity of $\dot w$, of $c(\cdot)$ and of $t\mapsto R_{t,v}(\rho_v)$ we obtain, $\P$-a.s. 
\[
\lim_{n\to\infty}\frac{D_{t_n,\theta(t_n)}}{\rho_{\theta(t_n)}}\dot w\big(t_n+\theta(t_n),c(t_n+\theta(t_n))+\rho_{\theta(t_n)}\big)=\frac{D_{t,\theta(t)}}{\rho_{\theta(t)}}\dot w\big(t+\theta(t),c(t+\theta(t))+\rho_{\theta(t)}\big).
\]
Therefore, the convergence in probability of the sequence $\left(U\big(t_n,\theta(t_n)\big)\right)_{n\in\N}$
reduces to the convergence of $\left(L_{t_n,\theta(t_n)}\right)_{n\in\N}$ to $L_{t,\theta(t)}$. In particular, given 
\[
I_1(t_n,t)=\int_0^{\theta(t_n)}\gamma_{t_n,v}(\rho_v)\ud \rho_v-\int_0^{\theta(t)}\gamma_{t,v}(\rho_v)\ud \rho_v,
\]
and
\[
I_2(t_n,t)=\int_0^{\theta(t_n)}\left(\gamma_{t_n,v}(\rho_v)\right)^2\ud v
      -\int_0^{\theta(t)}\left(\gamma_{t,v}(\rho_v)\right)^2\ud v,
\]
it suffices to prove that $I_1(t_n,t)$ and $I_2(t_n,t)$ converge in probability to zero as $n\to\infty$. This requires some work because $\gamma_{t,v}(\rho_v)$ contains $\dot c(\,\cdot\,)$, which a priori is not continuous (cf.\ \eqref{eq:gammatv}).

We first study $I_1(t_n,t)$. We have
\begin{align*}
\left| I_1(t_n,t)\right|&\leq \left|\int_0^{\theta(t_n)\wedge \theta(t)}
        \left(\gamma_{t_n,v}(\rho_v)-\gamma_{t,v}(\rho_v)\right)\ud \rho_v\right|+
        \left|\int_{\theta(t_n)\wedge \theta(t)}^{\theta(t_n)}
       \gamma_{t_n,v}(\rho_v)\ud \rho_v\right|\\
        &\quad+
       \left|\int_{\theta(t_n)\wedge \theta(t)}^{\theta(t)}
       \gamma_{t,v}(\rho_v)\ud \rho_v\right|
\end{align*}
Using the SDE $\ud\rho_v=\ud\beta^0_v+\rho_v^{-1}\ud v$ (cf.\ \eqref{eq:SDErho}), we have $\left| I_1(t_n,t)\right|\leq
   I_{1,1}(t_n,t)+ I_{1,2}(t_n,t)$, with
\begin{align*}
I_{1,1}(t_n,t)&=
\left|\int_0^{\theta(t_n)\wedge \theta(t)}
        \left(\gamma_{t_n,v}(\rho_v)-\gamma_{t,v}(\rho_v)\right)\ud \beta^0_v\right|+
        \left|\int_{\theta(t_n)\wedge \theta(t)}^{\theta(t_n)}
       \gamma_{t_n,v}(\rho_v)\ud\beta^0_v\right|\\
        &\quad+
       \left|\int_{\theta(t_n)\wedge \theta(t)}^{\theta(t)}
       \gamma_{t,v}(\rho_v)\ud \beta^0_v\right|,
\end{align*}
and
\begin{align*}
I_{1,2}(t_n,t)&=
\int_0^{\theta(t_n)\wedge \theta(t)}
        \left|\gamma_{t_n,v}(\rho_v)-\gamma_{t,v}(\rho_v)\right|\frac{\ud v}{\rho_v}+
        \int_{\theta(t_n)\wedge \theta(t)}^{\theta(t_n)}
      \left| \gamma_{t_n,v}(\rho_v)\right|\frac{\ud v}{\rho_v}\\
        &\quad+
      \int_{\theta(t_n)\wedge \theta(t)}^{\theta(t)}
        \left|\gamma_{t,v}(\rho_v)\right|\frac{\ud v}{\rho_v}.
\end{align*}
Thanks to the boundedness of $\gamma_{t,v}(\rho_v)$, It\^o's isometry and Lemma \ref{lem:contheta}, we have
\begin{align*}
&\lim_{n\to\infty}\E\Big[\Big|\int_{\theta(t_n)\wedge \theta(t)}^{\theta(t_n)}
       \gamma_{t_n,v}(\rho_v)\ud\beta^0_v\Big|^2+
       \Big|\int_{\theta(t_n)\wedge \theta(t)}^{\theta(t)}
       \gamma_{t,v}(\rho_v)\ud \beta^0_v\Big|^2\Big]\le C_\cR\lim_{n\to\infty}\E\big[\big|\theta(t)-\theta(t_n)\big|\big] =0,
\end{align*}
for a constant $C_\cR>0$ independent of $t$ and $t_n$.
Dominated convergence and Lemma \ref{lem:contheta} yield, $\P$-a.s.
\begin{align*}
\lim_{n\to\infty}\Big(\int_{\theta(t_n)\wedge \theta(t)}^{\theta(t_n)}
      \left| \gamma_{t_n,v}(\rho_v)\right|\frac{\ud v}{\rho_v}+
      \int_{\theta(t_n)\wedge \theta(t)}^{\theta(t)}
        \left|\gamma_{t,v}(\rho_v)\right|\frac{\ud v}{\rho_v}\Big)\le \lim_{n\to\infty} C_\cR\int_{\theta(t_n)\wedge \theta(t)}^{\theta(t_n)\vee\theta(t)}\frac{\ud v}{\rho_v}=0.
\end{align*}
Thus
\begin{align}\label{eq:limI1}
\begin{aligned}
&\limsup_{n\to\infty}\E\big[\big|I_1(t_n,t)\big|\big]\\
&\le C \limsup_{n\to\infty}\E\Big[\int_0^{\theta(t_n)\wedge \theta(t)}
        \!\!\left|\gamma_{t_n,v}(\rho_v)\!-\!\gamma_{t,v}(\rho_v)\right|\frac{\ud v}{\rho_v}\!+\!\Big|\int_0^{\theta(t_n)\wedge \theta(t)}
        \!\!\left(\gamma_{t_n,v}(\rho_v)\!-\!\gamma_{t,v}(\rho_v)\right)\ud \beta^0_v\Big|\Big].
\end{aligned}
\end{align}

Now we use the decomposition $\gamma_{t,v}(\xi)=\gamma^0_{t,v}(\xi)-\gamma^1_{t+v}$,
      where
      \[
      \gamma^0_{t,v}(\xi)=\gamma(t+v,c(t+v)+\xi)\quad\mbox{and}\quad \gamma^1_s=\dot c(s).
      \]
      For the stochastic integral, It\^o's isometry yields
\begin{align}\label{eq:stochint}
\E\Big[\Big|\int_0^{\theta(t_n)\wedge \theta(t)}
        \left(\gamma_{t_n,v}(\rho_v)-\gamma_{t,v}(\rho_v)\right)\ud \beta^0_v\Big|^2\Big]=\E\Big[\int_0^{\theta(t_n)\wedge \theta(t)}\left(\gamma_{t_n,v}(\rho_v)-\gamma_{t,v}(\rho_v)\right)^2\ud v\Big].
\end{align}
      Using continuity and boundedness of $t\mapsto \gamma^0_{t, v}(\xi)$ and dominated convergence, we have
      \begin{align}\label{eq:stochint2}
      \lim_{n\to \infty}\E\left[\int_0^{\theta(t_n)\wedge\theta(t)}\left(\gamma^0_{t_n,v}(\rho_v)-\gamma^0_{t,v}(\rho_v)\right)^2
             \ud v\right]=0.
      \end{align}
In order to prove the analogous results with $\gamma^1$ instead of $\gamma^0$ we use $\theta(t)\leq T_1-t$ and set $\dot c_{T_1} (s)=\dot c(s\wedge T_1)$ to get
    \begin{align}\label{eq:stochint3}
        \begin{aligned}
\lim_{n\to\infty}    \int_0^{\theta(t_n)\wedge \theta(t)}\left(\gamma^1_{t_n+v}-\gamma^1_{t+v}\right)^2\ud v
              &\le \lim_{n\to\infty}\int_0^{T_1}\left| \dot c_{T_1}(t_n+v)-\dot c_{T_1}(t+v)\right|^2\ud v\\
              &=\lim_{n\to\infty}\big\|\dot c_{T_1}(t_n+\cdot)-\dot c_{T_1}(t+\cdot)\big\|^2_{L^2(0,T_1)}=0,
\end{aligned}
\end{align}
where the final equality holds because $\dot c$ is bounded on $[0,T_1]$ and translations are continuous in $L^2$. Combining \eqref{eq:stochint}, \eqref{eq:stochint2} and \eqref{eq:stochint3} we have
\[
\lim_{n\to\infty}\E\Big[\Big|\int_0^{\theta(t_n)\wedge \theta(t)}
        \left(\gamma_{t_n,v}(\rho_v)-\gamma_{t,v}(\rho_v)\right)\ud \beta^0_v\Big|^2\Big]=0.
\]     
      
For the first term on the right-hand side of \eqref{eq:limI1}, continuity of $t\mapsto \gamma^0_{t, v}(\xi)$ and dominated convergence give  
\[
\lim_{n\to\infty}\int_0^{\theta(t_n)\wedge \theta(t)}
\left|\gamma^0_{t_n,v}(\rho_v)-\gamma^0_{t,v}(\rho_v)\right|\frac{\ud v}{\rho_v}=0, \quad\P-\text{a.s.}
\]
Instead, H\"older inequality with $p^{-1}+q^{-1}=1$ and $q\in(1,2)$ yields 
\begin{align*}
\E\Big[\int_0^{\theta(t_n)\wedge \theta(t)}
        \left|\gamma^1_{t_n+v}-\gamma^1_{t+v}\right|\frac{\ud v}{\rho_v}\Big]&\leq
\big\|\dot c_{T_1}(t_n+\cdot)-\dot c_{T_1}(t+\cdot)\big\|_{L^p(0,T_1)} \E\Big[\int_0^{T_1}\frac{1}{\rho_v^q}\ud v\Big]^{1/q}\\
&\leq
\big\|\dot c_{T_1}(t_n+\cdot)-\dot c_{T_1}(t+\cdot)\big\|_{L^p(0,T_1)} \Big(\E\Big[\frac{1}{\rho_1^q}\Big]\int_0^{T_1}\frac{1}{v^{q/2}}\ud v\Big)^{1/q}\\
&\le C_\cR \big\|\dot c_{T_1}(t_n+\cdot)-\dot c_{T_1}(t+\cdot)\big\|_{L^p(0,T_1)}.
\end{align*}
Therefore $\lim_{n\to \infty} \E[\int_0^{\theta(t_n)\wedge \theta(t)}|\gamma^1_{t_n+v}-\gamma^1_{t+v}|\frac{\ud v}{\rho_v}]=0$, concluding the proof of $\lim_{n\to \infty}I_1(t_n,t)=0$ in probability (cf.\ \eqref{eq:limI1}).
        
For the convergence of $I_2(t_n,t)$, by boundedness of $\gamma_{t,v}(\rho_v)$ we obtain
\begin{align*}
\left| I_2(t_n,t)\right|&\leq
 \int_0^{\theta(t_n)\wedge \theta(t)}
        \left|\big(\gamma_{t_n,v}(\rho_v)\big)^2-\big(\gamma_{t,v}(\rho_v)\big)^2\right|\ud v
        +C_{\cR}\left| \theta(t_n)- \theta(t)\right|,
\end{align*}
for a constant $C_\cR>0$ independent of $n\in\N$.
From Lemma \ref{lem:contheta}, $\lim_{n\to \infty}\left| \theta(t_n)- \theta(t)\right|=0$, $\P$-a.s.\ Using $|a^2-b^2|\le 2(a\vee b)|a-b|$ and Cauchy-Schwarz inequality we also obtain
\[
\int_0^{\theta(t_n)\wedge \theta(t)}
        \left|\big(\gamma_{t_n,v}(\rho_v)\big)^2-\big(\gamma_{t,v}(\rho_v)\big)^2\right|\ud v\leq C_{\cR}\left(\int_0^{\theta(t_n)\wedge \theta(t)}
        \left(\gamma_{t_n,v}(\rho_v)-\gamma_{t,v}(\rho_v)\right)^2\ud v\right)^{1/2}.
\]
It follows from the proof of the convergence of $I_1(t_n,t)$ that the last integral vanishes in probability as $n\to \infty$.
 This completes the proof of the continuity of $t\mapsto V_1(t)+V_2(t)$.
 
 It remains to prove continuity of the integral term in \eqref{eq:Lambda2}.
    Again we fix $t\in[0,T_2)$ and consider a sequence $(t_n)$ in $[0,T_2]$,
 such that $\lim_{n\to \infty}t_n=t$. Notice that $|V_s(t)|\le C_\cR\|L_{t,s}\|_p\|\rho^{-1}_s\|_q\le C_\cR/\sqrt{s}$ for $p^{-1}+q^{-1}=1$
 (cf.\ Remark \ref{rem:welldef}). Therefore, in order to prove 
 \[
 \lim_{n\to \infty}\int_0^{T_1-t_n}V_s(t_n)\ud s=\int_0^{T_1-t}V_s(t)\ud s,
 \]
 it suffices to prove that, for almost all $s\in(0,T_1-t)$, $\lim_{n\to \infty}V_s(t_n)=V_s(t)$. Note that if $s<T_1-t$,
 then $s<T_1-t_n$ for $n$ large enough. Then the problem reduces to the convergence in probability of $L_{t_n,s\wedge\theta(t_n)}$
 to $L_{t,s\wedge \theta(t)}$, which can be obtained in the same way as we proved the convergence in probability
 of $L_{t_n,\theta(t_n)}$ to $L_{t,\theta(t)}$. We omit further details. 
 \end{proof}

We can finally provide a formal proof of Theorem \ref{thm:main}.
\begin{proof}[{\bf Proof of Theorem \ref{thm:main}}]
Proposition \ref{prop:dotb} gives us a formula for the derivative of the boundary, hence showing differentiability of $b$ at all points $t\in[0,T_1)$. Moreover, Lemma \ref{lem:lambdacont} gives us continuity in $t\in[0,T_1)$ of the expression on the right-hand side of \eqref{eq:bdot}. Thus, $\dot b\in C([0,T_1))$ as claimed.
\end{proof}

\section{From optimal stopping to the Stefan problem}\label{sec:OS-Stefan}

In this section we establish in the generality of our setup that the time-derivative $\dot v$ of the value function of the optimal stopping problem \eqref{eq:v} and the optimal boundary $b$ are a solution pair of the Stefan problem. In a setting with only Brownian motion, Van Moerbeke \cite[Secs.\ 2.4 and 2.5]{van1974optimal.b} made this observation under the assumption that $v_{xxxx}$ be continuous in $\cC$ up to and including the optimal boundary (hence also $\ddot v$ and $\dot v_{xx}$ inherit such continuity); later he elaborated more on those ideas also in \cite{van1976optimal}, building on analogous observations from Schatz \cite{schatz1969free}. It is worth noticing that Van Moerbeke works in a setup for which $\dot b$ remains bounded near the maturity (cf.\ p.\ 122 in the proof of Theorem 1 in \cite{van1974optimal.b}). That is not the case, for example, in the American put option problem.  Also, the assumption on the continuity of $v_{xxxx}$ up to and including the boundary of $\cC$ is difficult to verify in practice.

Let us first state precisely in what sense we intend a solution of the Stefan problem. Our definition is slightly different from those in, e.g., \cite[Ch.\ 8, Sec.\ 1]{friedman2008partial}  or \cite[Ch.\ VIII]{kinderlehrer2000introduction}, in the sense that we allow for more general structure of the linear parabolic operator and the condition at the terminal time is allowed to be a measure. We use $C^\infty_c(\cI)$ to indicate continuous functions with infinitely many continuous derivatives in all variables and compact support in $\cI$.
\begin{definition}
Let $\psi\in C([0,T)\times\cI)$, $\varphi,\eta,\nu\in C([0,T))$ and let $\Sigma$ be a Radon measure on $\cI$. A pair $(p,s)$ is a solution of the Stefan problem with data $(\psi,\varphi,\eta,\nu,\Sigma)$ if $s\in C^1([0,T))$ and, letting $\cO_s\coloneqq\{(t,x)\in[0,T)\times\cI: x>s(t)\}$, the function $p$ is such that $p\in C(\overline \cO_s)\cap C^{1,2}(\cO_s)$ and it solves 
\begin{align}\label{eq:SP}
\begin{aligned}
\dot p(t,x)+(\cL p)(t,x)-r(t,x)p(t,x)&=-\psi(t,x),\qquad\qquad\qquad\quad (t,x)\in\cO_s,\\
p(t,s(t))&=\varphi(t),\qquad\qquad\qquad\qquad\quad t\in[0,T),\\
\dot s(t)&= -\eta(t)p_x(t,s(t))+\nu(t),\:\:\quad t\in[0,T),\\
\lim_{t\to T}\int_{s(t)}^\infty p(t,z)\xi(z)\ud z&=\int_{[s(T),\infty)} \xi(z)\Sigma(\ud z),
\end{aligned}
\end{align} 
for any $\xi\in C^\infty_c(\cI)$. Existence of the derivative $p_x$ at points $(t,s(t))$ and continuity of the mapping $t\mapsto p_x(t,s(t))$, $t\in[0,T)$, are part of the definition of solution.
\end{definition}

The next theorem establishes a link between optimal stopping and Stefan problem. It assumes that the value function of the optimal stopping problem is continuously differentiable in the whole space, which holds in a broad class of optimal stopping problems as illustrated in \cite{de2020global}, even if $g$ is not smooth. In the statement below we take $g(T,\cdot)$ to be the difference of two convex functions so that 
$\Sigma(\ud z)\coloneqq(\cL g)(T,\ud z)-r(T,z)g(T,z)\ud z$ is a signed measure.
\begin{theorem}\label{thm:stefan}
Let Assumption \ref{ass:main} hold and further assume that:
\begin{itemize}
\item[(i)] $\sigma\in C^2(\cI)$, $\mu\in C^1([0,T]\times\cI)$ and $r\in C^{0,1}([0,T]\times\cI)$; 
\item[(ii)] $x\mapsto g(T,x)$ is the difference of two convex functions; 
\item[(iii)] $v\in C([0,T]\times\cI)\cap C^1([0,T)\times\cI)$; 
\item[(iv)] $t\mapsto b(t)$ is locally Lipschitz on $[0,T)$.
\end{itemize}
Then, the pair $(\dot v,b)$ is solution of the Stefan problem with data:
\begin{align*}
\psi(t,x)&=\dot \mu(t,x)v_x(t,x)-\dot r(t,x)v(t,x),\quad\varphi(t)=\dot g(t,b(t)),\\
\eta(t)&=-\frac{\sigma^2(b(t))}{2h(t,b(t))},\quad\nu(t)=-\frac{\sigma^2(b(t))\dot g_x(t,b(t))}{2h(t,b(t))},\\
\Sigma(\ud z)&=(\cL g)(T,\ud z)-r(T,z)g(T,z)\ud z,
\end{align*}
where we recall $h(t,b(t))=\dot g(t,b(t))+(\cL g)(t,b(t))-r(t,b(t)) g(t,b(t))$ from \eqref{eq:h<0}.
\end{theorem}
\begin{proof}
By continuity of the value function $v$ in $[0,T]\times\cI$ it follows by standard arguments (cf.\ \cite[Thm.\ 2.7.7]{karatzas1998methods} and \cite[Prop.\ 2.6]{jacka1991optimal} for the American put problem and \cite[Ch.\ III.7]{peskir2006optimal} for a general overview of the method) that
\begin{align}\label{eq:PDEv}
\begin{aligned}
\dot v(t,x)+(\cL v)(t,x)-r(t,x)v(t,x)=0,\qquad (t,x)\in\cC.
\end{aligned}
\end{align}
Using \eqref{eq:PDEv} we can easily derive the terminal condition for the Stefan problem by first noticing that for any $\xi\in C^\infty_c(\cI)$
\begin{align*}
\begin{aligned}
\int_{b(t)}^\infty \dot v(t,z)\xi(z)\ud z&=-\int_{b(t)}^\infty\big((\cL v)(t,z)-r(t,z) v(t,z)\big)\xi(t,z)\ud z\\
&=\kappa(t)-\int_{b(t)}^\infty v(t,z)\big((\cL^* \xi)(t,z)-r(t,z) \xi(z)\big)\ud z, 
\end{aligned}
\end{align*}
where 
\[
\kappa(t)\coloneqq \frac{\sigma^2\big(b(t)\big)}{2}g_x\big(t,b(t)\big)\xi\big(b(t)\big)-g\big(t,b(t)\big)\frac{\partial }{\partial x}\Big(\frac{\sigma^2(\cdot)}{2}\xi(\cdot)\Big)\big(b(t)\big)+\mu\big(t,b(t)\big)g\big(t,b(t)\big)\xi\big(b(t)\big)
\]
and $\cL^*$ is the adjoint operator of $\cL$. In deriving the expression for $\kappa(t)$ we used the boundary conditions $v(t,b(t))=g(t,b(t))$ and $v_x(t,b(t))=g_x(t,b(t))$. Letting $t\to T$, using $\lim_{t\to T}v(t,x)=g(T,x)$ and dominated convergence, and then undoing integration by parts, yields
\begin{align*}
\begin{aligned}
\lim_{t\to T}\int_{b(t)}^\infty \dot v(t,z)\xi(z)\ud z&=\kappa(T)-\int_{b(T)}^\infty g(T,z)\big((\cL^* \xi)(T,z)-r(T,z) \xi(z)\big)\ud z\\
&=\int_{[b(T),\infty)}\xi(z)\Sigma(\ud z),
\end{aligned}
\end{align*}
where we recall that $\Sigma(\ud z)$ is a signed measure.

Continuity $\dot v$ yields the boundary condition $\dot v(t,b(t))=\dot g(t,b(t))$ and therefore $\varphi(t)=\dot g(t,b(t))$. Plugging the boundary conditions $v(t,b(t))=g(t,b(t))$, $\dot v(t,b(t))=\dot g(t,b(t))$ and $v_x(t,b(t))=g_x(t,b(t))$ for $t\in[0,T)$ into \eqref{eq:PDEv} we obtain
\begin{align}\label{eq:vxx}
\begin{aligned}
v_{xx}\big(t,b(t)\big)
=g_{xx}\big(t,b(t)\big)-\frac{2}{\sigma^2(b(t))}h\big(t,b(t)\big),\quad t\in[0,T),
\end{aligned}
\end{align}
where we recall the function $h$ from \eqref{eq:h<0}. This allows us to obtain the third condition in the Stefan problem \eqref{eq:SP}, arguing as in the first part of the proof of Proposition \ref{prop:dotb}. In particular, from \eqref{eq:dotbd} and knowing that $b\in C^1([0,T))$ (Theorem \ref{thm:main}) we have
\begin{align}\label{eq:bdelta}
\begin{aligned}
\dot b(t)=\lim_{\delta \to 0}\dot b_\delta(t)=-\lim_{\delta\to 0}\frac{(\dot v-\dot g)(t,b_\delta(t))}{(v_x-g_x)(t,b_\delta(t))},
\end{aligned}
\end{align}
for any $t\in[0,T)$. Now we notice that \eqref{eq:expansion} guarantees the existence of the mixed derivative $(\dot v_x-\dot g_x)(t,b(t))$ and it also shows
\[
(\dot v-\dot g)(t,b_\delta(t))=(b_\delta(t)-b(t))(\dot v_x-\dot g_x)(t,b(t))+o(b_\delta(t)-b(t)).
\]
Similarly, \eqref{eq:vxx} implies
\[
(v_x-g_x)(t,b_\delta(t))=-(b_\delta(t)-b(t))\frac{2h(t,b(t))}{\sigma^2(b(t))}+o(b_\delta(t)-b(t)).
\]
Notice that $h(t,b(t))<0$ for $t\in[0,T)$. Substituting into the right-hand side of \eqref{eq:bdelta} yields  
\begin{align*}
\dot b(t)=\frac{\sigma^2(b(t))(\dot v_x-\dot g_x)(t,b(t))}{2 h(t,b(t))}=-\eta(t)\dot v_x(t,b(t))+\nu(t),
\end{align*}
with $\eta(t)=-\sigma^2(b(t))/2h(t,b(t))$ and $\nu(t)=-\sigma^2(b(t))(\dot g_x/2h)(t,b(t))$. It is clear from the assumed regularity of $g$ that both $\eta$ and $\nu$ are continuous on $[0,T)$. Instead, continuity of $t\mapsto \dot v_x(t,b(t))$ on $[0,T)$ is by Lemma \ref{lem:lambdacont} (or equivalently by continuity of $\dot b$).
  
It remains to verify that $\dot v$ solves the PDE in the first line of \eqref{eq:SP}. Thanks to the regularity of the coefficients, we can differentiate \eqref{eq:PDEv} with respect to time (cf.\ \cite[Thm.\ 3.5.11]{friedman2008partial}) and obtain: 
\begin{align*}
\begin{aligned}
\ddot v(t,x)+(\cL \dot v)(t,x)-r(t,x)\dot v(t,x)= -\dot \mu(t,x)v_x(t,x)+\dot r(t,x)v(t,x),\qquad (t,x)\in\cC.
\end{aligned}
\end{align*}
Thus, setting $\psi(t,x)=\dot \mu(t,x)v_x(t,x)-\dot r(t,x)v(t,x)$ and $\cO_b=\cC$ we have concluded our proof.
\end{proof}

\subsection{The Stefan problem for the American call and put options}
We can apply the theorem above to establish the precise link between the American call and American put problem and the Stefan problem. This appears to be missing from the literature on the American option problem. Starting with the American put problem, we take $g(x)=(K-x)^+$, $r(t,x)=r\ge 0$ and specify the dynamics 
\begin{align}\label{eq:XAm}
\ud X_t=(r-\delta)X_t\ud t+\sigma X_t\ud B_t,
\end{align}
with $\delta\ge 0$ and $\sigma>0$. Notice that because $b(t)>0$ for all $t\in[0,T]$, then the diffusion is non-degenerate locally around the boundary, thus Assumption \ref{ass:main}-(iv) holds. The state space of the problem is $[0,T]\times(0,\infty)$. Then $\Sigma(\ud x)=\1_{(0,K)}(x)(\delta x-rK)\ud x+\sigma^2K^2/2\cdot\delta_K(\ud x)$, where $\delta_K(\ud x)$ denotes the Dirac's delta in $K$. It is well-known that the optimal exercise boundary is a continuous non-decreasing function $t\mapsto b(t)$ and two cases may arise depending on the relative size of the dividend rate $\delta$ and the risk-free rate $r$. If $r\ge \delta$, then $b(T)=K$. Instead, for $r<\delta$ it holds $b(T)=(r/\delta)K$. This implies that the terminal conditions in the Stefan problems are different in the two cases: when $r\ge \delta$ we have formally $\dot v(T,\ud x)=\sigma^2K^2/2\cdot\delta_K(\ud x)$ for $x\in[K,\infty)$, i.e., for $\xi\in C^\infty_c((0,\infty))$
\begin{align}\label{eq:terminal1}
\lim_{t\to T}\int_{b(t)}^\infty \dot v(t,z)\xi(z)\ud z=\frac{\sigma^2 K^2}{2}\xi(K),
\end{align}
whereas for $r<\delta$ we have $\dot v(T,\ud x)=\1_{((r/\delta)K,K)}(x)(\delta x-rK)\ud x+\sigma^2 K^2/2\cdot\delta_K(\ud x)$ for $x\in[(r/\delta)K,\infty)$, i.e., 
\begin{align}\label{eq:terminal2}
\lim_{t\to T}\int_{b(t)}^\infty \dot v(t,z)\xi(z)\ud z=\int_{(r/\delta)K}^K (\delta z-rK)\xi(z)\ud z + \frac{\sigma^2 K^2}{2}\xi(K).
\end{align}

The function $h$ evaluated along the boundary reads
$h(t,b(t))=\delta b(t)-rK<0$, for $t\in[0,T)$, from which we deduce
\[
\eta(t)=-\frac{\sigma^2 b^2(t)}{2(\delta b(t)-rK)}\quad\text{and}\quad \nu(t)=0,\quad \text{for $t\in[0,T)$}. 
\]
The function $\psi$ is easily verified to be $\psi(t,x)=0$. It is also well-known that $v$ is continuous on $[0,T]\times(0,\infty)$ and it is shown in \cite[Examples 12 and 17]{de2020global} that both $\dot v$ and $v_x$ are continuous on $[0,T)\times(0,\infty)$ (cf.\ also \cite[Thm.\ 3.5 and Lemma 5.4]{de2022stopping}). Therefore, $\varphi(t)=\dot v(t,b(t))=0$. Finally, the results of Section \ref{sec:lip} guarantee that $b$ is locally Lipschitz on $[0,T)$. Thus, using Theorem \ref{thm:stefan} we obtain that the pair $(\dot v,b)$ solves the following Stefan problem: 
\begin{align*}
\begin{aligned}
\ddot v(t,x)+\frac{\sigma^2x^2}{2} \dot v_{xx}(t,x)+(r-\delta)x \dot v_{x}(t,x)-r\dot v(t,x)&=0,\qquad\qquad\qquad\qquad\quad (t,x)\in\cC,\\
\dot v(t,b(t))&=0,\qquad\qquad\qquad\qquad\quad\: t\in[0,T),\\
\dot b(t)&= \frac{\sigma^2 b^2(t)}{2}\frac{\dot v_x(t,b(t))}{\delta b(t)-rK},\quad\:\: t\in[0,T),
\end{aligned}
\end{align*} 
with $\cO_b=\cC=\{(t,x)\in[0,T)\times(0,\infty) : x> b(t)\}$ and with terminal condition given by either \eqref{eq:terminal1} or \eqref{eq:terminal2}, depending on whether $r\ge \delta $ or $r<\delta$.

The arguments for the case of the American call option are very similar and we only sketch them here for brevity. The stopping payoff reads $g(x)=(x-K)^+$ and the discount rate is $r\ge 0$. The underlying dynamics is the same as in \eqref{eq:XAm} but with $\delta>0$, in order to avoid the trivial situation when $v(t,x)>g(x)$ for all $(t,x)\in[0,T)\times(0,\infty)$. In this setup the continuation set is bounded from above by the optimal boundary $t\mapsto b(t)$ (cf.\ Remark \ref{rem:loc}), which is continuous and non-increasing with $b(T)=K$ if $r\le\delta$ and $b(T)=(r/\delta)K$ if $r>\delta$. The value function is continuous on $[0,T]\times (0,\infty)$ and continuously differentiable on $[0,T)\times (0,\infty)$.
Since
$\Sigma(\ud x)=(rK-\delta x)\1_{(K,\infty)}(x)\ud x+\sigma^2 K^2/2\cdot\delta_K(\ud x)$, then for $\xi\in C^\infty((0,\infty))$ the terminal conditions 
\begin{align*}
\lim_{t\to T}\int^{b(t)}_0 \dot v(t,z)\xi(z)\ud z&=\frac{\sigma^2 K^2}{2}\xi(K),\\
\lim_{t\to T}\int^{b(t)}_0 \dot v(t,z)\xi(z)\ud z&=\int^{(r/\delta)K}_0 (rK-\delta z)\xi(z)\ud z + \frac{\sigma^2 K^2}{2}\xi(K).
\end{align*}
hold for $r\le\delta$ and $r>\delta$, respectively. Also in this case $\psi(t,x)=0$ and along the boundary we have $\dot v(t,b(t))=0=\varphi(t)$ and  $h(t,b(t))=rK-\delta b(t)<0$ for $t\in[0,T)$. We deduce $\nu(t)=0$ and $\eta(t)=-\sigma^2 b^2(t)[2(rK-\delta b(t))]^{-1}$. Then $(\dot v,b)$ solves the Stefan problem with  data $(\psi,\varphi,\eta,\nu,\Sigma)$ specified above and $\cO_b=\{(t,x)\in[0,T)\times(0,\infty) : 0< x < b(t)\}$. 

\medskip
\noindent{\bf Acknowledgment}: T.\ De Angelis received partial financial support from EU -- Next Generation EU -- PRIN2022 (2022BEMMLZ) and PRIN-PNRR2022 (P20224TM7Z). 

\bibliography{bibliography}{}
\bibliographystyle{abbrv}

\end{document}